\documentclass[10pt,letterpaper]{amsart}
\usepackage[ansinew]{inputenc}

\usepackage{url}

\newcommand{\Dual}{\mathsf{C}}
\newcommand{\Jacobian}{\nabla}

\newcommand{\sign}{\operatorname{sign}}
\newcommand{\adj}{\operatorname{adj}}

\newcommand{\sym}{\vee}
\newcommand{\Perm}{\operatorname{Perm}}
\newcommand{\Scherkset}{\calA}
\newcommand{\Scherkcof}{A}
\newcommand{\SphereArea}{A}
\newcommand{\TV}{\mathrm{TV}}

\newcommand{\contract}{:} 

\newcommand{\Orlicz}{\calA}
\newcommand{\DoubleOrlicz}{\calB}

\newcommand{\Cont}{\calC}
\newcommand{\Lebesgue}{\calL}
\newcommand{\Sobolev}{\calW}
\newcommand{\BV}{\operatorname{BV}}

\newcommand{\IS}{I}
\newcommand{\IT}{{\hat{I}}}
\newcommand{\OmegaS}{\Omega}
\newcommand{\OmegaT}{{\hat\Omega}}
\newcommand{\VEC}{V}
\newcommand{\vectest}{\Psi}
\newcommand{\test}{\psi}
\newcommand{\LIN}{\frakL}

\usepackage{amsfonts}
\usepackage{amsmath}
\usepackage{amssymb}
\usepackage{amsthm}
\usepackage{stmaryrd}
\usepackage{amscd}
\usepackage{mathrsfs}
\usepackage{textcomp}
\usepackage{mathtext}
\usepackage{yfonts}

\usepackage{enumerate}

\usepackage{tikz}

\newtheorem{theorem}{Theorem}[section]
\newtheorem{lemma}[theorem]{Lemma}
\newtheorem{corollary}[theorem]{Corollary}
\theoremstyle{definition}

\newtheorem{example}[theorem]{Example}

\newtheorem{proposition}[theorem]{Proposition}
\theoremstyle{remark}
\newtheorem{remark}[theorem]{Remark}

\numberwithin{equation}{section}

\usepackage{nicefrac}

\newcommand{\indexset}[2]{{\{#1\mathbin{:}#2\}}}

\makeatletter
\newcommand\suchthat{\@ifstar
  {\mathrel{}\middle\vert\mathrel{}}
  {\mid}}
\makeatother

\DeclareSymbolFont{bbold}{U}{bbold}{m}{n}
\DeclareSymbolFontAlphabet{\mathbbold}{bbold}

\DeclareMathOperator*{\esssup}{esssup}

\newcommand{\eps}{{\epsilon}}

\newcommand{\dif}{{\mathrm d}}

\newcommand{\divergence}{\operatorname{div}}

\newcommand{\equivalent}{ \Longleftrightarrow }

\newcommand{\dist}{\operatorname{dist}}

\newcommand{\supp}{\operatorname{supp}}

\newcommand{\bbN}{{\mathbb N}}

\newcommand{\bbR}{{\mathbb R}}

\newcommand{\bfD}{{\mathbf D}}

\newcommand{\bfF}{{\mathbf F}}

\newcommand{\calA}{{\mathcal A}}
\newcommand{\calB}{{\mathcal B}}
\newcommand{\calC}{{\mathcal C}}
\newcommand{\calD}{{\mathcal D}}

\newcommand{\calL}{{\mathcal L}}

\newcommand{\calP}{{\mathcal P}}

\newcommand{\calW}{{\mathcal W}}

\newcommand{\frakL}{{\mathfrak L}}

\newcommand{\frakP}{{\mathfrak P}}

\begin{document}

\title[Chain rule for tensor fields]
{Higher-order chain rules for tensor fields, generalized Bell polynomials, and estimates in Orlicz-Sobolev-Slobodeckij and total variation spaces.}

\author[Martin W.\ Licht]{Martin W.\ Licht}
\address{\'Ecole Polytechnique F\'ed\'erale de Lausanne (EPFL), 1015 Lausanne, Switzerland}
\email{martin.licht@epfl.ch}

\thanks{This material is based upon work supported by the National Science Foundation 
under Grant No.\ DMS-1439786 while the third author was in residence at the Institute for Computational and 
Experimental Research in Mathematics in Providence, RI, during the ``Advances in Computational Relativity'' program.}


\date{}

\begin{abstract}
    We describe higher-order chain rules for multivariate functions and tensor fields. 
    We estimate Sobolev-Slobodeckij norms, Musielak-Orlicz norms, and the total variation seminorms of the higher derivatives of tensor fields after a change of variables and determine sufficient regularity conditions for the coordinate change. 
    We also introduce a novel higher-order chain rule for composition chains of multivariate functions
    that is described via nested set partitions and generalized Bell polynomials;
    it is a natural extension of the Fa\`{a}~di~Bruno formula.
    Our discussion uses the coordinate-free language of tensor calculus and includes Fr\'echet-differentiable mappings between Banach spaces. 
\end{abstract}

\maketitle

\section{Introduction}\label{sec:intro}

The higher-order chain rule for the composition of functions in one real variable is generally known as \emph{Fa\`{a}~di~Bruno formula}.
Less well-known are its generalizations to multivariate functions or to multiple compositions. 
In this contribution, we combine these two possible generalizations: 
we formulate the higher-order chain rule in several variables and for compositions of several such functions.
This seems to not have been discussed in the literature before. 
The key observation is that,
just like the basic Fa\`{a}~di~Bruno formula can be interpreted in terms of set partitions, 
our generalization can be interpreted in terms of hierarchies of set partitions. 
We express the derivatives in terms of a new recursive generalization of Bell polynomials.
Formulating our results in the coordinate-free language of tensor analysis is crucial
and also elucidates the basic structure of the theorem, even in the univariate case. 

How to estimate function space norms of tensor fields and their higher derivatives after a coordinate transformation might be known in very broad strokes, 
but no higher-order chain rules and estimates are found in the literature, and some interesting subtleties are easily overlooked. 
We estimate the Sobolev-Slobodeckij norms of tensor fields with rough coefficients after coordinate transformations.
While the ubiquity of such estimates in analysis is out of question, we provide explicit formulas for the higher derivatives of pullbacks and study what regularity is necessary.
For example, vector fields and differential forms with Sobolev-Slobodeckij regularity $\Sobolev^{1+\theta,p}$ are preserved under pullbacks along transformations with H\"older regularity $\Cont^{2,\sigma}$ if $\sigma > \theta$. 

Subsequently, we generalize these estimates to tensor fields in Musielak-Orlicz spaces. 
Such spaces emerge in the study of nonlinear partial differential equations; 
the most famous special cases are the Lebesgue and Sobolev-Slobodeckij spaces with variable exponents,
double phase spaces, and logarithmic Sobolev spaces. 
More generally, we study the transformation of tensor fields in \emph{fractional Musielak-Orlicz} spaces
provided that the Musielak-Orlicz integrands are sufficiently well-behaved. 
We give a brief introduction into these spaces. 
The notion of \emph{pullback} of Musielak-Orlicz integrand emerges naturally in this setting.

Lastly, we study a generalized class of bounded variation spaces and estimate their total variation seminorm under a specific class of coordinate transformations. 
Specifically, the coordinate transformations require higher regularity than for estimates in Orlicz-Sobolev spaces:
the highest derivative must have an essentially bounded divergence. 
This emerges in the study of diffeomorphisms that minimize energy functionals. 
In that discussion it essential that we characterize bounded variation spaces via integration by parts against \emph{flat vector fields},
meaning essentially bounded vector fields with essentially bounded divergence. 
\\

We review the history of the higher-order chain rule and contextualize this discussion.
Given two sufficiently differentiable mappings $f$ and $\phi$ in one variable,
the $m$-th derivative of their composition is:
\begin{align}\label{math:faadibruno}
    (f \circ \phi)^{(m)} (x)
    =
    \sum_{ \substack{ (b_1, b_2, \ldots, b_m) \in \bbN_{0}^{m} \\ b_1 + 2 b_2 + \cdots + m b_m = m } }
    \frac{m!}{ b_1! b_2! \cdots b_m! }
f^{(\sum_{j=1}^{m} b_j)}
    ( \phi(x) )
    \prod_{j=1}^{m}
    \left( \frac{ \phi^{(j)}(x) }{ j!} \right)^{b_j}
    .
\end{align}
This chain rule of higher order is a direct extension of a classical result of real analysis,
just analogous to how the Leibniz rule generalizes the elementary product rule. 
Indeed, a possible motivation for representing the $m$-th derivative in the form~\eqref{math:faadibruno}, the \emph{Fa\`{a}~di~Bruno formula},
stems from comparison with the generalized product rule: 
if one develops the left-hand side of the formula for some low values of $m$ and combines like summands, 
one quickly starts wondering about the pattern in the occurring products of derivatives and their numerical coefficients.
The Fa\`{a}~di~Bruno formula, as stated above, directly answers this question.

We show that this formula specializes a more general representation of the $m$-th derivative to the univariate case.
Whereas formula~\eqref{math:faadibruno} does not generalize immediately to compositions of multivariate vector-valued functions,
we present a precursor formula, written in terms of set partitions, that does. 
The Fa\`{a}~di~Bruno formula naturally emerges from our discussion, as do expressions for the higher derivatives of composition chains of mappings between vector spaces.

The centuries-old publication history of the Fa\`{a}~di~Bruno formula~\eqref{math:faadibruno} has an interesting history. 
We give a very brief overview of the oft-repeated discovery of this chain rule,
based on the extensive discussion and sources collected in the accounts by Johnson~\cite{johnson2002curious} and Craik~\cite{craik2005prehistory}.
The first known written reference to the higher-order chain rule is found in Arbogast's 1800 contribution~\cite{arbogast1800calcul},
who discusses the basic idea and writes down examples, without stating Equation~\eqref{math:faadibruno} itself. 
The second edition of Lacroix's Treatise (1810,~\cite{lacroix1797traite}) states the formula together with the coefficients.
The higher-order chain rule~\eqref{math:faadibruno} was also stated explicitly by Scherk (1823,~\cite{scherk1823evolvenda}) 
and Tiburce Abadie (1850,~\cite{abadie1850differentiation}).
It appears in Fa\`{a}~di~Bruno's 1855 article~\cite{faa1855sullo} together with an original determinant formula,
and it gained wider circulation with his book~\cite{di1876theorie}. 
Since then, the Fa\`{a}~di~Bruno formula has received particular attention in combinatorics (see, e.g.,~\cite{roman1980formula}) due to its connection with Bell polynomials. 
However, the chain rule of higher order is sometimes described as a rather inaccessible result of classical analysis (see, e.g., Flanders~\cite{flanders2001ford}).

The chain rule of higher order in several variables seems to have been addressed much later. The work of Most (1871,~\cite{most1871ueber}) addresses compositions $f(x(t))$ where $x(t)$ is a univariate vector field. 
Cartan's book~\cite{cartan1967calcul} only addresses second and third derivatives of composed multivariate functions.
Without claiming to provide a complete literature review,
we also mention the contributions of Gzyl~\cite{gzyl1986multidimensional}, of Constantine and Savits~\cite{constantine1996multivariate}, and of Encinas and Masqu\'e~\cite{encinas2003short}. 
The aforementioned works state the multivariate higher-order chain rule in coordinate form. 
We instead use the coordinate-free language of tensor calculus and discuss successive directional derivatives. 
Thus, we present several variations of the chain rule while explicitly preserving underlying geometric structures
that would be less clear in coordinate notation.

The language of tensor calculus allows for a succinct formulation of the higher-order chain rule for compositions of \emph{multiple} multivariate vector-valued functions. 
Whereas the classical Fa\`{a}~di~Bruno formula is stated in terms of set partitions,
the generalization to composition chains involves \emph{nested} set partitions. 
To the author's best knowledge, the present manuscript's interpretation via nested partitions is novel,
and the general form for multivariate functions has not yet been discussed before in the literature either. 
We refer to Natalini and Ricci~\cite{natalini2004extension}, who apparently were the first to discuss higher-order chain formulae for composition chains of univariate functions,
and to Bernardini, Natalini, and Ricci~\cite{bernardini2005multidimensional},
who develop a higher-order chain rule for the case of a single multivariate function in several composition chains of univariate functions.
We address more general types of compositions here. Moreover, we address Fr\'echet-differentiable mappings between Banach spaces. 
\\

The main application of our generalized chain rule is deriving explicit formulae for estimating the pullback of covariant tensor fields along coordinate changes, measured in function space norms.
\emph{Sobolev spaces}~\cite{sobolev2008some} are among the most important achievements of functional analysis 
and appear in the solution theory of partial differential equations and their numerical analysis. 
\emph{Sobolev-Slobodeckij spaces}~\cite{slobodeckij1958generalized} have seen increased interest over the recent decades because they formalize the notion of Sobolev spaces with ``intermediate'' non-integer smoothness. 
Vector and tensor fields with coefficients in Sobolev-Slobodeckij spaces
are ubiquitous in the mathematical theories of electromagnetism, elasticity and general relativity.
The multivariate chain rule allows estimating the Sobolev-Slobodeckij norms of scalar and tensor fields after a coordinate change.
A qualitative difference between tensor fields and scalar fields is that the geometrically conforming coordinate transformation of tensor fields is the pullback,
which involves the Jacobian of the coordinate change.
Thus, while scalar fields with $m$ weak derivatives are preserved under a coordinate transformation with $m$ derivatives,
general tensor fields with $m$ weak derivatives are preserved only under a coordinate transformation with $m+1$ derivatives.
In addition, we show in detail how Sobolev-Slobodeckij spaces are only preserved under coordinate transformations of strictly higher smoothness. 
To the author's impression, these facts are often overlooked. 

We extend the discussion to \emph{Musielak-Orlicz spaces}~\cite{chlebicka2021partial,Harjulehto2019}. These Banach spaces naturally emerge in the discussion of nonlinear partial differential equations. They include Lebesgue spaces with variable exponents which appear, e.g., when modeling of electrorheological fluids~\cite{diening2011lebesgue}, or double phase spaces~\cite{colombo2015regularity}, and we also address their fractional counterparts~\cite{alberico2021fractional}.
For example, such spaces emerge in the discussion of the fractional Laplacian with variable exponent~\cite{kaufmann2017fractional}.

Finally, we address the pullback of tensor fields in higher-order bounded variation spaces~\cite{ambrosio2000functions}.
The $m$-th distributional derivative of a function may not exist as a locally integrable function,
but it may exist as a finite Radon measure. 
We discuss the preservation of that \emph{bounded variation} property under bi-Lipschitz mappings 
and prove higher-order chain rules for functions whose $m$-th weak derivative has bounded variation. 
In comparison to the estimate of Orlicz-Sobolev norms, estimating the total variation assumes that the higher derivatives of the transformation have essentially bounded tensor divergence. 
This is an extra smoothness qualitatively different from extra ``fractional'' smoothness. 
Transformations with such regularity conditions arise, e.g., from diffeomorphisms that minimize energy functionals.
\\

The remainder of this article is structured as follows. 
We summarize definitions and notation in Section~\ref{sec:notation}. 
We discuss the higher-order chain rules in Section~\ref{sec:chainrule}, which includes a proof of the original Fa\`{a}~di~Bruno formula. 
We discuss the higher-order chain rules for multicomposite functions in Section~\ref{sec:multicomposite}. 
We derive pointwise estimates for the pullback of tensor fields in Section~\ref{sec:pullbacks}. 
We develop estimates for the pullback of tensor fields in Sobolev-Slobodeckij spaces in Section~\ref{sec:sobolev}, 
introduce and discuss Musielak-Orlicz spaces in Section~\ref{sec:orlicz},
and finally consider generalized bounded variation spaces in Section~\ref{sec:bv}.

\section{Definitions and Notation}\label{sec:notation}

In this section, we review background material and introduce notation: 
this considers multilinear algebra, derivatives, and set partitions. 

The main applications of higher-order chain rules are mappings between finite-dimensional spaces. 
Nevertheless, we also address the more general infinite-dimensional setting, 
where we consider Frech\'et-differentiable mappings over Banach spaces. 
We take care to present the material in a manner 
that allows handling the infinite-dimensional setting as accessibly as the finite-dimensional setting.

\subsection{Multilinear Algebra}

We review notions of multilinear algebra and tensor analysis. 
For any normed vector space $X$, we let $X^{\ast}$ denote its dual space,
that is, the space of bounded linear functionals over $X$. 
We write $\langle x, x^\ast \rangle \equiv \langle x^\ast, x \rangle$ for the bilinear pairing of $x \in X$ and $x^\ast \in X^\ast$. 
We let $\LIN(X,Y)$ denote the space of bounded linear mappings from some normed vector space $X$ into another normed vector space $Y$.
More generally, we write $\LIN^{m}(X,Y)$ for the space of bounded $m$-linear mappings 
that take $m$ variables in $X$ and map into the space $Y$. 
The \emph{spectral norm} on the vector space $\LIN^{m}(X,Y)$ is 
\begin{align*}
    \| u \|
    :=
    \sup_{ x_1, x_2, \ldots, x_m \in X \setminus \{0\} } 
    \dfrac{
        \| u(x_1,x_2,\ldots,x_m) \|_{Y}
    }{ 
        \|x_1\|_{X} \|x_2\|_{X} \cdots \|x_m\|_{X}
    },
    \quad 
    u \in \LIN^{m}(X,Y)
    .
\end{align*}
In the case of a $1$-linear mapping, this is just the usual \emph{operator norm}. 

Multilinear mappings naturally lead to the discussion of tensor products. 
When $X_1$ and $X_2$ are normed vector spaces, then $X_1 \otimes X_2$ denotes their tensor product space.
As usual, $x_1 \otimes x_2 \in X_1 \otimes X_2$ denotes the simple tensor product of any $x_1 \in X_1$ and $x_2 \in X_2$.
We abbreviate $\otimes^m X$ for the $m$-fold tensor product of the space $X$ with itself;
by convention, $\otimes^0 X = \bbR$ is just the scalars.
We equip the tensor product with the \emph{projective norm}
\begin{align*}
    \| \mu \|
    :=
    \inf
    \left\{
        \sum_{i} \|x_{1,i}\|_{X_1} \|x_{2,i}\|_{X_2} \suchthat* \mu = \sum_{i} x_{1,i} \otimes x_{2,i}
    \right\},
    \quad 
    \mu \in X_{1} \otimes X_{2}.
\end{align*}
The projective norm satisfies the \emph{cross norm property}:
$\| x_1 \otimes x_2 \| = \|x_1\|_{X_1} \|x_2\|_{X_2}$ for any $x_1 \in X_1$ and $x_2 \in X_2$. 
The definition of the projective norm naturally extends to the tensor products of more than two normed spaces. 

We use the projective norm whenever we address the norm of a tensor product, unless otherwise noted. 
However, apart from the finite-dimensional case, the tensor product with the projective norm is not a complete vector space.
In order to keep the notation simple, \emph{we tacitly denote the closure of the algebraic tensor product under the projective norm by $X_{1} \otimes X_{2}$ as well}.

We review the relation between tensor products and multilinear mappings. We write 
\begin{align*}
    \LIN^{m}_{n}(X,Y)
    := 
    \LIN\left( \otimes^m X, \otimes^n Y \right)
    .
\end{align*}
The members of $\LIN^{m}_{n}(X,Y)$ are naturally interpreted as multilinear mappings. 
For example, $\LIN^{m}_{1}(X,Y) = \LIN^{m}(X,Y)$ is a multilinear form that maps $m$ members of $X$ into the space $Y$,
and $\LIN^{m}_{0}(X,Y)$ is a scalar-valued multilinear form in $m$ members of $X$.
Notice also that $\LIN^{1}_{0}(X,Y) = X^{\ast}$ and $\LIN^{0}_{1}(X,Y) = Y$, and that $\LIN^{0}_{0}(X,Y)$ is the scalars. 
We also note that the operator norm on $\LIN^{m}_{n}(X,Y)$,
when interpreting its members as linear mappings from $\otimes^m X$ into $\otimes^n Y$,
is the same as the spectral norm when interpreting its members as $m$-linear mappings into $\otimes^n Y$.

We adopt some terminology that is common in tensor analysis and differential geometry. 
We say that every member of $\LIN^{m}_{n}(X,Y)$ has $m$ covariant indices and $n$ contravariant indices. 
Indeed, this terminology of covariant and contravariant indices is well-established for spaces of the form 
\begin{align*}
    \underbrace{X^{\ast} \otimes \cdots \otimes X^{\ast}}_{m\text{-times}}
    \otimes 
    \underbrace{Y \otimes \cdots \otimes Y}_{n\text{-times}}
    .
\end{align*}
This space naturally embeds into $\LIN^{m}_{n}(X,Y)$ (and coincides with it in the finite-dimensional setting), and therefore we extend the terminology to multilinear forms. 
\\

We may consider the tensor product of two linear mappings. 
Whenever $\mu \in \LIN^{m}_{n}(X,Y)$ and $\nu \in \LIN^{r}_{s}(X,Y)$,
we may identify their tensor product $\mu \otimes \nu$ 
naturally as a member of $\LIN^{m+r}_{n+s}(X,Y)$. 
This is a minor abuse of notation but very practical.
The projective norm on the tensor $\mu \otimes \nu$ coincides with the operator norm 
when interpreting that tensor as a member of $\LIN^{m+r}_{n+s}(X,Y)$. 

The contraction of tensors will help us formalize higher-order chain rules,
especially when composing multiple functions. 
When $\mu \in \LIN^{n}_{r}(Y,Z)$ and $\nu \in \LIN^{m}_{n}(X,Y)$,
then $\mu \contract \nu \in \LIN^{m}_{r}(X,Z)$ denotes the composition of these two linear mappings.
This is again inspired by notions of tensor calculus:
there, $\mu \contract \nu$ denotes the contraction between the $r$ covariant indices of $\mu$ and the $r$ contravariant indices of $\nu$. 
The contraction operation is commutative, which means that for any $\omega \in \LIN^{r}_{s}(Z,W)$ 
we also have $\omega \contract (\mu \contract \nu) = (\omega \contract \mu) \contract \nu$. 

Because we will formalize higher-order derivatives as symmetric multilinear forms, 
we discuss \emph{symmetric products}. 
Whenever $x_1, x_2, \ldots, x_m \in X$ are members of a vector space $X$,
we write 
\begin{align}\label{math:symmetricproduct}
    x_1 \sym x_2 \sym \cdots \sym x_m
    :=
    \frac{1}{m!}
    \sum_{ \pi \in \Perm(m) } 
    x_{\pi(1)} \otimes x_{\pi(2)} \otimes \cdots \otimes x_{\pi(m)}
\end{align}
for their symmetric product, 
where $\Perm(m)$ denotes the set of permutations of $\{1,2,\ldots,m\}$.
In particular, the symmetric tensor $x_{1} \sym x_{2} \sym \cdots \sym x_{m}$
is invariant under reordering of the vectors. 
Note that tensors of the form~\eqref{math:symmetricproduct} constitute a linear subspace of $\otimes^{m} X$,
which is denoted by $\sym^{m} X$ and called the \emph{$m$-th symmetric power of $X$}.

\subsection{Norms on tensors in finite dimensions}

The vector space $\bbR^{N}$ can be equipped with the usual $\ell^{p}$-norm $\|\cdot\|_{p}$. 
Whenever $p,q \in [1,\infty]$ with $p \leq q$, we have 
\begin{gather}\label{math:lpcomparison}
    \|x\|_{q} \leq \|x\|_{p} \leq N^{\frac 1 p - \frac 1 q} \|x\|_{q},
    \quad 
    x \in \bbR^N
    .
\end{gather}
When $p,q \in [1,\infty]$ with $1 = 1/p + 1/q$, then 
\begin{gather}\label{math:hoelderinequality}
    |\langle x, y \rangle| \leq \|x\|_{p} \|y\|_{q},
    \quad 
    x, y \in \bbR^N.
\end{gather}
Correspondingly, the dual space of $\bbR^{N}$ with the $\ell^{p}$-norm can be identified with the space $\bbR^{N}$ with the $\ell^{q}$-norm. 

More generally, the vector space ${\otimes}^d \bbR^{N}$ can be identified with the space $\LIN^{d}(\bbR^{N},\bbR)$ of $d$-linear functionals. We take a look at the spectral norm of those forms.
First, we consider the spectral norm of the simple tensor $x_1 \otimes x_2 \otimes \dots \otimes x_d$ composed of the vectors $x_1, x_2, \dots, x_d \in \bbR^N$.
Interpreting this as a multilinear form over the space $\bbR^N$ equipped with the $\ell^{p}$-norm,
its spectral norm is 
\begin{gather*}
    \left\|
        x_1 \otimes x_2 \otimes \dots \otimes x_d
    \right\|
    = 
    \|x_1\|_{q}
    \|x_2\|_{q}
    \cdots 
    \|x_d\|_{q}
    .
\end{gather*}
More generally, when $\omega \in {\otimes}^d \bbR^{N}$,
then it is evident that 
\begin{gather*}
    \left\| \omega \right\|
    \leq 
    \inf
    \left\{\;
        \textstyle{\sum_{i}}
        \|x_{1,i}\|_{q} \|x_{2,i}\|_{q} \cdots \|x_{d,i}\|_{q}
        \suchthat* 
        \omega = \textstyle{\sum_{i}} 
        x_{1,i} \otimes x_{2,i} \otimes \cdots \otimes x_{d,i}
    \;\right\}
     .
\end{gather*}
If we interpret $\omega$ as a vector in $\bbR^{N^d}$, using the canonical standard basis, then $\|\omega\|_{\infty} \leq \|\omega\| \leq \|\omega\|_{1}$.

We notice that the spectral norm satisfies the contraction inequality 
\begin{align}\label{math:contractioninequality}
    \| \mu \contract \nu \| \leq \|\mu\| \cdot \|\nu\|
    ,
    \quad \mu \in \LIN^{n}_{r}(Y,Z), \quad \nu \in \LIN^{m}_{n}(X,Y)
    .
\end{align}

\subsection{Derivatives}

Let $X$ and $Y$ be Banach spaces. 
Suppose that $\OmegaS \subseteq X$ is an open set 
and that we have a function $f : \OmegaS \rightarrow Y$.
We say that $f$ is differentiable at $x \in \OmegaS$ 
if there exists a bounded linear mapping $\nabla f_{|x} \in \LIN(X,Y)$
satisfying 
\begin{align*}
    \lim_{ v \rightarrow 0 }
    \dfrac{
        \| f ( x + v ) - f(x) - \nabla f_{|x}(v) \|_{Y}
    }{
        \| v \|_{X} 
    }
    = 0.
\end{align*}
In that case, we also have
\begin{align*}
    \nabla f_{|x}(v)
    =
    \lim_{ \tau \rightarrow 0 }
    \dfrac{
        f ( x + \tau v ) - f(x)
    }{
        \tau  
    },
    \quad 
    v \in X
    .
\end{align*}
The reader may notice that our notion of differentiability is exactly the notion of Frech\'et-differentiability. 
We call $\nabla f_{|x}$ the (first) \emph{derivative} of $f$ at $x$.
We write $\nabla_{v} f_{|x} = \nabla f_{|x}(v)$ 
and call this the directional derivative of $f$ at $x$ in direction $v$. 

The derivatives of higher order are defined recursively in the usual manner,
provided that $f$ is sufficiently regular. 
We write $\nabla^{m} f_{|x}$ for the $m$-th derivative of $f$ at $x$. 
Owing to the identification of $\LIN(X,\LIN^{m-1}(X,Y))$ with $\LIN^m(X,Y)$, 
the $m$-th derivative of $f$ at $x \in \OmegaS$ 
is a bounded multilinear form in $\LIN^{m}_{1}(X,Y)$,
having $m$ covariant indices and one contravariant index. 
The relevance of multilinear forms to the notion of derivatives is thus clear. 
Generalizing our notation for directional derivatives, 
we write 
\begin{align*}
    \nabla_{ v_1, v_2, \ldots, v_m } f_{|x}
    = 
    \nabla^{m} f_{|x}( v_1, v_2, \ldots, v_m ),
    \quad 
    v_1, v_2, \ldots, v_m \in X.
\end{align*}
Moreover, the  $m$-linear form $\nabla^{m} f_{|x}$ is symmetric,
that is, $\nabla^{m} f_{|x}( v_1, v_2, \ldots, v_m ) \in Y$ 
is invariant under reordering the vectors $v_1, v_2, \ldots, v_m \in X$.
That just means that $\nabla^{m} f_{|x}( v_1 \otimes v_2 \otimes \cdots \otimes v_m )$ 
only depends on the symmetric part of its argument. 
With some abuse of notation, we write 
\begin{align*}
    \nabla^{m} f_{|x}( v_1, v_2, \ldots, v_m )
    =
    \nabla^{m} f_{|x}( v_1 \otimes v_2 \otimes \cdots \otimes v_m )
    =
    \nabla^{m} f_{|x}( v_1 \sym v_2 \sym \cdots \sym v_m )
    ,
\end{align*}
where we interpret $\nabla^{m} f$ either 
as a multilinear $m$-form, as a linear form on the tensor product $\otimes^{m} X$, or
as a linear form on the symmetric power $\sym^{m} X$.  We will liberally switch between those different notations. 
\\

We explicitly state two auxiliary results on the Frech\'et derivative of mappings between Banach spaces.
We omit the proofs.

\begin{lemma}\label{lemma:frechetderivativeproductrule}
    Suppose that $X$, $Y_1$, and $Y_2$ are Banach spaces.
    Suppose that $\OmegaS \subseteq X$ is an open set.
If two mappings $f_1 : \OmegaS \rightarrow Y_1$ and  $f_2 : \OmegaS \rightarrow Y_2$
        are differentiable at $x \in \OmegaS$, then their pointwise tensor product
        \begin{align*}
            f : \OmegaS \rightarrow Y_1 \otimes Y_2,
            \quad 
            x \mapsto f_1(x) \otimes f_2(x)
        \end{align*}
        is differentiable at $x$, and we have 
\begin{align*}
            \nabla f_{|x}
            =
            \nabla f_{1|x} \otimes f_{2|x}
            +
            f_{1|x} \otimes \nabla f_{2|x}
            .
        \end{align*}
If two mappings $g_1 : \OmegaS \rightarrow Y_1$ and  $g_2 : \OmegaS \rightarrow Y_1^{\ast}$
        are differentiable at $x \in \OmegaS$,
        then their pointwise dual pairing
        \begin{align*}
            g : \OmegaS \rightarrow \bbR, \quad 
            x \mapsto \langle g_1(x), g_2(x) \rangle
        \end{align*}
        is differentiable at $x$, and we have 
\begin{align*}
            \nabla g_{|x}
            =
            \langle \nabla g_{1|x}, g_{2|x} \rangle
            +
            \langle g_{1|x}, \nabla g_{2|x} \rangle
            .
        \end{align*}
\end{lemma}

\subsection{Nested Partitions}

We work with sets of indices and now state some further combinatorial definitions. 
We write $\frakP(A)$ for the power set of any set $A$,
that is, the set of all subsets of $A$. 

The set of all ordered partitions of $A$ into $d+1$ disjoint, possibly empty sets is written 
\begin{align}\label{math:partition:maybeempty}
    {\calP}_{0}(A,d)
    &
    := 
    \left\{
        \calP = ( {P}_0, {P}_1, \ldots, {P}_d ) \in \frakP(A)^{d+1}
        \suchthat* 
        \begin{array}{c}
            A = \cup_{j=0}^{d} {P}_{j},
            \\
\forall i < j : {P}_{i} \cap {P}_{j} = \emptyset
        \end{array}
    \right\}
    .
\end{align}
The set of all (unordered) partitions of $A$ into $k$ pairwise disjoint non-empty sets is written 
\begin{align}\label{math:partition:nonempty}
    {\calP}(A,k)
    &
    := 
    \left\{
        \calP = \{ {P}_1, {P}_2, \ldots, {P}_k \} \subseteq \frakP(A)
        \suchthat* 
        \begin{array}{c}
            A = \cup_{j=1}^{k} {P}_{j},
\;
            \emptyset \notin \calP,
            \\
\forall i < j : {P}_{i} \cap {P}_{j} = \emptyset
        \end{array}
    \right\}
    .
\end{align}
We also write 
\begin{align}\label{math:partition:nonempty_all}
    {\calP}(A)
    &
    := 
    {\calP}(A,0)
    \cup
    {\calP}(A,1)
    \cup
    {\calP}(A,2)
    \cup
    \cdots 
    .
\end{align}
Note that we use different indexing conventions in the two definitions~\eqref{math:partition:maybeempty} and~\eqref{math:partition:nonempty}, 
and that the members of ${\calP}_{0}(A,k)$ are ordered whereas the members of ${\calP}(A,k)$ generally are not. 
That being said, in what follows, we occasionally fix an arbitrary order on the members of ${\calP}(A)$ to simplify the exposition; this will always be made explicit. 
\\

We generalize the classes of partitions introduced above. 
Specifically, we introduce hierarchies of partitions, each hierarchy having one level for each composition. 
For any set $A$, we recursively define 
\begin{align}
    \calP^{0}(A) := A,
    \quad 
    \calP^{1}(A) := \calP(A),
    \quad 
\calP^{l+1}(A) := \left\{ \calP(\calC) \suchthat* \calC \in \calP^{l}(A) \right\}.
\end{align}
We thus have the recursive relation 
\begin{align}
    \calP^{l+1}(A)
    = 
    \left\{
        \calD = \{ D_1, \ldots, D_k \} 
        \suchthat* 
        \begin{array}{c}
            \{ {P}_1, \ldots, {P}_k \} \in \calP(A),
            \\
            D_i = \calP^{l}({P}_i), \quad 1 \leq i \leq k
            .
        \end{array}
    \right\}
    .
\end{align}
Obviously, the set $\calP^{1}(A)$ is again the set of all partitions of $A$ into non-empty pairwise disjoint sets. 
We obtain $\calP^{2}(A)$ as the set of all possible partitions of partitions of $A$, 
and so forth for all $l > 2$. 
Again, we occasionally fix an arbitrary order on the members of ${\calP^{l}}(A)$ to simplify the exposition, and will be explicit whenever we use that.

We write $\indexset{1}{m} := \{ 1, 2, \ldots, m \}$ for any positive integer $m$,
and we abbreviate 
\begin{gather}\label{math:partition:integersets}
    \begin{gathered}
        {\calP}_{0}(m,d) := {\calP}_{0}(\indexset{1}{m},d),
        \quad 
        \calP(m,k)       := \calP(\indexset{1}{m},k),
        \\ 
        \calP(m)         := \calP(\indexset{1}{m}),
        \quad 
        \calP^{l}(m)     := \calP^{l}(\indexset{1}{m})
        .
    \end{gathered}
\end{gather}

\section{Chain rules of higher order}\label{sec:chainrule}

We commence our main work. 
We first prove a higher-order chain rule for the composition of two functions. 
We emphasize that this is not yet the famous Fa\`{a}~di~Bruno formula, 
which is only developed later in this section. 

\begin{proposition}\label{prop:higherorderchainrule}
    Assume that $X$, $Y$, and $\VEC$ are Banach spaces,
    and that $\OmegaS \subseteq X$ and $\OmegaT \subseteq Y$ are open sets. 
    Let $\phi : \OmegaS \rightarrow \OmegaT$ and ${f} : \OmegaT \rightarrow \VEC$ be mappings.
    If $x \in \OmegaS$ such that $\phi$ is $m$-times differentiable at $x$ and $f$ is $m$-times differentiable at $\phi(x)$,
    then ${f} \circ \phi$ is $m$-times differentiable at $x$ and we have 
\begin{align}\label{math:higherorderchainrule}
        \begin{split}
        &
        \nabla^{m}\left( {f} \circ \phi \right)_{|x}
        \left( v_1, v_2, \ldots, v_m \right)
        \\&\quad 
        =
        \sum_{ \substack{ 1 \leq k \leq m \\ \calP \in \calP(m,k) } }
        \nabla^{k} {f}_{|\phi(x)}
        \left( 
            \nabla^{|{P}_1|}\phi_{|x}( \sym_{ i \in {P}_{1} } v_{i} ), 
            \nabla^{|{P}_2|}\phi_{|x}( \sym_{ i \in {P}_{2} } v_{i} ), 
            \ldots, 
            \nabla^{|{P}_k|}\phi_{|x}( \sym_{ i \in {P}_{k} } v_{i} )
        \right)
\end{split}
    \end{align}    
    for all $v_1, v_2, \ldots, v_m \in X$.
\end{proposition}

\begin{proof}
    We use an induction argument.
    In the base case $m=1$, the statement is a reformulation of the classical chain rule
    for the composition of multivariate functions~\cite[Chapter~1.1]{hormander2015analysis}: 
    if $\phi$ is Frech\'et-differentiable at $x$ and $f$ is Frech\'et-differentiable at $\phi(x)$,
    then $f \circ \phi$ is Frech\'et-differentiable at $x$, and we have 
    \begin{align*}
        \nabla\left( {f} \circ \phi \right)_{|x}
        =
        \nabla {f}_{|\phi(x)} \contract \nabla\phi_{|x}
        .
    \end{align*}
    For the induction step, we assume that the statement is valid for $m$ derivatives,
    and we prove its validity for $m+1$ derivatives. 
    
    To that end, 
    suppose that $\calP \in \calP(m,k)$ is any partition of $\{1,\ldots,m\}$ 
    into $k$ non-empty pairwise disjoint sets, $\calP = \{ {P}_{1}, {P}_{2}, \ldots, {P}_{k} \}$. 
    Given any direction vector $v_{m+1}$,
    we apply the Leibniz rule (Lemma~\ref{lemma:frechetderivativeproductrule}) and the classical chain rule once more to find 
    \begin{align*}
        &
        \nabla_{v_{m+1}} 
        \left( 
        \nabla^{k} {f}_{|\phi(x)}
            \left(
                \nabla^{|{P}_1|}\phi_{|x}( \sym_{ i \in {P}_{1} } v_{i} ), 
                \nabla^{|{P}_2|}\phi_{|x}( \sym_{ i \in {P}_{2} } v_{i} ), 
                \ldots, 
                \nabla^{|{P}_k|}\phi_{|x}( \sym_{ i \in {P}_{k} } v_{i} )
            \right)
        \right)
        \\&\quad
        =
        \nabla^{k+1} {f}_{|\phi(x)}
        \left( 
            \nabla_{v_{m+1}}\phi_{|x}, 
            \nabla^{|{P}_1|}\phi_{|x}( \sym_{ i \in {P}_{1} } v_{i} ), 
            \nabla^{|{P}_2|}\phi_{|x}( \sym_{ i \in {P}_{2} } v_{i} ), 
            \ldots, 
            \nabla^{|{P}_k|}\phi_{|x}( \sym_{ i \in {P}_{k} } v_{i} )
        \right)
        \\&\quad\quad
        \quad+
        \nabla^{k} {f}_{|\phi(x)}
        \left( 
            \nabla_{v_{m+1}}\nabla^{|{P}_1|}\phi_{|x}( \sym_{ i \in {P}_{1} } v_{i} ), 
            \nabla^{|{P}_2|}\phi_{|x}( \sym_{ i \in {P}_{2} } v_{i} ), 
            \ldots, 
            \nabla^{|{P}_k|}\phi_{|x}( \sym_{ i \in {P}_{k} } v_{i} )
        \right)
        \\&\quad\quad
        \quad+
        \nabla^{k} {f}_{|\phi(x)}
        \left( 
            \nabla^{|{P}_1|}\phi_{|x}( \sym_{ i \in {P}_{1} } v_{i} ), 
            \nabla_{v_{m+1}}\nabla^{|{P}_2|}\phi_{|x}( \sym_{ i \in {P}_{2} } v_{i} ), 
            \ldots, 
            \nabla^{|{P}_k|}\phi_{|x}( \sym_{ i \in {P}_{k} } v_{i} )
        \right)
        \\&\quad\quad
        \quad\quad\quad\quad\vdots
        \\&\quad\quad
        \quad+
        \nabla^{k} {f}_{|\phi(x)}
        \left( 
            \nabla^{|{P}_1|}\phi_{|x}( \sym_{ i \in {P}_{1} } v_{i} ), 
            \nabla^{|{P}_2|}\phi_{|x}( \sym_{ i \in {P}_{2} } v_{i} ), 
            \ldots, 
            \nabla_{v_{m+1}}\nabla^{|{P}_k|}\phi_{|x}( \sym_{ i \in {P}_{k} } v_{i} )
        \right)
        .
    \end{align*}
Each summand on the right corresponds to a partition $\calP(m+1,k)$. 
    The first term on the right-hand side of the equation corresponds to a partition $\calP^{(0)} \in \calP(m+1,k)$
    that we obtain by adding a new singleton set to the existing partition $\calP$,
    that is, $\calP^{(0)} = \{\{m+1\}\} \cup \calP$. 
    The remaining $k$ summands correspond to partitions $\calP^{(1)}, \calP^{(2)}, \ldots, \calP^{(k)} \in \calP(m+1,k)$ obtained by adding the index $m+1$ to any of the existing $k$ members of $\calP$, that is,
    \begin{align*}
        \calP^{(l)} := ( \calP \setminus {P}_{l} ) \cup ( {P}_{l} \cup \{m+1\} ), \quad 1 \leq l \leq k.
    \end{align*}
    Any partition newly created in that manner is created from exactly one ``parent partition''. 
    The induction step is now evident, and the proof is complete. 
\end{proof}

\begin{remark}
    Under the assumptions of the proposition, the statement can be written equivalently in numerous ways. 
    For example, 
    \begin{align}
        \nabla^{m}\left( f \circ \phi \right)_{|x}
        \left( v_1 \sym v_2 \sym \cdots \sym v_m \right)
        &=
        \sum_{ \substack{ 1 \leq k \leq m \\ \calP \in \calP(m,k) } }
        \nabla^{k} {f}_{|\phi(x)}
        \left( 
            \sym_{{P} \in \calP} \nabla^{|{P}|}\phi_{|x}\left( \sym_{ i \in P } v_{i} \right)
        \right)
        ,
    \end{align}
    where the symmetric products are taken only in the contravariant indices of the tensors.
\end{remark}

\begin{remark} 
    The $k$-th derivative is a symmetric tensor.
    However, the reader may notice that the individual summands appearing in the sum are not symmetric tensors.
    In each summand, we can generally only swap vectors within a block, 
    and hence that tensor enjoys symmetries with respect to a subgroup of the symmetric group.
    It is only the summing over all partitions that restores the full symmetry.
\end{remark}

\begin{example}\label{example:inversefunction:multivariate}
    We illustrate Proposition~\ref{prop:higherorderchainrule}
    with an application to the inverse function theorem. 
    Suppose that $\phi : \OmegaS \rightarrow \OmegaT$ is invertible. 
    If $\phi$ is $m$-times differentiable at $x \in \OmegaS$,
    then its inverse $\phi^{-1}$ is $m$-times differentiable at $y := \phi(x) \in \OmegaT$ (see~\cite[Ch1.1]{hormander2015analysis}). 
    Using~\eqref{math:higherorderchainrule} with the composition $\phi \circ \phi^{-1}$ and any $m \geq 2$,
    we find 
    \begin{align*}
        0
        &
        =
        \sum_{ \substack{ 1 \leq k \leq m \\ \calP \in \calP(m,k) } }
        \nabla^{k} \phi_{|\phi^{-1}(y)}
        \left( 
            \nabla^{|{P}_1|}\phi^{-1}_{|y}( \sym_{ i \in {P}_{1} } v_{i} ), 
            \nabla^{|{P}_2|}\phi^{-1}_{|y}( \sym_{ i \in {P}_{2} } v_{i} ), 
            \ldots, 
            \nabla^{|{P}_k|}\phi^{-1}_{|y}( \sym_{ i \in {P}_{k} } v_{i} )
        \right)
        .
    \end{align*}
    Note that $\calP(m,1) = \{\{\indexset{1}{m}\}\}$ has exactly one member. 
    This observation leads to 
    \begin{align*}
        &
        -
        \nabla \phi_{|\phi^{-1}(y)}
        \nabla^{m}\phi^{-1}_{|y}(v_1,\ldots,v_m)
        \\&\quad 
        =
        \sum_{ \substack{ 2 \leq k \leq m \\ \calP \in \calP(m,k) } }
        \nabla^{k} \phi_{|\phi^{-1}(y)}
        \left( 
            \nabla^{|{P}_1|}\phi^{-1}_{|y}( \sym_{ i \in {P}_{1} } v_{i} ), 
            \nabla^{|{P}_2|}\phi^{-1}_{|y}( \sym_{ i \in {P}_{2} } v_{i} ), 
            \ldots, 
            \nabla^{|{P}_k|}\phi^{-1}_{|y}( \sym_{ i \in {P}_{k} } v_{i} )
        \right)
        .
    \end{align*}
    The right-hand side of that identity only includes the derivatives of $\phi^{-1}$ up to order $m-1$. 
    According to the inverse function theorem, 
    the inverse of $\nabla \phi_{|\phi^{-1}(y)}$ equals the differential $\nabla \phi^{-1}_{y}$.
    We obtain 
    \begin{align}\label{math:explicitformulaforinverse}
        \begin{aligned}
        &
        \nabla^{m}\phi^{-1}_{|y}(v_1,\ldots,v_m)
        \\&\quad 
        =
        -
        \nabla \phi^{-1}_{|y}
        \sum_{ \substack{ 2 \leq k \leq m \\ \calP \in \calP(m,k) } }
        \nabla^{k} \phi_{|\phi^{-1}(y)}
        \left( 
            \nabla^{|{P}_1|}\phi^{-1}_{|y}( \sym_{ i \in {P}_{1} } v_{i} ), 
\ldots, 
            \nabla^{|{P}_k|}\phi^{-1}_{|y}( \sym_{ i \in {P}_{k} } v_{i} )
        \right)
        .
        \end{aligned}
    \end{align}
Equation~\eqref{math:explicitformulaforinverse} provides a recursive formula 
    that produces successively higher derivatives of the inverse mapping. 
\end{example}

Proposition~\ref{prop:higherorderchainrule} provides a general formula for the higher derivatives of compositions,
including multivariate vector-valued functions.
Next, 
we address the special case of univariate functions.
On that occasion, we also introduce Bell polynomials.

We consider the special case of a composition $f \circ \phi$
of univariate functions $f$ and $\phi$.
In that case, the linearity of the differential means that all directional derivatives of $\phi$ agree up to scaling. 
We can write the higher-order chain rule as 
\begin{align}\label{math:higherorderchainrule:univariate}
    \nabla^{m} \left( f \circ \phi \right)_{|x}
    = 
    \sum_{ \substack{ 1 \leq k \leq m \\ \calP \in \calP(m,k) } }
    \nabla^{|\calP|}f_{|\phi(x)}
    \left( 
    \nabla^{|{P}_1|} \phi_{|x},
    \nabla^{|{P}_2|} \phi_{|x},
    \ldots 
    \nabla^{|{P}_k|} \phi_{|x} 
    \right)
    .
\end{align}
Some of the summands in that expression are redundant, since directional derivatives are all colinear in the univariate case.
We therefore simplify the expression. 

For every ${P} \in \calP(m,k)$, we consider the cardinalities of its members ${j}_{1}, {j}_{2}, \ldots, {j}_{k}$ 
in increasing order; none of them are zero. Members of $\calP(m,k)$ that give the same sequence of cardinalities lead to the same summand in~\eqref{math:higherorderchainrule:univariate}. 
For every such cardinality signature ${j}_{1}, {j}_{2}, \ldots, {j}_{k}$ of some member of $\calP(m,k)$, 
let $b_{i}$ indicate how many of these cardinalities equal $i$. 
Thus we obtain another sequence $b_{1},\ldots,b_m \in \bbN_{0}$, 
which by construction must satisfy
\begin{gather*}
    b_1 + b_2 + \cdots + b_m = k,
    \quad 
    1 b_1 + 2 b_2 + \cdots + m b_m = m,
\end{gather*}
In fact, this procedure produces exactly all the non-negative tuples $(b_1,\ldots,b_m)$ whose entries satisfy these summation identities, 
and every increasing cardinality sequence ${j}_{1}, {j}_{2}, \ldots, {j}_{k}$ with $1 \leq k \leq m$
can be reconstructed from the sequence $b_{1},\ldots,b_m \in \bbN_{0}$.

We now ask how many summands in~\eqref{math:higherorderchainrule:univariate} lead to the same cardinality signature ${j}_{1}, {j}_{2}, \ldots, {j}_{k}$, because those summands are like terms. 
Intuitively, given $k$ buckets that are only distinguishable by their sizes ${j}_1 \leq \ldots \leq {j}_{k}$, 
the number of ways of distributing the $m$ derivatives into each bucket is $m!$ 
divided by ${j}_1! {j}_2! \cdots {j}_m!$ and by $b_1! b_2! \cdots b_m!$. 
The first divisor expresses that derivatives commute 
and the second divisor expresses that the blocks of equal size are unordered. 
This insight leads to an expression using multinomial coefficients: 
\begin{align*}
    \nabla^{m}\left( f \circ \phi \right)_{|x}
    =
    \sum_{k=1}^{m}
    \sum_{ \substack{ 
        {j}_1, \ldots, {j}_{k} \in \bbN 
        \\ 
        {j}_1 \leq \cdots \leq {j}_{k} 
        \\ 
        {j}_1 + \cdots + {j}_{k} = m
    } }
{ m \choose {j}_1, \ldots, {j}_k }
\frac{1}{ b_1! b_2! \cdots b_m! }
\nabla^{k} f_{|\phi(x)}
    ( \nabla^{{j}_1} \phi_{|x}, \nabla^{{j}_2} \phi_{|x}, \ldots, \nabla^{{j}_k} \phi_{|x} )
    ,
\end{align*}
where the numbers $b_1,\ldots,b_m$ are defined for each $j_1,\ldots,j_k$ as outlined above. 
Note that $j_k \leq m-k+1$ in each summand above.
However, it is more common to write the sum in terms of the numbers $b_1,\ldots,b_m$ discussed above. 
This leads to the \emph{Fa\`a~di~Bruno formula}:
\begin{align}\begin{split}\label{math:faadibrunoformula}
    \nabla^{m} (f \circ \phi)_{|x}
    &
    =
    \sum_{k=1}^{m}
    \sum_{ \substack{
        b = (b_1, b_2, \ldots, b_m) \in \bbN_{0}^{m}
        \\
        b_1 + 2 b_2 + \cdots + m b_m = m
        \\
        b_1 +   b_2 + \cdots +   b_m = k
        } 
    }
    \frac{m!}{ b_1! b_2! \cdots b_m! }
     \nabla^{(k)} f_{|\phi(x)}
\prod_{j=1}^{m-k+1}
    \left( \frac{ \nabla^{j} \phi_{|x} }{ j!} \right)^{b_j}
    \\&
    =
    \sum_{ \substack{
        b = (b_1, b_2, \ldots, b_m) \in \bbN_{0}^{m}
        \\
        b_1 + 2 b_2 + \cdots + m b_m = m } 
    }
    \frac{m!}{ b_1! b_2! \cdots b_m! }
    \nabla^{(b_1+b_2+\cdots+b_m)} f_{|\phi(x)}
\prod_{j=1}^{m}
    \left( \frac{ \nabla^{j} \phi_{|x} }{ j!} \right)^{b_j}
    .
\end{split}\end{align}

\begin{remark}
    The Fa\`{a}~di~Bruno formula~\eqref{math:faadibrunoformula}
    can be interpreted as a common specialization of the higher-order chain rule when composing two univariate functions. 
    That specialization reduces the number of summands. 
    But that specialization is not constitutive for the chain rule of higher order. 
    Its underlying structure is elucidated by formulating the generalized chain rule via tensor analysis in Proposition~\ref{prop:higherorderchainrule}.
    We also remark that the case of a multivariate outer function ${f}$ and an univariate inner function $\phi$ was treated first by Most~\cite{most1871ueber}.
\end{remark}

We reformulate the Fa\`{a}~di~Bruno formula concisely in terms of Bell polynomials. 
That formalism is particularly common in combinatorics. 
To make the notation more concise, we first introduce what we call the \emph{Scherk indices},
\begin{align}
    \Scherkset(m,k)
    = 
    \left\{
        b = (b_1, b_2, \ldots, b_m) \in \bbN_{0}^{m}
        \suchthat* 
\sum_{i=1}^{m} b_i = k,
\;
\sum_{i=1}^{m} i b_i = m
\right\},
\end{align}
and the corresponding \emph{Scherk coefficients},
\begin{align}
    \Scherkcof^b
    = 
    \frac{m!}{ b_1! b_2! \cdots b_m! (1!)^{b_1} (2!)^{b_2} \cdots (m!)^{b_m}}
    ,
    \quad 
    b \in \Scherkset(m,k)
    .
\end{align}
The Scherk coefficient\footnote{It seems that Heinrich Ferdinand Scherk's 1823 dissertation is the first known occurrence of these coefficients in writing, including the fact that the coefficients multiplied to $\nabla^k f$ only involve derivatives up to the index $m-k+1$.} $\Scherkcof^b$ counts in how many ways we can distribute $m$ elements into $b_1$ unordered buckets of size $1$, $b_2$ buckets of size $2$, etc.,
where the buckets do not order their content, and buckets can be reordered arbitrarily. 

The Bell polynomials arise in the context of the higher-order chain formula.
We define the partial Bell polynomials $B_{m,k}$ by 
\begin{gather}\label{math:bellpolynomial:partial}
    B_{m,k}(x_1,\ldots,x_k)
    =
    \sum_{ b \in \Scherkset(m,k) }
    \Scherkcof^b
    \prod_{j=1}^{m-k+1}
    (x_j)^{b_j} ,
\end{gather}
and the full Bell polynomials $B_{m}$ as
\begin{gather}\label{math:bellpolynomial:full}
B_{m}( x^{[1]}_{1}, x^{[1]}_{2}, \ldots, x^{[1]}_{m}, x^{[0]}_{1}, x^{[0]}_{2}, \ldots, x^{[0]}_{m} ) 
=
    \sum_{ 1 \leq k \leq m }
    x^{[1]}_{k}
    B_{m,k}\left(
        x^{[0]}_{1}, x^{[0]}_{2}, \ldots, x^{[0]}_{m-k+1} 
    \right)
.
\end{gather}
Note that these polynomials have non-negative coefficients 
and that $B_{m,k}$ is homogeneous of degree $k$. 
These polynomials frequently appear in subsequent discussions. 
For example, we can rewrite the $m$-th derivative of the composition $f \circ \phi$
of two univariate functions $f$ and $\phi$ as 
\begin{align}\label{math:faadibruno:bellpolynomial}
    \begin{split}
    \nabla^{m} (f \circ \phi)_{|x}
    &= 
    \sum_{ 1 \leq k \leq m }
    \nabla^{k} f_{|\phi(x)}
    B_{m,k}( \nabla \phi_{|x}, \nabla^{2} \phi_{|x}, \ldots, \nabla^{m-k+1} \phi_{|x} )
    \\&=     
    B_{m}\left(
        \nabla^{1} f_{|\phi(x)}, \nabla^{2} f_{|\phi(x)}, \ldots, \nabla^{m} f_{|\phi(x)},
        \nabla \phi_{|x}, \nabla^{2} \phi_{|x}, \ldots, \nabla^{m} \phi_{|x} 
    \right)
    .
    \end{split}
\end{align}

\begin{remark}
The values $B_{m,k}(1,\ldots,1)$, 
which are the sums of Scherk coefficients corresponding to the indices in $\Scherkset(m,k)$,
are also known as the \emph{Stirling numbers of second kind}.
These count the possible ways to distribute $m$ distinguishable items into $k > 0$ non-empty sets.
Their sum over $k$ is the $m$-th Bell number,
which counts the number of partitions of $m$ distinguishable items into non-empty sets. 
\end{remark}

\begin{example} We resume the discussion of inverse functions that we started in Example~\ref{example:inversefunction:multivariate}. 
    We impose the additional assumption that $\OmegaS, \OmegaT \subseteq \bbR$ are subsets of the real line,
    so $\phi$ and its inverse are univariate functions. 
We see that 
    \begin{align*}
        0
        =
        \nabla^{m} (\phi \circ \phi^{-1})_{|y}
        = 
        \sum_{ 1 \leq k \leq m }
        \nabla^{k} \phi_{|\phi^{-1}(y)}
        B_{m,k}( \nabla \phi^{-1}_{|y}, \nabla^{2} \phi^{-1}_{|y}, \ldots, \nabla^{m-k+1} \phi^{-1}_{|y} )
        .
    \end{align*}
    We extract the term corresponding to $k=1$, and multiply with the inverse of $\nabla\phi_{\phi^{-1}(y)}$. An application of the inverse function theorem leads to 
    \begin{align*}
        \nabla^{m} \phi^{-1}_{|y}
        &
        = 
        -
        \nabla \phi^{-1}_{|y}
        \sum_{ 2 \leq k \leq m }
        \nabla^{k} \phi_{|\phi^{-1}(y)}
        B_{m,k}( \nabla \phi^{-1}_{|y}, \nabla^{2} \phi^{-1}_{|y}, \ldots, \nabla^{m-k+1} \phi^{-1}_{|y} )
        \\&
        = 
        -
        \nabla \phi^{-1}_{|y}
        \sum_{ \substack{
            2 \leq k \leq m
            \\
            (b_1, b_2, \ldots, b_m) \in \bbN_{0}^{m}
            \\
            b_1 +   b_2 + \cdots +   b_m = k
            \\
            b_1 + 2 b_2 + \cdots + m b_m = m
        } }
        \nabla^{k} \phi_{|\phi^{-1}(y)}
        \frac{ 
            m! 
        }{
            b_1! b_2! \cdots b_m!
        }
        \prod_{j=1}^{m-k+1}
        \left( \frac{ \nabla^{j} \phi^{-1}(y) }{ j!} \right)^{b_j}
    \end{align*}
    as a recursively formula for $m$-th of the inverse of any univariate function. 
\end{example}

\section{Higher-order chain rule for multicomposite functions}\label{sec:multicomposite}

Having discussed the higher derivatives of a composition of two functions, 
we now consider the higher derivatives of a composition of multiple functions. 
While we have formalized the generalized chain rule for composed functions using partitions of index sets, 
we formalize the generalized chain rule for multicomposite functions using nested partitions of index sets. 
\\

Throughout this section, we assume that 
$X_{0}, X_{1}, \ldots, X_{l}, X_{l+1}$ are Banach spaces 
and that we have open subsets 
\begin{align*}
    \OmegaS_{0} \subseteq X_{0}, 
    \quad 
    \OmegaS_{1} \subseteq X_{1},
    \quad
    \ldots,
    \quad
    \OmegaS_{l} \subseteq X_{l},
    \quad 
    \OmegaS_{l+1} \subseteq X_{l+1}
    .
\end{align*}
The multivariate higher-order chain rule concerns compositions $f_{l} \circ f_{l-1} \circ \cdots \circ f_1 \circ f_0$ of functions 
\begin{align*}
    f_i : \OmegaS_{i} \subseteq X_{i} \rightarrow \OmegaS_{i+1} \subseteq X_{i+1}.
\end{align*}
The full composition $f_{l} \circ \cdots \circ f_0$ thus maps from $\OmegaS_{0}$ into $\OmegaS_{l+1}$. 
We take successive derivatives in $m$ vectors $v_1,v_2,\ldots,v_m \in X_{0}$,
and so our goal is finding a formula for 
\begin{align*}
    \nabla^{m}\left( f_{l} \circ \cdots \circ f_1 \circ f_0 \right)_{|x}
    \left( v_1, v_2, \ldots, v_m \right)
    .
\end{align*}
The $m$-th derivative of the full composition has $m$ covariant indices. 
\\

For the discussion of the chain rule, we introduce more notation. 
When $x \in \OmegaS_{0}$, we set $x_{0} := x$
and write $x_{i} := ( f_{i} \circ \cdots \circ f_{0} )(x)$ for any $0 \leq i \leq l$.
In other words, we recursively define 
\begin{align*}
    x_{l+1} := f_{l  }(x_{l  })
    ,\quad 
    x_{l  } := f_{l-1}(x_{l-1})
    ,\quad 
    \ldots 
    \quad 
    x_{2  } := f_{1  }(x_{1  })
    ,\quad 
    x_{1  } := f_{0  }(x_{0  })
    .
\end{align*}
When $A \subseteq \indexset{1}{m}$, we set 
\begin{align}\label{math:recursivenotation:base}
    \nabla_{A} f_{0|x_{0}}
    :=
    \nabla^{|A|}f_{0|x_{0}}\left( \vee_{i \in A} v_{i} \right).
\end{align}
More generally, we make the following recursive definition. 
When $l \in \bbN_{0}$ and $\calP \in \calP^{l}(A)$, we define 
\begin{align}\label{math:recursivenotation:step}
    \nabla_{\calP} f_{l+1|x_{l+1}}
    :=
    \nabla^{|\calP|}f_{l+1|x_{l+1}}\left( \vee_{{P} \in \calP} \nabla_{P} f_{l|x_{l}} \right).
\end{align}
Here, the symmetrization over the members of $\calP$ is only taken over the contravariant indices. 
\\

The notation is now sufficient to formulate the higher-order chain rule. 

\begin{proposition}\label{prop:multicompositechainrule}
    Under assumptions of this section, 
    \begin{align*}
        \nabla^{m}\left( f_{l} \circ \cdots \circ f_1 \circ f_0 \right)_{|x}
        \left( v_1, v_2, \ldots, v_m \right)
        = 
        \sum_{ \calP \in \calP^{l}(m) }
        \nabla_{\calP} f_{l|x_{l}}
.
    \end{align*}
\end{proposition}

\begin{proof}
 This is proven by an induction argument over $l$.
 In the base $l=1$, this is Proposition~\ref{prop:higherorderchainrule} again. 
 So let us assume that the statement is true for some positive integer $l \in \bbN$.
 Write $\phi = f_{l} \circ \cdots \circ f_1 \circ f_0$.
 Proposition~\ref{prop:higherorderchainrule} implies 
 \begin{align*}
    &
    \nabla^{m}\left( f_{l+1} \circ \phi \right)_{|x}
    \left( v_1, v_2, \ldots, v_m \right)
    \\&\quad
    =
    \sum_{ \calP \in \calP(m) }
    \nabla^{|\calP|} f_{|\phi(x)}
    \left( 
        \nabla^{|{P}_1|}\phi_{|x}( \vee_{ i \in {P}_{1} } v_{i} ), 
        \nabla^{|{P}_2|}\phi_{|x}( \vee_{ i \in {P}_{2} } v_{i} ), 
        \ldots, 
        \nabla^{|{P}_{|\calP|}|}\phi_{|x}( \vee_{ i \in {P}_{|B|} } v_{i} )
    \right)
    .
 \end{align*}    
 By the induction assumption and with the notation~\eqref{math:recursivenotation:base} and~\eqref{math:recursivenotation:step},
 we have 
 \begin{align*}
    \nabla^{|{P}_j|}\phi_{|x}( \vee_{ i \in {P}_{j} } v_{i} )
    =
    \sum_{ \calD \in \calP^{l}({P}_j) }
    \nabla_{\calD} f_{l|x_{l}}
\end{align*}
 in the summands above. Substituting these sums, we arrive at 
 \begin{align*}
    &
    \nabla^{m}\left( f_{l+1} \circ \phi \right)_{|x}
    \left( v_1, v_2, \ldots, v_m \right)
    \\&\quad
    =
    \sum_{ k=1 }^{\infty}
    \sum_{ \substack{ \calP \in \calP(m) \\ |\calP|=k }}
    \sum_{ \calD_{1} \in \calP^{l}({P}_1) }
    \cdots 
    \sum_{ \calD_{k} \in \calP^{l}({P}_k) }
    \nabla^{|\calP|} f_{|\phi(x)}
    \left( 
        \nabla_{\calD_{1}} f_{l|x_{l}}, 
        \nabla_{\calD_{2}} f_{l|x_{l}}, 
        \ldots, 
        \nabla_{\calD_{k}} f_{l|x_{l}} 
    \right)
    \\&\quad
    =
    \sum_{ \substack{ \calP \in \calP^{l+1}(m) }}
    \nabla^{|\calP|} f_{|\phi(x)}
    \left( 
        \vee_{ \calD \in \calP }
        \nabla_{\calD} f_{l|x_{l}} 
    \right)
    .
 \end{align*}    
 This has the form as in the statement of this proposition. 
 The proof is complete. 
\end{proof}

The next step is rewriting this higher-order chain formula for compositions 
in terms of tensor contractions. 
Towards that end, we once more introduce additional definitions. 
Suppose that $\calP \in \calP^{l}(m)$.
We define 
\begin{align}\label{math:specialnotation:tensors}
    \begin{split}
    D^{l,\calP} f_{l|x_{l}}
    &:=
    \nabla^{|\calP|} f_{l|x_{l}}
    ,
    \\
    D^{l-1,\calP} f_{l-1|x_{l-1}}
    &:=
    \bigotimes_{ {P}_{l-1} \in \calP   } \nabla^{|{P}_{l-1}|} f_{l-1|x_{l-1}}
    ,
    \\
    D^{l-2,\calP} f_{l-2|x_{l-2}}
    &:=
    \bigotimes_{ {P}_{l-1} \in \calP   }
    \bigotimes_{ {P}_{l-2} \in {P}_{l-1} }
    \nabla^{|{P}_{l-2}|} f_{l-2|x_{l-2}}
    ,
    \\
    &\quad\vdots 
    \\
    D^{1,\calP} f_{1|x_1}
    &:=
    \bigotimes_{ {P}_{l-1} \in \calP   }
    \bigotimes_{ {P}_{l-2} \in {P}_{l-1} }
    \cdots 
    \bigotimes_{ {P}_{1  } \in {P}_{2  } }
    \nabla^{|{P}_{1  }|} f_{1|x_1}
    ,
    \\
    D^{0,\calP} f_{0|x_0}
    &:=
    \bigotimes_{ {P}_{l-1} \in \calP   }
    \bigotimes_{ {P}_{l-2} \in {P}_{l-1} }
    \cdots 
    \bigotimes_{ {P}_{1  } \in {P}_{2  } }
    \bigotimes_{ B       \in {P}_{1} }
    \nabla^{|B|} f_{0|x_0}
    .
    \end{split}
\end{align}
In an analogous manner, 
given $v_1,\ldots,v_m \in X_{0}$, we define 
\begin{align}\label{math:specialnotation:vectors}
    \mu^{\calP}
    :=
    \bigotimes_{ {P}_{l-1} \in \calP   }
    \bigotimes_{ {P}_{l-2} \in {P}_{l-1} }
    \cdots 
    \bigotimes_{ {P}_{  0} \in {P}_{  1} }
    \bigotimes_{ i       \in {P}_{  0} }
    v_{i}
    .
\end{align}
Notice that these definitions rely on the arbitrary but fixed order of the sets over whose members we sum. 

\begin{proposition}\label{prop:multicompositechainrule:alternative}
    Under the assumptions of Proposition~\ref{prop:multicompositechainrule}
    and using the notation~\eqref{math:specialnotation:tensors}
    and~\eqref{math:specialnotation:vectors},
    we have 
    \begin{align*}
        &
        \nabla^{m}
        \left( f_{l} \circ \cdots \circ f_1 \circ f_0 \right)_{|x}
        \left( v_1, v_2, \ldots, v_m \right)
        \\&\quad
        = 
        \sum_{ \calP \in \calP^{l} }
        D^{l  ,\calP} f_{l  |x_{l  }}
        \contract
        D^{l-1,\calP} f_{l-1|x_{l-1}}
        \contract
        \cdots
        \contract
        D^{  1,\calP} f_{  1|x_{  1}}
        \contract
        D^{  0,\calP} f_{  0|x_{  0}}
        \contract
        \mu_{\calP}
        \\&\quad
        = 
        \sum_{ {P}_{l} \in \calP^{l} }
        \nabla^{| {P}_{l} |} f_{l  |x_{l  }}
        \contract
        \left( 
            \bigotimes_{ {P}_{l-1} \in {P}_{l} }
            \nabla^{ |{P}_{l-1}| } f_{l-1|x_{l-1}}
            \contract
            \left( 
                \bigotimes_{ {P}_{l-2} \in {P}_{l-1} }
                \nabla^{ |{P}_{l-2}| } f_{l-2|x_{l-2}}
                \contract
                \cdots 
                    \left( 
                        \bigotimes_{ {P}_{0} \in {P}_{1} }
                        \nabla^{ |{P}_{0}| } f_{0|x_{0}}
                        \contract
                        \mu_{{P}_{0}}
                    \right)
                \cdots 
            \right)
        \right) 
        .
    \end{align*}
\end{proposition}

\begin{proof}
    We use induction. 
    The statement is true for $l=1$.
    For the induction step, we assume the statement holds when for some number $l$. 
    As in the preceding proof,
    \begin{align*}
        \nabla^{m}\left( f_{l+1} \circ \phi \right)_{|x}
        \left( v_1, v_2, \ldots, v_m \right)
        &
        =
        \sum_{ \substack{ \calP \in \calP^{l+1}(m) }}
        \nabla^{|\calP|} f_{|\phi(x)}
        \left( 
            \vee_{ \calD \in \calP }
            \nabla_{\calD} f_{l|x_{l}} 
        \right)
        \\&
        =
        \sum_{ \substack{ \calP \in \calP^{l+1}(m) }}
        \nabla^{|\calP|} f_{|\phi(x)}
        \contract 
        \bigotimes_{ \calD \in \calP }
        \nabla_{\calD} f_{l|x_{l}} 
        .
    \end{align*}
    This completes the proof.
\end{proof}

We once more consider the special case of univariate functions
and specialize the general higher-order chain rule for this special case. 
Suppose that $f_{l} \circ f_{l-1} \circ \cdots \circ f_{1} \circ f_{0}$
is the composition of univariate real-valued functions $f_{i} : \OmegaS_{i} \rightarrow \OmegaS_{i+1}$. 
We define $B^{[1]}_{m,k} = B_{m,k}$ and $B^{[1]}_{m} = B_{m}$, so that 
\begin{align*}
    \nabla^{m} (f_{1} \circ f_{0})
    _{|x_0}
    &= 
    \sum_{ 1 \leq k \leq m }
    \nabla^{k} f_{1|x_1}
    B^{[1]}_{m,k}( \nabla f_{0|x}, \ldots, \nabla^{m-k+1} f_{0|x} )
    \\&=  
    B^{[1]}_{m}\left(
        \nabla^{1} f_{1|x_1}, \ldots, \nabla^{m} f_{1|x_1},
        \nabla f_{0|x}, \ldots, \nabla^{m} f_{0|x} 
    \right)
    .
\end{align*}
We obviously have the recursive relation 
\begin{align*}
    &
    \nabla^{m} (f_{l+1} \circ f_{l} \circ \cdots \circ f_{0})
    _{|x_0}
    \\&\quad
    = 
    \sum_{ 1 \leq k \leq m }
    \nabla^{k} f_{l+1|x_{l+1}}
    B_{m,k}\left(
        \nabla (f_{l} \circ \cdots \circ f_{0})_{|x_0}, \ldots, \nabla^{m-k+1} (f_{l} \circ \cdots \circ f_{0})_{|x_0}
    \right)
    \\&\quad
    = 
    B_{m}\left(
        \nabla^{1} f_{l+1|x_{l+1}}, \ldots, \nabla^{m} f_{l+1|x_{l+1}},
        \nabla (f_{l} \circ \cdots \circ f_{0})_{|x_0}, \ldots, \nabla^{m} (f_{l} \circ \cdots \circ f_{0})_{|x_0} 
    \right)
    .
\end{align*}
This leads to a discussion of higher-level Bell polynomials. 
We recursively define the higher-level partial Bell polynomials $B^{[l+1]}_{m,k}$ and Bell polynomials $B^{[l+1]}_{m}$ by setting 
\begin{align*}
B^{[l+1]}_{m,k}( x^{[l]}_{1}, \ldots, x^{[l]}_{m-k+1}, \ldots, x^{[0]}_{1}, \ldots, x^{[0]}_{m-k+1} )
:= 
    B_{m,k}\big(
        &
        B^{[l]}_{1}    ( x^{[l]}_{1}, \ldots, x^{[0]}_{1} ),
        \\&
        B^{[l]}_{2}    ( x^{[l]}_{1}, x^{[l]}_{2}, \ldots, x^{[0]}_{1}, x^{[0]}_{2} ),
        \\&
        \ldots,
        \\&
        B^{[l]}_{m-k+1}( x^{[l]}_{1}, \ldots, x^{[l]}_{m-k+1}, \ldots, x^{[0]}_{1}, \ldots, x^{[0]}_{m-k+1} ) 
    \big)
\end{align*}
and 
\begin{align*}
    &
    B^{[l+1]}_{m}( x^{[l+1]}_{1}, \ldots, x^{[l+1]}_{m}, x^{[l]}_{1}, \ldots, x^{[l]}_{m}, \ldots, x^{[0]}_{1}, \ldots, x^{[0]}_{m} )
    \\&\quad 
    = 
    \sum_{ 1 \leq k \leq m }
    x^{[l+1]}_{k}
    B^{[l+1]}_{m,k}( x^{[l]}_{1}, \ldots, x^{[l]}_{m-k+1}, \ldots, x^{[0]}_{1}, \ldots, x^{[0]}_{m-k+1} )
\end{align*}
Evidently, the $m$-th derivative of a composition of $l$ univariate functions is 
\begin{align*}
    &
    \nabla^{m} (f_{l+1} \circ f_{l} \circ \cdots \circ f_{0})
    _{|x_0}
    \\&\quad 
    =
    B^{[l+1]}_{m}\big(
\nabla^{ } f_{l+1|x_{l+1}}, \ldots, \nabla^{m} f_{l+1|x_{l+1}}, 
\nabla^{ } f_{l  |x_{l  }}, \ldots,
\nabla^{m} f_{1  |x_{1  }}, 
        \nabla^{ } f_{0  |x_{0  }}, \ldots, \nabla^{m} f_{0  |x_{0  }}  
\big)
    .
\end{align*}

\section{Pullbacks of Tensor Fields}\label{sec:pullbacks}

The geometrically conforming coordinate transformation of tensor fields leads to the notion of pullback. 
The pullback of (covariant) tensor fields along a coordinate transformation thus generalizes the composition of scalar-valued functions with a coordinate change. 
In this section, we review pullbacks of tensors and derive pointwise estimates. 
\\

Let $X$ and $Y$ be Banach spaces and let $\OmegaS \subseteq X$ and $\OmegaT \subseteq Y$ be open sets.
We suppose that $\phi : \OmegaS \rightarrow \OmegaT$ is a mapping differentiable at $x \in \OmegaS$.
Recall that a \emph{tensor field} over $\OmegaT$ is a mapping ${u} : \OmegaT \rightarrow \LIN^{d}(Y,\VEC)$, where $\VEC$ is another vector space.
The \emph{pullback} of such a tensor field ${u}$ at $x$ along the coordinate transformation $\phi$ is a member of $\LIN^{d}(X,\VEC)$,
defined by the identity 
\begin{align*}
    \phi^{\ast} {u}_{|x}\left( {v}_1, {v}_2, \ldots, {v}_d \right)
    = 
    {u}_{\phi(x)}\left(
        \nabla\phi_{|x} \cdot {v}_1, \nabla\phi_{|x} \cdot {v}_2, \ldots, \nabla\phi_{|x} \cdot {v}_d 
    \right)
    ,
    \quad 
    {v}_{1}, {v}_{2}, \ldots, {v}_{d} \in X
    .
\end{align*}
If the transformation $\phi$ is differentiable (almost everywhere) over $\OmegaS$,
then the pullback defines another $d$-tensor field $\phi^{\ast} {u} : \OmegaS \rightarrow \LIN^{d}(X,\VEC)$ (almost everywhere) over $\OmegaS$.

\begin{example}
    Real-valued functions can be interpreted as $0$-tensors. 
    When ${u} : \OmegaT \rightarrow \bbR$ is a function, 
    then the pullback along $\phi : \OmegaS \rightarrow \OmegaT$ is just the composition ${u} \circ \phi$. 
    In particular, the pullback of a scalar-valued function is well-defined even along nowhere differentiable functions. 
    But the situation is different for covariant $d$-tensors of higher valence $d$:
    the pullback of general $d$-tensors involves the Jacobian and thus requires a coordinate transformation that is differentiable at the given point. 
    
    Suppose that $X = Y = \bbR^{N}$ are finite-dimensional. 
    We can identify any vector field ${u} : \OmegaT \rightarrow \bbR^{N}$
    with a covariant $1$-tensor. The pullback along $\phi$ of that $1$-tensor corresponds to the vector field $\nabla\phi^{t} \cdot u_{|\phi}$ over the domain $\OmegaS$. 
    Similarly, we can identify any matrix field ${u} : \OmegaT \rightarrow \bbR^{N \times N}$
    with a covariant $2$-tensor. The pullback along $\phi$ of that $1$-tensor corresponds to the matrix field $\nabla\phi^{t} \cdot u_{|\phi} \cdot \nabla\phi$ over $\OmegaS$. 
\end{example}

We are interested in pointwise estimates for the derivative of pullbacks. 
Given directions $v_1,v_2,\ldots,v_m \in X$ 
we apply a higher-order Leibniz formula:
\begin{align*}
    &
    \nabla_{v_1,v_2,\ldots,v_m} 
\left( \phi^{\ast} {u} \right)_{|x}
    =
    \sum_{ {P} \in {\calP}_{0}(m,d) }
\nabla_{ {P}_{0} } ({u}\circ\phi)_{|x} 
\left( 
\nabla_{ {P}_{1} } \phi_{|x}
       ,
\nabla_{ {P}_{2} } \phi_{|x} 
       ,
\ldots 
       ,
\nabla_{ {P}_{d} } \phi_{|x} 
    \right)
    .
\end{align*}
For any ${P_{0}} \subseteq \indexset{1}{m}$, 
the higher-order chain rule implies 
\begin{align*}
    \nabla_{ {P_{0}} } \left( {u}\circ\phi \right)_{|x}
    =
    \sum_{ \substack{ 0 \leq k \leq |P_{0}| \\ \calC \in \calP(P_{0},k) } }
    \nabla^{k} {u}_{|\phi(x)}
\left( 
        \nabla_{ C_1 }\phi_{|x}
        ,
        \nabla_{ C_2 }\phi_{|x} 
        ,
        \ldots 
        ,
        \nabla_{ C_k }\phi_{|x}
    \right)
    .
\end{align*}
Note that here we sum up contractions in $k$ indices that yield $d$-linear mappings. 
Putting this together,
the $m$-th derivative of the pullback is 
\begin{align}\label{math:derivativeofpullback}
    \begin{aligned}
    &
    \nabla_{v_1,v_2,\ldots,v_m} 
    \left( \phi^{\ast} {u} \right)_{|x}
\\&\quad 
    =
    \sum_{ \substack{ {P} \in \calP_{0}(m,d) \\ 0 \leq k \leq |{P}_{0}| \\ \calC \in \calP({P}_{0},k) } }
    \nabla^{k} {u}_{|\phi(x)}
\left( 
\nabla_{ {C}_{1} } \phi_{|x}
       ,
       \ldots 
       ,
       \nabla_{ {C}_{k} } \phi_{|x}
       , 
       \nabla_{ {P}_{1} } \nabla\phi_{|x}
       ,
       \nabla_{ {P}_{2} } \nabla\phi_{|x} 
       ,
       \ldots 
       ,
       \nabla_{ {P}_{d} } \nabla\phi_{|x} 
    \right)
    .
    \end{aligned}
\end{align}
In the sum above, every partition set $\calC$ has been equipped with an arbitrary ordering
for the sake of readability, but the summands do not depend on that ordering.

\begin{remark}
    The situation of first derivatives of functions is special:
    The pullback of the first derivative of a function is the first derivative of the function's pullback.
    However, taking the pullback does not commute with taking derivatives in general. 
    For the example, the $m$-th derivative of the pullback of the $d$-th derivative of a function is generally not the $(m+d)$-th derivative of the pullback of the function. 
\end{remark}

Having derived a formula for the higher derivatives of a pullback, 
we want to estimate its spectral norm. 
For that purpose we introduce a number of definitions. 
We introduce the \emph{generalized Scherk coefficients}
\begin{align}\label{math:linglingindices}
    \Scherkcof^{b}_{h}
    :=
    \Scherkcof^{b_1,b_2,\ldots,b_m}_{h_0,h_1,h_2,\ldots,h_m}
    :=
    \frac{
        (\sum b_{i})!
        (\sum h_{j})!
    }{
        \prod_{i=1}^{m} b_{i}! (i!)^{b_i}
        \prod_{j=0}^{m} h_{j}! (j!)^{h_j}
    }
\end{align}
whenever $b=(b_1,b_2,\ldots,b_m)$ and $h=(h_0,h_1,h_2,\ldots,h_m)$ are sequences of non-negative integers. 
Note that in the case where the integers $h_0,h_1,h_2,\ldots,h_m$ are all zero, 
we recover the Scherk coefficient $\Scherkcof^b = \Scherkcof^{b_1,b_2,\ldots,b_m}_{}$ as a special case. 
We use the coefficients~\eqref{math:linglingindices} when summing 
over several constrained indices: we define the \emph{generalized Scherk indices} as 
\begin{align}\label{math:magesets} 
\Scherkset(k,m,d)
    :=
    \left\{
        (b,h)
    \suchthat*
        \begin{array}{l}
            b = (b_{1}, b_{2}, \ldots, b_{m}) \in \bbN_{0}^{m}
            ,
            \\
            h = (h_{0}, h_{1}, \ldots, h_{m}) \in \bbN_{0}^{m+1}
            , 
        \end{array}
        \;
        \sum_{i=1}^{m} b_{i} = k
        \;
        ,
        \sum_{j=0}^{m} h_{j} = d
        \;
        ,
        \sum_{i=1}^{m} i b_{i} + \sum_{j=0}^{m} j h_{j} = m
    \right\}
    .
\end{align}
These definitions allow for the following concise statement.

\begin{proposition}\label{prop:pullbackoftensors}
    Let $X$, $Y$ and $\VEC$ be Banach spaces, and let $\OmegaS \subseteq X$ and $\OmegaT \subseteq Y$ be open sets.
    Let $\phi : \OmegaS \rightarrow \OmegaT$ be a mapping $(m+1)$-times differentiable at $x \in \OmegaS$,
    and let ${u} : \OmegaT \rightarrow \LIN^{d}_{1}(Y,\VEC)$ be $m$-times differentiable at $\phi(x) \in \OmegaT$. 
    Then 
    \begin{align*}
        \left| 
            \nabla^{m} \left( \phi^{\ast} {u} \right)_{|x}
        \right|
        \leq 
\sum_{ \substack{ 
            0 \leq k \leq m 
            \\
            (b,h) \in \Scherkset(k,m,d)
        } }
        \left| \nabla^{k} {u}_{|\phi(x)} \right|
        \cdot 
        \Scherkcof^{b_1,b_2,\ldots,b_m}_{h_0,h_1,h_2,\ldots,h_m}
        \prod_{i=1}^{m} \left| \nabla^{i  } \phi_{|x} \right|^{b_i}
        \prod_{j=0}^{m} \left| \nabla^{j+1} \phi_{|x} \right|^{h_j}
        .
    \end{align*}
    If $d=0$, then the statement is also true when 
    $\phi : \OmegaS \rightarrow \OmegaT$ is $(m+1)$-times differentiable at $x \in \OmegaS$. 
\end{proposition}

\begin{proof}
    The proof involves a number of combinatorial observations. 
    We apply the Leibniz rule to $\nabla^{m}\left( \phi^{\ast} {u} \right)$:
    there are ${m \choose n}$ ways of distributing $0 \leq n \leq m$ derivatives onto the first factor ${u}_{|\phi}$ 
    and $m-n$ derivatives onto the second factor $\otimes^{d} \nabla\phi$. 
    For the $n$-th derivative of the first factor, $0 \leq k \leq n$, we calculate 
    \begin{align*}
        \left|
        \nabla^{n} {u}_{|\phi(x)}
        \right|
        \leq 
        \sum_{ \substack{ 
            0 \leq k \leq n \\ b_{1} + b_{2} + \cdots + b_{n} = k \\ b_{1} + 2 b_{2} + \cdots + n b_{n} = n 
        } }
        \left| \nabla^{k} {u}_{|\phi(x)} \right|
        \cdot 
        \frac{
            n!
        }{
            b_{1}! b_{2}! \cdots b_{n}! (1!)^{b_1} (2!)^{b_2} \cdots (n!)^{b_n} 
        }
        \prod_{i=1}^{n} \left| \nabla^{i} \phi_{|x} \right|^{b_i}
        .
    \end{align*}
    This is checked for $n=0$ and shown for $n > 0$ 
    by an argument analogous to the proof of the Fa\`{a}~di~Bruno formula. 
    We would like to handle the second factor in a similar manner. 
    When taking the derivative of order $m-n$ of the second factor, 
    we distribute $m-n$ derivatives onto $d$ different copies of $\nabla\phi_{|x}$. 
    For each such distribution, let $h_{0}, h_{1}, h_{2}, \ldots$ be the number of factors receiving $0,1,2,\ldots$ derivatives, respectively.
    The number of distributions giving the same such sequence of numbers is
    \begin{align*}
        \frac{d!}{h_{0}! h_{1}! \cdots h_{m-n}! }
        \times 
        \frac{(m-n)!}{ (0!)^{0} (1!)^{1} \cdots (m-n)!^{m-n} }
        .
    \end{align*}
    Putting these observations together,
    the norm of 
    $\nabla^{m} \left( \phi^{\ast} {u} \right)_{|x}$
    has the upper bound 
    \begin{align*}
        \sum_{ \substack{ 
            0 \leq k \leq n \leq m 
            \\
            b_{1}, b_{2}, \ldots, b_{n} \in \bbN_{0}
            \\
            b_{1} + b_{2} + \cdots + b_{n} = k
            \\
            b_{1} + 2 b_{2} + \cdots + n b_{n} = n 
            \\
            h_{0}, h_{1}, h_{2}, \ldots, h_{m-n} \in \bbN_{0}
            \\
            h_{0} + h_{1} + h_{2} + \cdots + h_{m-n} = d
            \\
            1 h_{1} + 2 h_{2} + \cdots + (m-n) h_{m-n} = m-n } }
        \left| \nabla^{k} {u}_{|\phi(x)} \right|
        \cdot 
        { m \choose n }
        \frac{
            n!
        }{
            \prod_{i=1}^{n} b_{i}! (i!)^{b_i}
        }
        \frac{
            d! (m-n)!
        }{
            \prod_{j=0}^{m-n} h_{j}! (j!)^{h_j}
        }
\prod_{i=1}^{n  } \left| \nabla^{i} \phi_{|x} \right|^{b_i}
        \prod_{j=0}^{m-n} \left| \nabla^{j} \nabla \phi_{|x} \right|^{h_j}
        .
    \end{align*}
    We reduce the fractions further,
    and slightly rewriting the index set over which we sum, 
    we find 
    \begin{align*}
        \sum_{ \substack{ 
            0 \leq k \leq n \leq m 
            \\
            b_{1}, b_{2}, \ldots, b_{m} \in \bbN_{0}
            \\
            h_{0}, h_{1}, \ldots, h_{m} \in \bbN_{0}
            \\
            b_{1} + b_{2} + \cdots + b_{m} = k
            \\
            h_{0} + h_{1} + \cdots + h_{m} = d
            \\
            b_{1} + 2 b_{2} + \cdots + m b_{m} = n 
            \\
            h_{1} + 2 h_{2} + \cdots + m h_{m} = m-n
} }
        \left| \nabla^{k} {u}_{|\phi(x)} \right|
        \cdot 
        \frac{
            m!
            d!
        }{
            \prod_{i=1}^{m} b_{i}! (i!)^{b_i}
            \prod_{j=0}^{m} h_{j}! (j!)^{h_j}
        }
        \prod_{i=1}^{m} \left| \nabla^{i} \phi_{|x} \right|^{b_i}
        \prod_{j=0}^{m} \left| \nabla^{j} \nabla \phi_{|x} \right|^{h_j}
        .
    \end{align*}
    The desired statement follows, using Definition~\eqref{math:linglingindices}. 
\end{proof}

\begin{remark}
    This result showcases a difference between scalar and tensor fields 
    that we will repeatedly highlight in forthcoming sections in different forms:
    the pullback of a function with $m$ derivatives has $m$ derivatives again 
    if the transformation has $m$ derivatives. 
    Yet the analogous statement for tensor fields is only true when the transformation has $m+1$ derivatives. 
    This is natural because the pullback of tensor fields involves the Jacobian of the transformation. 
\end{remark}

It is of interest to condense the norm inequality in Proposition~\ref{prop:pullbackoftensors} even further:
the derivatives of the transformation may appear in both products on the right-hand side of the inequality. 
For that purpose we introduce the index sets  
\begin{align}\label{math:magesets:summarized}
    \hat \Scherkset(k,m,d)
    :=
    \left\{\;
p = (p_{0}, p_{1}, \ldots, p_{m}) \in \bbN_{0}^{m+1}
\suchthat*
        p = (0,b) + h, 
        \;
        (b,h) \in \Scherkset(k,m,d)
    \;\right\}
\end{align}
and corresponding coefficients 
\begin{align}\label{math:linglingindices:summarized}
    \hat \Scherkcof_{p_0,p_1,\ldots,p_m}
    :=
    \sum_{ \substack{ (b,h) \in \Scherkset(k,m,d) \\ p = b + d } }
    \frac{
        (\sum b_{i})!
        (\sum h_{j})!
    }{
        \prod_{i=1}^{m} b_{i}! (i!)^{b_i}
        \prod_{j=0}^{m} h_{j}! (j!)^{h_j}
    }
    .
\end{align}
We note that $\hat \Scherkset(k,m,0)$ is just obtained by padding a zero in front of every $\Scherkset(k,m)$.
The following simple corollary of Proposition~\ref{prop:pullbackoftensors} is clear.

\begin{corollary}\label{corollary:pullbackoftensors:condensed}
    Under the assumptions of Proposition~\ref{prop:pullbackoftensors},
\begin{align*}
        \left| 
            \nabla^{m} \left( \phi^{\ast} {u} \right)_{|x}
        \right|
        \leq 
        \sum_{ \substack{ 
            0 \leq k \leq m 
            \\
            p \in \hat \Scherkset(k,m,d)
        } }
        \left| \nabla^{k} {u}_{|\phi(x)} \right|
        \cdot 
        \Scherkcof_{p_0,p_1,\ldots,p_m}
        \prod_{j=0}^{m+1} \left| \nabla^{j} \phi_{|x} \right|^{p_j}
        .
    \end{align*}
\end{corollary}

We rewrite the pointwise inequality for the higher-order derivatives. 
Earlier, we have used Bell polynomials for a concise exact representation of the derivatives of composed univariate functions. 
Now concisely inequalities using generalized Bell polynomials. 
We set 
\begin{align}\label{math:bellpolynomial:dgeneralized}
    \begin{aligned}
        &
        B_{k,m,d}\left(
            x^{}_{1  },
            x^{}_{2  },
            \ldots,
            x^{}_{m+1},
        \right)
        \\&\quad 
:=
        \sum_{ \substack{ 
            (b,d) \in \Scherkset(k,m,d)
        } }
        \Scherkcof^{b}_{d}
        \prod_{i=1}^{m} \left( x^{}_{i  } \right)^{b_i}
        \prod_{j=0}^{m} \left( x^{}_{j+1} \right)^{d_j}
=
        \sum_{ \substack{ 
            p \in \hat \Scherkset(k,m,d)
        } }
        \hat \Scherkcof_{p}
        \prod_{j=0}^{m+1} \left( x^{}_{j} \right)^{p_j}
        .
    \end{aligned}
\end{align}
With that, we once more rewrite the pointwise estimate for the higher derivatives of the pullback.

\begin{corollary}
    Under the assumptions of Proposition~\ref{prop:pullbackoftensors},
    \begin{align}\label{math:pointwisepullbackinequality}
        \left| 
            \nabla^{m} \left( \phi^{\ast} {u} \right)_{|x}
        \right|
        \leq 
\sum_{k=0}^{m}
        | \nabla^{k} {u}_{|\phi(x)} |
        \cdot 
        B_{k,m,d}\left(
            | \nabla^{} \phi_{|x} |, | \nabla^{2} \phi_{|x} |, \ldots, | \nabla^{m+1} \phi_{|x} |
        \right)
        .
    \end{align}
\end{corollary}
    
Considerable simplifications apply when $u$ is a zero tensor,
for example, in the special case where $u$ is a real-valued function. 
Either specializing Proposition~\ref{prop:pullbackoftensors} or using a straightforward consequence of~\eqref{math:faadibruno:bellpolynomial},
we get an inequality in terms of Bell polynomials. 
Note that $B_{m,k,0}$ only depends on the first $m$ derivatives of $\phi$,
and we can thus identify $B_{m,k,0} = B_{m,k}$ with some minor abuse of notation.

\begin{corollary}\label{corollary:pullbackoftensors:zerotensors}
    Let $X$, $Y$ and $\VEC$ be Banach spaces, and let $\OmegaS \subseteq X$ and $\OmegaT \subseteq Y$ be open sets.
    Let $\phi : \OmegaS \rightarrow \OmegaT$ be a mapping $m$-times differentiable at $x \in \OmegaS$,
    and let ${u} : \OmegaT \rightarrow \VEC$,
    be $m$-times differentiable at $\phi(x) \in \OmegaT$. 
    Then 
    \begin{align*}
        \left| 
            \nabla^{m} \left( {u} \circ \phi \right)_{|x}
        \right|
        &
        \leq 
        \sum_{ \substack{ 
            1 \leq k \leq m 
            \\
            b \in \Scherkset(k,m)
        } }
        \left| \nabla^{k} {u}_{|\phi(x)} \right|
        \cdot 
        \Scherkcof^{b_1,b_2,\ldots,b_m}
        \prod_{i=1}^{m} \left| \nabla^{i  } \phi_{|x} \right|^{b_i}
\\&
        =
        \sum_{k=1}^{m}
        | \nabla^{k} u_{|\phi(x)} |
        \cdot 
        B_{m,k}\left( | \nabla^{1} \phi_{|x} |, | \nabla^{2} \phi_{|x} |, \ldots, | \nabla^{m-k+1} \phi_{|x} | \right) 
        \\&
        =
        B_{m}\left(
            | \nabla^{1} u_{|\phi(x)} |, | \nabla^{2} u_{|\phi(x)} |, \ldots, | \nabla^{m} u_{|\phi(x)} |,
            | \nabla^{1} \phi_{|x}    |, | \nabla^{2} \phi_{|x}    |, \ldots, | \nabla^{m} \phi_{|x}    | 
        \right) 
        .
    \end{align*}
\end{corollary}

\section{Sobolev-Slobodeckij Spaces and Estimates}\label{sec:sobolev}

In this section, we discuss Sobolev spaces and Sobolev-Slobodeckij spaces. 
One of our main objectives in this exposition is estimating the Sobolev-Slobodeckij norm of tensor fields under coordinate transformations. 
That involves higher-order chain rules for the pullback of tensor fields.
\\

We specialize to a finite-dimensional geometric ambient. 
Let $X = Y = \bbR^{N}$. 
Let $\OmegaS \subseteq \bbR^{n}$ be open and let $\VEC$ be a Banach space. 
We emphasize that we do not necessarily assume that $\OmegaS$ is bounded or has a regular boundary. 

We begin with reviewing a few classical function spaces. 
By $\Cont^{\infty}_{c}(\OmegaS,\VEC)$ we denote the vector space of smooth functions over $\OmegaS$ with compact support that map into the Banach space $\VEC$. 
We let $\Cont^{m}(\OmegaS,\VEC)$ be the vector space of continuous functions from $\OmegaS$ into $\OmegaT$
that have continuous derivatives up to degree $m$.
Note that this vector space as such is not normed. 
However, the members of $\Cont^{m}(\OmegaS,\VEC)$ for which the \emph{uniform norm}
\begin{align}\label{math:uniformnorm}
    \left\| u \right\|_{\Cont^{m}(\OmegaS,\VEC)}
    :=
    \max_{0 \leq k \leq m}
    \sup_{ x \in \OmegaS }
    \left| \nabla^{k} u_{|x} \right|_{\VEC}
\end{align}
is finite constitute a Banach space. 
At this point we remark that we occasionally drop the codomain from the notation.

Suppose that $\dist(\cdot,\cdot)$ is a metric on $\Omega$. 
When $\theta \in (0,1]$, 
the \emph{H\"older space} $\Cont^{0,\theta}(\OmegaS,\VEC)$ consists of those $u \in \Cont^{0}(\OmegaS,\VEC)$ for which 
there exists a constant $C \geq 0$ such that
\begin{align}\label{math:hoelderspacedefinition}
    | u(x) - u(x') |_{\VEC}
    \leq
    C
    \dist(x,x')^{\theta},
    \quad 
    x,x' \in \OmegaS,
\end{align}
and the minimum of such constants is written $|u|_{\Cont^{0,\theta}(\OmegaS,\VEC)}$.
This quantity defines a seminorm on $\Cont^{0,\theta}(\OmegaS,\VEC)$.
The functions in $\Cont^{0,1}(\OmegaS,\VEC)$ are also called \emph{Lipschitz}.
Every member of $\Cont^{0,1}(\OmegaS,\VEC)$ is differentiable almost everywhere in $\OmegaS$
according to Rademacher's theorem.
More generally, we write $\Cont^{m,\theta}(\OmegaS,\VEC)$ for the subspace of $\Cont^{m}(\OmegaS,\VEC)$ 
whose members have continuous derivatives up to degree $m$ in the H\"older space $\Cont^{0,\theta}(\OmegaS,\LIN^m(X,\VEC))$.
Similarly as above, the members of $\Cont^{m,\theta}(\OmegaS,\VEC)$ for which the norm 
\begin{align}\label{math:higherhoeldernorm}
    \left\| u \right\|_{\Cont^{m,\theta}(\OmegaS,\VEC)}
    :=
    \max_{0 \leq k \leq m}
    \sup_{ x \in \OmegaS }
    \left| \nabla^{k} u_{|x} \right|_{\VEC}
+
    \left| \nabla^{k} u \right|_{\Cont^{0,\theta}(\OmegaS,\LIN^{d}(X,\VEC))}
\end{align}
is finite constitute a Banach space. 
\\

Before we continue, we prove a result on classical function spaces that may be interpreted
as an inverse mapping theorem for functions satisfying a H\"older or Lipschitz condition. 
Using the recursive construction used in some earlier examples,
we derive estimates for the H\"older seminorms of an invertible differentiable mapping.

\begin{proposition}\label{proposition:inversefunctiontheorem:hoelder}
    Let $m \geq 2$.
    Suppose that $\OmegaS, \OmegaT \subseteq \bbR^{N}$ are open sets
    and  that $\phi \in \Cont^{m}(\OmegaS,\OmegaT)$ has an inverse $\phi^{-1} \in \Cont^{m}(\OmegaT,\OmegaS)$.
    We then have 
    \begin{align*}
        &
        \left| \nabla^{m}\phi^{-1} \right|_{\Cont^{0,\theta}(\OmegaT)}
        \\&\quad
        \leq 
        \left\| \phi^{-1} \right\|_{\Cont^{0,\theta}(\OmegaT)}
        \sum_{ \substack{ 2 \leq k \leq m \\ \calP \in \calP(m,k) } }
        (k+1)
        \left\| \nabla^{k} \phi \right\|_{\Cont^{0,\theta}(\OmegaS)}
        B_{m,k}
        \left(
            \left\| \nabla^{1}\phi^{-1} \right\|_{\Cont^{0,\theta}(\OmegaT)}
            ,
            \ldots
            ,
            \left\| \nabla^{m-k+1}\phi^{-1} \right\|_{\Cont^{0,\theta}(\OmegaT)}
        \right)
        \\
        &\quad \quad 
        + 
        \left\| \nabla \phi^{-1} \right\|_{\Cont^{0}(\OmegaT)}
        \left| \phi^{-1} \right|_{\Cont^{0,1}(\OmegaT)}^{\theta} 
        \sum_{ \substack{ 2 \leq k \leq m \\ \calP \in \calP(m,k) } }
        \left| \nabla^{k} \phi \right|_{\Cont^{0,\theta}(\OmegaS)}
        B_{m,k}
        \left(
            \left| \nabla^{1}\phi^{-1} \right|_{\Cont^{0}(\OmegaT)}
            ,
            \ldots
            ,
            \left| \nabla^{m-k+1}\phi^{-1} \right|_{\Cont^{0}(\OmegaT)}
        \right)
        ,
    \end{align*}
    provided that the right-hand side is finite. 
    In particular, $\phi^{-1} \in \Cont^{m,\theta}(\OmegaT,\OmegaS)$ 
    if
    \begin{gather*}
        \left\| \phi \right\|_{\Cont^{m,\theta}(\OmegaS,\OmegaT)}
        +
        \left\| \nabla\phi^{-1} \right\|_{\Cont^{0,\theta}(\OmegaT)} +
        \left| \phi^{-1} \right|_{\Cont^{0,1}(\OmegaT,\OmegaS)} <
        \infty.
    \end{gather*}
\end{proposition}

\begin{proof}
    Let $v_1,\ldots,v_m \in Y$. 
    Recalling Example~\ref{example:inversefunction:multivariate},
    the assumptions imply that $\phi^{-1} \in \Cont^{m}(\OmegaT,\OmegaS)$ and that the $m$-th derivative of the inverse function satisfies 
    \begin{align*}
        \nabla^{m}\phi^{-1}_{|y}(v_1,\ldots,v_m)
        =
        -
        \nabla \phi^{-1}_{|y}
        \sum_{ \substack{ 2 \leq k \leq m \\ \calP \in \calP(m,k) } }
        \nabla^{k} \phi_{|\phi^{-1}(y)}
        \left( 
            \nabla^{ {P}_1 }\phi^{-1}_{|y}, 
            \nabla^{ {P}_2 }\phi^{-1}_{|y}, 
            \ldots, 
            \nabla^{ {P}_k }\phi^{-1}_{|y}
        \right)
        .
    \end{align*} 
    Let $y, y' \in \OmegaT$. 
    We apply a telescope sum and find 
    \begin{align*}
        &
        -
        \nabla^{m}\phi^{-1}_{|y }(v_1,\ldots,v_m)
        +
        \nabla^{m}\phi^{-1}_{|y'}(v_1,\ldots,v_m)
        \\&\quad
        =
        \left( 
            \nabla \phi^{-1}_{|y} - \nabla \phi^{-1}_{|y'}
        \right)
        \sum_{ \substack{ 2 \leq k \leq m \\ \calP \in \calP(m,k) } }
        \nabla^{k} \phi_{|\phi^{-1}(y)}
        \left( 
            \nabla^{ {P}_1 }\phi^{-1}_{|y}, 
            \nabla^{ {P}_2 }\phi^{-1}_{|y}, 
            \ldots, 
            \nabla^{ {P}_k }\phi^{-1}_{|y}
        \right)
        \\&\quad\quad\quad 
        +
        \nabla \phi^{-1}_{|y'}
        \sum_{ \substack{ 2 \leq k \leq m \\ \calP \in \calP(m,k) } }
        \left( 
            \nabla^{k} \phi_{|\phi^{-1}(y)}
            -
            \nabla^{k} \phi_{|\phi^{-1}(y')}
        \right)
        \left( 
            \nabla^{ {P}_1 }\phi^{-1}_{|y}, 
            \nabla^{ {P}_2 }\phi^{-1}_{|y}, 
            \ldots, 
            \nabla^{ {P}_k }\phi^{-1}_{|y}
        \right)
        \\&\quad\quad\quad 
        +
        \nabla \phi^{-1}_{|y'}
        \sum_{ \substack{ 2 \leq k \leq m \\ \calP \in \calP(m,k) } }
        \nabla^{k} \phi_{|\phi^{-1}(y')}
        \left( 
            \nabla^{ {P}_1 }\phi^{-1}_{|y} - \nabla^{|{P}_1|}\phi^{-1}_{|y'}, 
            \nabla^{ {P}_2 }\phi^{-1}_{|y}, 
            \ldots, 
            \nabla^{ {P}_k }\phi^{-1}_{|y}
        \right)
        \\&\quad\quad\quad\quad\quad 
        \vdots 
        \\&\quad\quad\quad 
        +
        \nabla \phi^{-1}_{|y'}
        \sum_{ \substack{ 2 \leq k \leq m \\ \calP \in \calP(m,k) } }
        \nabla^{k} \phi_{|\phi^{-1}(y')}
        \left( 
            \nabla^{ {P}_1 }\phi^{-1}_{|y'}, 
            \nabla^{ {P}_2 }\phi^{-1}_{|y'}, 
            \ldots, 
            \nabla^{ {P}_k }\phi^{-1}_{|y} - \nabla^{|{P}_k|}\phi^{-1}_{|y'}
        \right)
        .
    \end{align*} 
    Note furthermore that 
    \begin{align*}
        \frac{ 
            \nabla^{k} \phi_{|\phi^{-1}(y)}
            - 
            \nabla^{k} \phi_{|\phi^{-1}(y')}
        }{
            \dist(y,y')^{\theta}
        }
=
\dfrac{
            \nabla^{k} \phi_{|\phi^{-1}(y)}
            - 
            \nabla^{k} \phi_{|\phi^{-1}(y')}
        }{
            \dist(\phi^{-1}(y),\phi^{-1}(y'))^{\theta}
        }
        \cdot 
        \frac{
            \dist(\phi^{-1}(y),\phi^{-1}(y'))^{\theta}
        }{
            \dist(y,y')^{\theta}
        }
        .
    \end{align*}
    We use this last observation, 
    the triangle inequality applied to the telescope sum above,
    and arguments like the ones in Section~\ref{sec:chainrule}.
    This completes the proof. 
\end{proof}

\begin{remark}
    An easy consequence of Proposition~\ref{proposition:inversefunctiontheorem:hoelder} resembles the inverse mapping theorem.
    Suppose that $\phi \in \Cont^{m,\theta}(\OmegaS,\OmegaT)$ is invertible at $x \in \OmegaS$ and has an invertible Jacobian there. 
    If $m \geq 2$, then then this Proposition ensures that there exists a neighborhood $U$ of $x$ such that the restriction $\phi_{|U}$ is invertible 
    with inverse $\phi^{-1} \in \Cont^{m,\theta}(\phi(U),\OmegaS)$. 

    It is instructive to consider the interesting special case when $\OmegaS$ and $\OmegaT$ are bounded with, say, smooth boundary. 
    Since the domain is bounded and has regular boundary,  
    $\left\| \phi \right\|_{\Cont^{m}(\OmegaS)}$ and $\left\| \nabla^m \phi \right\|_{\Cont^{0,\theta}(\OmegaS)}$ being finite
    implies that $\left\| \phi \right\|_{\Cont^{m,\theta}(\OmegaS)}$ is finite. 
    
    Now, if $\phi$ is invertible and the Jacobian $\nabla\phi$ has bounded inverse,
    then the calculations in Example~\ref{example:inversefunction:multivariate} show that 
    $\left\| \phi^{-1} \right\|_{\Cont^{m}(\OmegaT)}$ is finite. 
    The regularity of the boundary and the domain's bounded diameter 
    then imply that $\left| \phi^{-1} \right|_{\Cont^{0,1}(\OmegaT)}$
    and $\left| \nabla\phi^{-1} \right|_{\Cont^{0,\theta}(\OmegaT)}$ are finite. 
It is now Proposition~\ref{proposition:inversefunctiontheorem:hoelder} that shows 
    $\phi^{-1} \in \Cont^{m,\theta}(\OmegaT,\OmegaS)$ and constructs an explicit upper bound for 
    $\left\| \phi^{-1} \right\|_{\Cont^{m,\theta}(\OmegaT)}$.
    If we resolve the recursive estimates used here, 
    then this last upper bound is only in terms of the $\Cont^{m,\theta}$-norm of $\phi$ and an upper bound for the inverse of its Jacobian.
    
    Applying this result requires $m \geq 2$. 
    Proposition~\ref{proposition:inversefunctiontheorem:hoelder} does not establish that mere $\Cont^{1,\theta}$-regularity of $\phi$
    implies the same for its inverse. However, a local version of such a result is easily seen:
    if $\phi \in \Cont^{1,\theta}(\OmegaS,\OmegaT)$, 
    then \[
        \nabla \phi^{-1}_{|y}
        =
        (\nabla \phi_{|\phi^{-1}(y)})^{-1}
        =
        \det( \nabla \phi_{|\phi^{-1}(y)} )^{-1}
        \adj( \nabla \phi_{|\phi^{-1}(y)} ) 
\] for all $y \in \OmegaT$. Here, we use Cramer's rule. 
    We use the last identity in the previous together with the fact that matrix determinant and adjunct matrix is differentiable over the open set of regular matrices 
    to see that $\phi^{-1}$ is $\Cont^{1,\theta}$-regular in a neighborhood of $y$.
\end{remark}

We now review Lebesgue and Sobolev-Slobodeckij spaces. 
We write $\Lebesgue^{p}(\OmegaS,\VEC)$ for the Lebesgue space of $\VEC$-valued fields over $\OmegaS$ to the integrability exponent $p \in [1,\infty]$. 
This is a Banach space equipped with the norm ${\| \cdot \|}_{\Lebesgue^{p}(\OmegaS)}$, defined as follows: 
\begin{align}
    {\| {u} \|}_{\Lebesgue^{\infty}(\OmegaS,\VEC)}
    &:=
    \esssup_{ x \in \OmegaS} |u(x)|_{\VEC}
    ,
    \\
    {\| {u} \|}_{\Lebesgue^{p}(\OmegaS,\VEC)}
    &:=
    \left( \int_{\OmegaS} |u(x)|^{p}_{\VEC} \,\dif x \right)^{\frac 1 p},
    \quad 
    p \in [1,\infty)
    .
\end{align}
In what follows, we may drop the notation of the normed space $\VEC$ when there is no danger of ambiguity.

For any $\theta \in (0,1)$ and $p \in [1,\infty]$ we define the Sobolev-Slobodeckij seminorm 
\begin{align}\label{math:slobodeckijnormseminorm}
    |u|_{\Sobolev^{\theta,p}(\OmegaS,\VEC)}
    =
    \left\| 
    | u(x) - u(y) |_{\VEC}
\dist( x, y )^{-\theta - \frac N p}
    \right\|_{\Lebesgue^{p}(\OmegaS\times\OmegaS)}
    , \quad 
    {u} \in \Lebesgue^{p}(\OmegaS,\VEC)
    .
\end{align}
The subspace of $\Lebesgue^{p}(\OmegaS,\VEC)$ over which~\eqref{math:slobodeckijnormseminorm} is finite constitutes a Banach space. 

We discuss weak derivatives.~\footnote{Throughout this manuscript, all integrals are to be understood in the sense of Bochner integrals.} 
Let $\Lebesgue^{1}_{loc}(\OmegaS,\VEC)$ be the vector space of measurable functions that are absolutely integrable over compact subsets of $\OmegaS$. 
We say that $u \in \Lebesgue^{1}_{loc}(\OmegaS,\VEC)$ has a weak derivative $w \in \Lebesgue^{1}_{loc}(\OmegaS,\LIN^m(X,\VEC))$ of order $m$ if for all compactly supported smooth functions $\test \in \Cont^\infty_{c}(\OmegaS,\VEC^\ast)$ that take values in the dual space of $\VEC$ we have the identity 
\begin{align}
    \int_\OmegaS 
    \langle u_{|x}, \nabla_{v_1,\ldots,v_m} \test_{|x} \rangle 
    \,\dif{x}
    =
    -
    \int_\OmegaS 
    \langle w_{|x}( v_1,\ldots,v_m ), \test_{|x} \rangle 
    \,\dif{x}
    ,
    \quad 
    v_1,\ldots,v_m \in X
    .
\end{align}
We then write $\nabla^m u := w$ for this weak derivative. 
We write $\nabla_{v_1,\ldots,v_m} u := \nabla^m u( v_1,\ldots,v_m )$ for any $v_1,\ldots,v_m \in X$.

We write $\Sobolev^{m,p}(\OmegaS,\VEC)$ for the \emph{Sobolev space} of measurable $\VEC$-valued fields over $\OmegaS$ 
for which all distributional derivatives up to order $m$ are in $\Lebesgue^{p}(\OmegaS,\VEC)$,
where $m \in \mathbb N_0$ and $p \in [1,\infty]$.
For every $u \in \Sobolev^{m,p}(\OmegaS,\VEC)$ we define 
\begin{align}\label{math:sobolevnormen}
    \| {u} \|_{\Sobolev^{m,p}(\OmegaS,\VEC)} 
    :=
    \sum_{ k = 0 }^{m}
    \| \nabla^{k} {u} \|_{\Lebesgue^{p}(\OmegaS)}
    ,
    \quad 
    | {u} |_{\Sobolev^{m,p}(\OmegaS,\VEC)} 
    :=
    \| \nabla^{m} {u} \|_{\Lebesgue^{p}(\OmegaS)}
    ,
\end{align}
which are a norm and a seminorm over the Sobolev space. 
The Sobolev spaces together with these norms are Banach spaces.

When $m \in \bbN_{0}$ and $\theta \in (0,1)$,
then the \emph{Sobolev-Slobodeckij} space $\Sobolev^{m+\theta,p}(\OmegaS,\VEC)$ 
is the subspace of functions in $\Sobolev^{m,p}(\OmegaS,\VEC)$ 
whose weak derivatives up to order $m$ have bounded Sobolev-Slobodeckij seminorms $|\cdot|_{\Sobolev^{\theta,p}(\OmegaS,\VEC)}$. 
This is a Banach space with respect to the norm 
\begin{align}
\| {u} \|_{\Sobolev^{m+\theta,p}(\OmegaS,\VEC)} 
    :=
    \sum_{ k = 0 }^{m}
    \| \nabla^{k} {u} \|_{\Lebesgue^{p}(\OmegaS)}
    + 
    | \nabla^{k} {u} |_{\Sobolev^{\theta,p}(\OmegaS)},
    \quad 
    u \in \Sobolev^{m+\theta,p}(\OmegaS,\VEC)
    . 
\end{align}
Sobolev-Slobodeckij spaces generalize H\"older spaces. 
We remark that without further assumptions on the domain, 
not all weakly differentiable functions have finite Sobolev-Slobodeckij norms~\footnote{It seems that ``\emph{Sobolev-Slobodeckij space}'' is more common than the alphabetical ``\emph{Slobodeckij-Sobolev space}''.}. 
\\

We prove an auxiliary estimate that addresses the transformation of functions in Lebesgue spaces and lowest-order Sobolev-Slobodeckij spaces along sufficiently regular coordinate transformations. 
We only need to recall standard arguments and state this for completeness of the exposition. 

\begin{lemma}\label{lemma:transformationssatz:lp}
Let $\OmegaS, \OmegaT \subseteq \bbR^{N}$ be open sets, and let $\VEC$ be a Banach space.
    Suppose that $\phi : \OmegaS \rightarrow \OmegaT$ is continuous and has a locally Lipschitz inverse $\phi^{-1} : \OmegaT \rightarrow \OmegaS$. 
    If  $p \in [1,\infty]$ and $u \in \Lebesgue^{p}(\OmegaT,\VEC)$,
    then $u \circ \phi \in \Lebesgue^{p}(\OmegaS,\VEC)$ with 
    \begin{align}\label{math:transformationssatz:lp:scalar}
        \left\| 
            u_{|\phi}
        \right\|_{\Lebesgue^{p}(\OmegaS,\VEC)}
        \leq 
        \left\| 
            {u} 
        \right\|_{\Lebesgue^{p}(\OmegaT,\VEC)}
        \left\| 
            \det \nabla \phi^{-1}
        \right\|_{\Lebesgue^{\infty}(\OmegaT)}^{\frac 1 p}
        .
    \end{align}
    Suppose in addition that $\phi$ is locally Lipschitz. 
    If  $p \in [1,\infty]$ and $u \in \Lebesgue^{p}(\OmegaT,\LIN^{d}(Y,\VEC))$,
    then $\phi^\ast u \in \Lebesgue^{p}(\OmegaS,\LIN^{d}(X,\VEC))$ with 
    \begin{align}\label{math:transformationssatz:lp:tensor}
        \left\| 
            \phi^\ast u
        \right\|_{\Lebesgue^{p}(\OmegaS,\VEC)}
        \leq 
        \left\| 
            {u} 
        \right\|_{\Lebesgue^{p}(\OmegaT,\VEC)}
        \left\| 
            \det \nabla \phi^{-1}
        \right\|_{\Lebesgue^{\infty}(\OmegaT)}^{\frac 1 p}
        \left\| 
            \nabla \phi
        \right\|_{\Lebesgue^{\infty}(\OmegaS )}^{d}
        .
    \end{align}
    
\end{lemma}

\begin{proof}
    Note that $\phi^{-1}$ maps sets of measure zero to sets of measure zero. 
    Inequality~\eqref{math:transformationssatz:lp:scalar} is evident in the case $p = \infty$. 
    In the case $p < \infty$, we recall that the coordinates of $\phi^{-1}$ have weak derivatives that are locally essentially bounded, due to Rademacher's theorem. 
    We recall the integral identity 
    \begin{align*}
        \int_{\OmegaS}
        g(\phi(x))
        \,\dif x
        =
        \int_{\OmegaT}
        g(y)  |\det\nabla\phi^{-1}_{|y}|
        \,\dif y,
        \quad 
        g \in \Lebesgue^1(\OmegaT,\bbR)
        .
    \end{align*}
    Whenever $p < \infty$ and $u \in \Lebesgue^{p}(\OmegaT,\VEC)$,
    then applying this identity to $g = |u|_{\VEC}^{p} \in \Lebesgue^1(\OmegaT,\VEC)$ gives~\eqref{math:transformationssatz:lp:scalar}.
    
    For the second part of the proof,
    we suppose that $u \in \Lebesgue^{p}(\OmegaT,\LIN^{d}(Y,\VEC))$. 
    Assuming that $\phi$ is locally Lipschitz, 
    the coordinates of $\phi$ have weak derivatives that are locally essentially bounded.
    We have $\phi^\ast u_{|x} = u_{|\phi(x)} \contract {\otimes}^d \nabla\phi_{|x}$ for almost all $x \in \OmegaS$. 
    It follows that almost everywhere, 
    \begin{align}\label{math:auxiliary:pointwise}
        |\phi^\ast u_{|x}| \leq ( |u| \circ \phi )(x) \cdot |\nabla\phi_{|x}|^d
        .
    \end{align}
    Applying~\eqref{math:transformationssatz:lp:scalar} to the function $|u| \in \Lebesgue^{p}(\OmegaT,\bbR)$, we get~\eqref{math:transformationssatz:lp:tensor}.
    The proof is complete. 
\end{proof}

We are in a position to describe the higher-order chain rule for the weak derivatives of functions in Sobolev spaces.

\begin{proposition}[Weak Derivative of Pullback]\label{prop:weakderivativesofpullback}
Let $\OmegaS, \OmegaT \subseteq \bbR^{N}$ be open sets, and let $\VEC$ be a Banach space.
Let $m \in \bbN$. 
    Suppose that $\phi : \OmegaS \rightarrow \OmegaT$ is locally bi-Lipschitz. 
Suppose that the weak derivatives $\nabla\phi, \dots, \nabla^{m+1}\phi$ are essentially bounded. 
    If $u \in \Lebesgue^{1}_{loc}(\OmegaT,\LIN^d(Y,\VEC))$ has weak derivatives up to order $m$, 
    then $\phi^\ast u \in \Lebesgue^{1}_{loc}(\OmegaS,\LIN^d(X,\VEC))$ has weak derivatives up to order $m$,
    and almost everywhere 
    \begin{align}\label{math:weakderivativeofpullback}
        \begin{aligned}
            & 
            \nabla_{v_1,v_2,\ldots,v_m} 
            \left( \phi^{\ast} {u} \right)_{|x}
            \\&\quad 
            =
            \sum_{ \substack{ {P} \in \calP_{0}(m,d) \\ 0 \leq k \leq |{P}_{0}| \\ \calC \in \calP({P}_{0},k) } }
            \nabla^{k} {u}_{|\phi(x)}
            \contract
            \left( 
                \nabla_{ C_1 }\phi_{} \otimes \cdots \otimes \nabla_{ C_k }\phi_{}
                \otimes 
                \nabla_{ P_1 }\nabla\phi_{} \otimes \cdots \otimes \nabla_{ P_d }\nabla\phi_{}
            \right)
            .
        \end{aligned}
    \end{align}
    If $d = 0$, then it suffices to assume that the weak derivatives $\nabla\phi, \dots, \nabla^{m}\phi$ are essentially bounded. 
\end{proposition}

\begin{proof}
    Let $\test \in \Cont^{\infty}_{c}( \OmegaS, {\otimes}^d X \otimes \VEC^{\ast} )$ be a smooth and compactly supported function.
    We will use that there exists an open set $\OmegaS_{0}$ that compactly contained in $\OmegaS$, has smooth boundary, 
    and such that the support of $\test$ is compactly contained in $\OmegaS_{0}$.

    Since $\phi \in \Sobolev^{m+1,\infty}(\OmegaS_{0},\OmegaT)$, 
    there exists a family of smooth functions $\phi_{\eps} : \OmegaS_{0} \rightarrow \OmegaT$
    such that $\nabla^k \phi_{\eps} \rightarrow \nabla^{k} \phi$ almost everywhere over $\OmegaS_{0}$ as $\eps$ goes to zero 
    for all $0 \leq k \leq m$. 
For example, we can use local mollification.

    First, consider the case that $u$ is smooth. 
    Then ${u}_{\eps|\phi(x)}$ converges to ${u}_{|\phi(x)}$ almost everywhere over $\supp\test$. 
    By the dominated convergence theorem, 
    \begin{align*}
        &
        \int_{\OmegaS}
        \left\langle \phi^{\ast} {u}_{|x}, \nabla_{v_1,v_2,\ldots,v_m} {\test}_{|x} \right\rangle
        \,\dif{x}
=
        \int_{\OmegaS_{0}}
        \left\langle {u}_{|\phi(x)} \contract \otimes^{d} \nabla\phi_{|x}, \nabla_{v_1,v_2,\ldots,v_m} {\test}_{|x} \right\rangle
        \,\dif{x}
        \\&\quad 
        =
        \lim\limits_{ \epsilon \rightarrow 0 }
        \int_{\OmegaS_{0}}
        \left\langle {u}_{\eps|\phi(x)} \contract \otimes^{d} \nabla\phi_{\eps|x}, \nabla_{v_1,v_2,\ldots,v_m} {\test}_{|x} \right\rangle
        \,\dif{x}
=
        \lim\limits_{ \epsilon \rightarrow 0 }
        \int_{\OmegaS_{0}}
        \left\langle \phi_{\eps}^{\ast} {u}_{|x}, \nabla_{v_1,v_2,\ldots,v_m} {\test}_{|x} \right\rangle
        \,\dif{x}
        .
    \end{align*}
    For any $\eps > 0$, we get 
    \begin{align*}
        &
        (-1)^{m}
        \int_{\OmegaS_{0}}
        \left\langle \phi_{\eps}^{\ast} {u}_{|x}, \nabla_{v_1,v_2,\ldots,v_m} {\test}_{|x} \right\rangle
        \,\dif{x}
        \\&
=
        \sum_{ \substack{ {P} \in \calP_{0}(m,d) \\ 0 \leq k \leq |{P}_{0}| \\ \calC \in \calP({P}_{0},k) } }
        \int_{\OmegaS_{0}}
        \left\langle 
            \nabla^{k} {u}_{|\phi_{\eps}(x)}
            \contract
            \left( 
                \nabla_{ C_1 }\phi_{\eps|x} \otimes \cdots \otimes \nabla_{ C_k }\phi_{\eps|x}
                \otimes 
                \nabla_{ P_1 }\nabla\phi_{\eps|x} \otimes \cdots \otimes \nabla_{ P_d }\nabla\phi_{\eps|x}
            \right)
        ,
        {\test}_{|x}
        \right\rangle
        \,\dif{x}
        .
    \end{align*}
    Letting $\eps$ go to zero, we again use the dominated convergence theorem, together with the assumption that $u$ is smooth, 
    and Equation~\eqref{math:weakderivativeofpullback} follows.
    
    Second, consider the case that ${u} \in \Sobolev^{m,1}(\OmegaT,\LIN^d(Y,\VEC))$.
    Then there exists a family $u_\eps \in \Cont^{\infty}(\OmegaT,\LIN^d(Y,\VEC))$
    that converges to $u$ within the space $\Sobolev^{m,1}(\phi(\supp\test),\LIN^d(Y,\VEC))$. 
    We apply the dominated convergence theorem, and find 
    \begin{align*}
        & 
        \int_{\OmegaS}
        \left\langle \phi^{\ast} {u}_{|x}, \nabla_{v_1,v_2,\ldots,v_m} {\test}_{|x} \right\rangle
        \,\dif{x}
        \\&\quad 
        =
        \int_{\OmegaS}
        \left\langle {u}_{|\phi(x)} \contract {\otimes}^d \nabla\phi_{|x}, \nabla_{v_1,v_2,\ldots,v_m} {\test}_{|x} \right\rangle
        \,\dif{x}
        \\&\quad 
        =
        \int_{\OmegaT}
        \left\langle {u}_{|y} \contract {\otimes}^d \nabla\phi_{|\phi^{-1}(y)}, (\nabla_{v_1,v_2,\ldots,v_m} {\test})_{|\phi^{-1}(y)} 
        \right\rangle
        | \det\nabla\phi^{-1}_{|y} |
        \,\dif{y}
        \\&\quad 
        =
        \lim_{ \eps \rightarrow 0 }
        \int_{\OmegaT}
        \left\langle {u}_{\eps|y} \contract {\otimes}^d \nabla\phi_{|\phi^{-1}(y)}, (\nabla_{v_1,v_2,\ldots,v_m} {\test})_{|\phi^{-1}(y)} 
        \right\rangle
        | \det\nabla\phi^{-1}_{|y} |
        \,\dif{y}
        \\&\quad 
        =
        \lim_{ \eps \rightarrow 0 }
        \int_{\OmegaS}
        \left\langle {u}_{\eps|\phi(x)} \contract {\otimes}^d \nabla\phi_{|x}, \nabla_{v_1,v_2,\ldots,v_m} {\test}_{|x} \right\rangle
        \,\dif{x}
        \\&\quad 
        =
        \lim_{ \eps \rightarrow 0 }
        \int_{\OmegaS}
        \left\langle \phi^{\ast} {u}_{\eps|x}, \nabla_{v_1,v_2,\ldots,v_m} {\test}_{|x} \right\rangle
        \,\dif{x}
        .
    \end{align*}
    For any $\epsilon > 0$, we use that $u_{\eps}$ is smooth and a change of variables: 
\begin{align*}
        &
        (-1)^{m}
        \int_{\OmegaS}
        \langle \phi^{\ast} {u}_{\epsilon|x}, \nabla_{v_1,v_2,\ldots,v_m} {\test}_{|x} \rangle
        \,\dif{x}
        \\&=
        \sum_{ \substack{ {P} \in \calP_{0}(m,d) \\ 0 \leq k \leq |{P}_{0}| \\ \calC \in \calP({P}_{0},k) } }
        \int_{\OmegaS}
        \left\langle 
            \nabla^{k} {u}_{\epsilon|\phi(x)}
            \contract
            \left( 
                \nabla_{ C_1 }\phi_{|x} \otimes \cdots \otimes \nabla_{ C_k }\phi_{|x}
                \otimes 
                \nabla_{ P_1 }\nabla\phi_{|x} \otimes \cdots \otimes \nabla_{ P_d }\nabla\phi_{|x}
            \right)
        ,
        {\test}_{|x}
        \right\rangle
        \,\dif{x}
        \\&=
        \sum_{ \substack{ {P} \in \calP_{0}(m,d) \\ 0 \leq k \leq |{P}_{0}| \\ \calC \in \calP({P}_{0},k) } }
        \int_{\OmegaS}
        \left\langle 
            \nabla^{k} {u}_{\epsilon|y}
            \contract
            \left( 
                \bigotimes_{i=1}^{k} \nabla_{ C_i }\phi_{|\phi^{-1}(y)}
                \otimes 
                \bigotimes_{i=1}^{d} \nabla_{ P_i }\phi_{|\phi^{-1}(y)}
\right)
        ,
        {\test}_{|\phi^{-1}(y)}
        \right\rangle
        | \det\nabla\phi^{-1}_{|y} |
        \,\dif{y}
        .
    \end{align*}
    Once more taking the limit, using the dominated convergence theorem, and undoing the change of coordinates, 
    we get 
    \begin{align*}
        &
        (-1)^{m}
        \int_{\OmegaS}
        \langle \phi^{\ast} {u}_{|x}, \nabla_{v_1,v_2,\ldots,v_m} {\test}_{|x} \rangle
        \,\dif{x}
        \\&=
        \sum_{ \substack{ {P} \in \calP_{0}(m,d) \\ 0 \leq k \leq |{P}_{0}| \\ \calC \in \calP({P}_{0},k) } }
        \int_{\OmegaS}
        \left\langle 
            \nabla^{k} {u}_{|\phi(x)}
            \contract
            \left( 
                \nabla_{ C_1 }\phi_{|x} \otimes \cdots \otimes \nabla_{ C_k }\phi_{|x}
                \otimes 
                \nabla_{ P_1 }\nabla\phi_{|x} \otimes \cdots \otimes \nabla_{ P_d }\nabla\phi_{|x}
            \right)
        ,
        {\test}_{|x}
        \right\rangle
        \,\dif{x}
        .
    \end{align*}
    This completes the proof.  
\end{proof}

Thus we have a representation of the weak derivatives of the pullback of locally integrable tensor fields. 
When the tensor fields are in Sobolev-Slobodeckij spaces, this implies the following estimates. 

\begin{corollary}\label{corollary:weakderivativeofpullback:sobolev}
    Let $\OmegaS, \OmegaT \subseteq \bbR^{N}$ be open sets, and let $\VEC$ be a Banach space.
    Let $m \in \bbN$. 
    Suppose that $\phi : \OmegaS \rightarrow \OmegaT$ is locally bi-Lipschitz.
Let $p \in [1,\infty]$.

    \noindent 
    If the weak derivatives $\nabla\phi, \dots, \nabla^{m+1}\phi$ are essentially bounded
    and $u \in \Sobolev^{m,p}(\OmegaT,\LIN^{d}(Y,\VEC))$,
    then $\phi^\ast u \in \Sobolev^{m,p}(\OmegaS,\LIN^{d}(X,\VEC))$ and 
\begin{align}
        \begin{aligned}
        \left\| 
            \nabla^{m} 
            \left( \phi^{\ast} {u} \right)
\right\|_{\Lebesgue^{p}(\OmegaS)}
        &
        \leq 
        \sum_{k=0}^{m}
        \|
            \nabla^{k} u_{|\phi}
        \|_{\Lebesgue^{p}(\OmegaS)}
\left\|
        B_{k,m,d}\left(
            | \nabla^{} \phi |,
\ldots,
            | \nabla^{m+1} \phi |
        \right)
        \right\|_{\Lebesgue^{\infty}(\OmegaS)}
        \\&
        \leq 
        \sum_{k=0}^{m}
        \|
            \nabla^{k} u_{|\phi}
        \|_{\Lebesgue^{p}(\OmegaS)}
B_{k,m,d}\left(
            \| \nabla^{   } \phi \|_{\Lebesgue^{\infty}(\OmegaS)},
\ldots,
            \| \nabla^{m+1} \phi \|_{\Lebesgue^{\infty}(\OmegaS)}
        \right)
        .
        \end{aligned}
    \end{align}
    If the weak derivatives $\nabla\phi, \dots, \nabla^{m}\phi$ are essentially bounded 
    and $u \in \Sobolev^{m,p}(\OmegaT,\VEC)$,
    then $u \circ \phi \in \Sobolev^{m,p}(\OmegaS,\VEC)$ and 
    \begin{align}
        \begin{aligned}
        \left\| 
        \nabla^{m} 
        \left( {u} \circ \phi \right)
        \right\|_{\Lebesgue^{p}(\OmegaS)}
        &
        \leq 
        \sum_{k=1}^{m}
        \|
            \nabla^{k} u_{|\phi}
        \|_{\Lebesgue^{p}(\OmegaS)}
\left\|
        B_{k,m}\left(
            | \nabla^{} \phi |,
\ldots,
            | \nabla^{m} \phi |
        \right)
        \right\|_{\Lebesgue^{\infty}(\OmegaS)}
        \\&
        \leq 
        \sum_{k=1}^{m}
        \|
            \nabla^{k} u_{|\phi}
        \|_{\Lebesgue^{p}(\OmegaS)}
B_{k,m}\left(
            \| \nabla^{ } \phi \|_{\Lebesgue^{\infty}(\OmegaS)},
\ldots,
            \| \nabla^{m} \phi \|_{\Lebesgue^{\infty}(\OmegaS)}
        \right)
        .
        \end{aligned}
    \end{align}
\end{corollary}

\begin{proof}
    We use Proposition~\ref{prop:weakderivativesofpullback},
    H\"older's inequality,
    and the fact that the (generalized) Bell polynomials 
   ~\eqref{math:bellpolynomial:partial} and~\eqref{math:bellpolynomial:dgeneralized} have non-negative coefficients.  
\end{proof}

\begin{example}\label{example:sobolevvectorfield}
    Suppose that $u \in \Sobolev^{m,p}(\OmegaT,\bbR^{N})$ is a vector field 
    and that $\phi : \OmegaS \rightarrow \OmegaT$ is a bi-Lipschitz mapping of regularity $\Cont^{m,1}$
    between open sets in $\bbR^{N}$.
    Identifying vector fields with covariant $1$-tensors (via transposition), 
    the pullback $\phi^{\ast} u$ equals the new vector field $(\nabla \phi)^{t} u_{|\phi}$. 
    The derivatives of that vector field satisfy 
    \begin{align*}
        \left\| 
            \nabla^{m} 
            \left( \phi^{\ast} {u} \right)
        \right\|_{\Lebesgue^{p}(\OmegaS)}
        &\leq 
        \sum_{k=0}^{m}
        \|
            \nabla^{k} u_{|\phi}
        \|_{\Lebesgue^{p}(\OmegaS)}
        B_{k,m,1}\left(
            \| \nabla^{   } \phi \|_{\Lebesgue^{\infty}(\OmegaS)},
            \ldots,
            \| \nabla^{m+1} \phi \|_{\Lebesgue^{\infty}(\OmegaS)}
        \right)
        \\&\leq 
        \sum_{k=0}^{m}
        \|
            \nabla^{k} u
        \|_{\Lebesgue^{p}(\OmegaT)}
        \left\| 
            \det \nabla \phi^{-1}
        \right\|_{\Lebesgue^{\infty}(\OmegaT)}^{\frac 1 p}
        B_{k,m,1}\left(
            \| \nabla^{   } \phi \|_{\Lebesgue^{\infty}(\OmegaS)},
            \ldots,
            \| \nabla^{m+1} \phi \|_{\Lebesgue^{\infty}(\OmegaS)}
        \right)
        ,
    \end{align*}    
    where we take the Lebesgue $p$-norms of the pointwise spectral norms of $u$ on both sides. 
\end{example}

\begin{example}\label{example:sobolevmatrixfield}
    Suppose that $u \in \Sobolev^{m,p}(\OmegaT,\bbR^{N \times N})$ is a matrix field 
    and that $\phi : \OmegaS \rightarrow \OmegaT$ is a bi-Lipschitz mapping of regularity $\Cont^{m,1}$
    between open sets in $\bbR^{N}$.
    Identifying matrix fields with covariant $2$-tensors, 
    the pullback $\phi^{\ast} u$ equals the new matrix field $(\nabla \phi)^{t} ( u \circ \phi ) (\nabla \phi)$. 
    The derivatives of that matrix field satisfy the inequality 
    \begin{align*}
        \left\| 
            \nabla^{m} 
            \left( \phi^{\ast} {u} \right)
\right\|_{\Lebesgue^{p}(\OmegaS)}
        &
        \leq 
        \sum_{k=0}^{m}
        \|
            \nabla^{k} u_{|\phi}
        \|_{\Lebesgue^{p}(\OmegaS)}
        \left\| 
            \det \nabla \phi^{-1}
        \right\|_{\Lebesgue^{\infty}(\OmegaT)}^{\frac 1 p}
        B_{k,m,2}\left(
            \| \nabla^{   } \phi \|_{\Lebesgue^{\infty}(\OmegaS)},
            \ldots,
            \| \nabla^{m+1} \phi \|_{\Lebesgue^{\infty}(\OmegaS)}
        \right)
        \\&\leq 
        \sum_{k=0}^{m}
        \|
            \nabla^{k} u
        \|_{\Lebesgue^{p}(\OmegaT)}
        \left\| 
            \det \nabla \phi^{-1}
        \right\|_{\Lebesgue^{\infty}(\OmegaT)}^{\frac 1 p}
        B_{k,m,2}\left(
            \| \nabla^{   } \phi \|_{\Lebesgue^{\infty}(\OmegaS)},
            \ldots,
            \| \nabla^{m+1} \phi \|_{\Lebesgue^{\infty}(\OmegaS)}
        \right)
    \end{align*}    
    in terms of generalized Bell polynomials. 
    Here, we take the Lebesgue $p$-norms of the pointwise spectral norms of $u$ on both sides.
\end{example}

Next, we address tensor fields in Sobolev-Slobodeckij norms and derive an analogous result for pullback.
Before stating the main result, we prove several auxiliary lemmas. 
We begin with estimating the Sobolev-Slobodeckij norm of a tensor contraction~\cite[Proposition~3.2.1]{wilbrandt2019stokes}).

\begin{lemma}\label{lemma:sobolevslobodeckijproduct}
    Let $\OmegaS \subseteq \bbR^{N}$ be an open set and let $\VEC_1, \VEC_2, \VEC_3$ be Banach spaces. 
    Let $p \in [1,\infty]$ and $\theta,\sigma \in (0,1]$,
    and assume that either $p = \infty$ and $\theta \leq \sigma$, or that $p < \infty$ and $\theta < \sigma$. 
    If ${u} \in \Sobolev^{\theta,p}(\OmegaS,\LIN(\VEC_1,\VEC_2))$
    and $\psi \in \Cont^{0,\sigma}(\OmegaS,\LIN(\VEC_2,\VEC_3))$, then 
\noindent 
\begin{align}\label{math:sobolevslobodeckijproduct}
            \begin{aligned}
                | {u} \contract \psi |_{\Sobolev^{\theta,p}(\OmegaS)}
                &
                \leq 
                | {u} |_{\Sobolev^{\theta,p}(\OmegaS)}
                \| \psi \|_{\Lebesgue^{\infty}(\OmegaS)}
                \\&\quad\quad
                +
\frac{ \SphereArea_{N-1}^{\frac{1}{p}} }{ (\sigma-\theta)^{\frac{1}{p}} p^{\frac{1}{p}} }
                \| {u} \|_{\Lebesgue^{p}(\OmegaS)}
                | \psi |_{\Cont^{0,\sigma}(\OmegaS)}
                +
2
                \frac{ \SphereArea_{N-1}^{\frac{1}{p}}}{\theta^{\frac{1}{p}} p^{\frac{1}{p}}}
                \| {u} \|_{\Lebesgue^{p}(\OmegaS)}
                \| \psi \|_{\Lebesgue^{\infty}(\OmegaS)}
                .
            \end{aligned}
        \end{align}        
Here, $\SphereArea_{N-1}$ denotes the area of the Euclidean unit sphere of dimension $N-1$, 
    and we understand the constants in the limit if $p=\infty$. 
\end{lemma}

\begin{proof}
    We first observe that for almost all $(x,x') \in \Omega\times\Omega$ we have 
    \begin{align*}
| u(x) \contract \psi(x) - u(x') \contract \psi(x') |
&
        \leq 
        | u(x) \contract \psi(x) - u(x') \contract \psi(x) |
+
        | u(x') \contract \psi(x) - u(x') \contract \psi(x') |
\\&
        \leq 
        | u(x) - u(x') | \cdot \| \psi \|_{\Lebesgue^{\infty}}
+
        | u(x') \contract \psi(x) - u(x') \contract \psi(x') |
.
    \end{align*}
    Using the Minkowski inequality for the $\Lebesgue^{p}$ norm over $\OmegaS \times \OmegaS$,
    we see
    \begin{align*}
        | {u} \contract \psi |_{\Sobolev^{\theta,p}(\OmegaS)}
        \leq 
        | {u} |_{\Sobolev^{\theta,p}(\OmegaS)}
        \| \psi \|_{\Lebesgue^{\infty}(\OmegaS)}
        +
        \| 
            \dist( x, x' )^{-\theta - \frac N p} 
            ( u(x') \contract \psi(x) - u(x') \contract \psi(x') )  
        \|_{\Lebesgue^{p}(\OmegaS \times \OmegaS)} 
        .
    \end{align*}
    If $p = \infty$,
    then we use a case distinction. 
    For almost all $x, x' \in \OmegaS$ with $\dist(x,x') < 1$,
    \begin{align*}
        | u(x') \contract \psi(x) - u(x') \contract \psi(x') |
        \cdot \dist( x, x' )^{-\theta}
        &
        \leq 
        | u(x') | \cdot | \psi(x) - \psi(x') |
        \cdot \dist( x, x' )^{-\theta}
        \\& 
        \leq 
        | \psi |_{\Sobolev^{\sigma,\infty}}
        | u(x') | \cdot \dist( x, x' )^{\sigma}
        \cdot \dist( x, x' )^{-\theta}
        \\&
        \leq 
        | \psi |_{\Sobolev^{\sigma,\infty}}
        | u(x') |
        .
    \end{align*}
    For almost all $x, x' \in \OmegaS$ with $\dist(x,x') > 1$,
    \begin{align*}
        | u(x') \contract \psi(x) - u(x') \contract \psi(x') |
        \cdot \dist( x, x' )^{-\theta}
        &
        \leq 
        | u(x') | \cdot | \psi(x) - \psi(x') |
        \cdot \dist( x, x' )^{-\theta}
        \\& 
        \leq 
        2 | \psi |_{\Lebesgue^{\infty}}
        | u(x') | 
        .
    \end{align*}
    This proves~\eqref{math:sobolevslobodeckijproduct} when $p = \infty$.
    
If $p < \infty$, we split the $\Lebesgue^{p}$ norm over $\OmegaS \times \OmegaS$ into two integrals. 
    On the one hand,
    \begin{align*}
        &
        \iint\limits_{ \substack{ \OmegaS \times \OmegaS \\ \dist(x,x') < 1 } }
        \frac{
            | u(x') \contract \psi(x) - u(x') \contract \psi(x') |^{p}
        }{
            \dist( x, x' )^{ \theta p + N }
        }
        \,\dif x \,\dif x'
        \\&\quad
        \leq
        | \psi |_{\Cont^{0,\sigma}(\OmegaS)}^{p} 
        \int_{\OmegaS}
        | u(x') |^{p}
        \int_{ \OmegaS : \dist(x,x') < 1}
        \dist( x, x' )^{ ( \sigma - \theta ) p - N }
        \,\dif x \,\dif x'
        \\&\quad
        \leq
        \frac{ \SphereArea_{N-1} }{ ( \sigma - \theta ) p }
        | \psi |_{\Cont^{0,\sigma}(\OmegaS)}^{p} 
        \int_{\OmegaS}
        | u(x') |^{p}
        \,\dif x'
        .
    \end{align*}
    On the other hand, 
    \begin{align*}
        &
        \iint\limits_{ \substack{ \OmegaS \times \OmegaS \\ \dist(x,x') > 1 } }
        \frac{
            | u(x') \contract \psi(x) - u(x') \contract \psi(x') |^{p}
        }{
            \dist( x, x' )^{ \theta p + N }
        }
        \,\dif x \,\dif x'
        \\&\quad 
        \leq
        2^p \| \psi \|_{\Lebesgue^{\infty}(\OmegaS)}^{p}
        \int_{\OmegaS}
        | u(x') |^{p} 
        \int_{\OmegaS : \dist(x,x') > 1}
        \dist( x, x' )^{ - \theta p - N }
        \,\dif x \,\dif x'
        \\&\quad 
        \leq
        2^p \frac{ \SphereArea_{N-1} }{ \theta p }
        \| \psi \|_{\Lebesgue^{\infty}(\OmegaS)}^{p}
        \int_{\OmegaS}
        | u(x') |^{p} 
        \,\dif x'
        .
    \end{align*}
Here, we have used that the integrals in $x$ converge because of the conditions on the exponents. 
Specifically, 
    \begin{gather*}
        \int_{\OmegaS : |x| < 1}
        | x |^{ ( \sigma - \theta ) p - N }
        \,\dif x 
        =
        \SphereArea_{N-1}
        \int_{0}^{1}
        r^{ ( \sigma - \theta ) p - N + N - 1}
        \,\dif r
        =
        \frac{ \SphereArea_{N-1} }{ ( \sigma - \theta ) p },
        \\
\int_{\OmegaS : |x| > 1}
        | x |^{ - \theta p - N }
        \,\dif x
        =
        \SphereArea_{N-1}
        \int_{1}^{\infty}
        r^{ - \theta p - N + N - 1}
        \,\dif r
        =
        \SphereArea_{N-1}
        \int_{1}^{\infty}
        r^{ - \theta p - 1}
        \,\dif r
        =
        \frac{ \SphereArea_{N-1} }{ \theta p }
        .
    \end{gather*}
    This shows~\eqref{math:sobolevslobodeckijproduct}, thus completing the proof of the auxiliary lemma. 
\end{proof}

We now estimate the Sobolev-Slobodeckij seminorms of tensors after a pullback. 

\begin{lemma}\label{lemma:transformationssatz:slobodeckij}
Let $\OmegaS, \OmegaT \subseteq \bbR^{N}$ be open sets, and let $\VEC$ be a Banach space.
    Suppose that $\phi : \OmegaS \rightarrow \OmegaT$ is bi-Lipschitz. 
    Let  $p \in [1,\infty]$ and $\theta \in (0,1]$.
    If $u \in \Sobolev^{\theta,p}(\OmegaT,\VEC)$, then $u \circ \phi \in \Sobolev^{\theta,p}(\OmegaS,\VEC)$ with
    \begin{align}\label{math:transformationssatz:slobodeckij:scalar}
        \left| u \circ \phi \right|_{\Sobolev^{\theta,p}(\OmegaS,\VEC)}
        \leq 
        \left| u \right|_{\Sobolev^{\theta,p}(\OmegaT,\VEC)}
        \left\| 
            \det \nabla \phi^{-1}
        \right\|_{\Lebesgue^{\infty}(\OmegaT)}^{\frac 2 p}
        \left| 
            \phi
        \right|_{\Cont^{0,1}(\OmegaS)}^{ \theta + \frac N p }
        .
    \end{align}
    Now assume that either $p = \infty$ and $\theta \leq \sigma$, or that $p < \infty$ and $\theta < \sigma$. 
    If $u \in \Sobolev^{\theta,p}(\OmegaT,\LIN^{d}(Y,\VEC))$, then $\phi^{\ast} u \in \Sobolev^{\theta,p}(\OmegaS,\LIN^{d}(X,\VEC))$ with
    \begin{align}\label{math:transformationssatz:slobodeckij:tensor}
        \begin{aligned}
        | \phi^{\ast} u |_{\Sobolev^{\theta,p}(\OmegaS)}
        &
        \leq 
        | {u} \circ \phi |_{\Sobolev^{\theta,p}(\OmegaS)}
        \| {\otimes}^d \nabla\phi \|_{\Lebesgue^{\infty}(\OmegaS)}
        \\&\quad\quad
        +
\frac{ \SphereArea_{N-1}^{\frac{1}{p}} }{ (\sigma-\theta)^{\frac{1}{p}} p^{\frac{1}{p}} }
        \| {u} \circ \phi \|_{\Lebesgue^{p}(\OmegaS)}
        | {\otimes}^d \nabla\phi |_{\Cont^{0,\sigma}(\OmegaS)}
        +
2
        \frac{ \SphereArea_{N-1}^{\frac{1}{p}}}{\theta^{\frac{1}{p}} p^{\frac{1}{p}}}
        \| {u} \circ \phi \|_{\Lebesgue^{p}(\OmegaS)}
        \| {\otimes}^d \nabla\phi \|_{\Lebesgue^{\infty}(\OmegaS)}
        .
    \end{aligned}
\end{align}
    
\end{lemma}

\begin{proof}
    Let $u \in \Sobolev^{\theta,p}(\OmegaT,\VEC)$.
    For almost all $x, x' \in \OmegaS$ we observe that 
    \begin{align*}
        &
        | u(\phi(x)) - u(\phi(x')) |_{\VEC} \dist(x,x')^{ - \theta - \frac N p }
        \\&\quad 
        = 
        | u(\phi(x)) - u(\phi(x')) |_{\VEC}
        \dist(\phi(x),\phi(x'))^{ - \theta - \frac N p }
        \cdot 
        \frac{
            \dist(\phi(x),\phi(x'))^{ \theta + \frac N p }
        }{
            \dist(x,x')^{ \theta + \frac N p }
        }
        \\&\quad 
        \leq  
        | u(\phi(x)) - u(\phi(x')) |_{\VEC}
        \dist(\phi(x),\phi(x'))^{ - \theta - \frac N p }
        |\phi|_{\Cont^{0,1}(\OmegaS)}^{ \theta + \frac N p }
        .
    \end{align*}
    We immediately conclude~\eqref{math:transformationssatz:slobodeckij:scalar}
    in the case $p = \infty$.
    Now assume $p < \infty$.
    In conjunction with change of variables over the measure space $\OmegaS \times \OmegaS$, 
    we estimate 
    \begin{align*}
        &
        \left\| 
            | u(\phi(x)) - u(\phi(x')) |_{\VEC} \dist(x,x')^{ - \theta - \frac N p }
        \right\|_{\Lebesgue^{p}(\OmegaS\times\OmegaS)}
        \\&\quad 
        \leq
        \left( 
            \iint_{ \OmegaS \times \OmegaS }
                | u(\phi(x)) - u(\phi(x')) |_{\VEC}^{p}
                \dist(\phi(x),\phi(x'))^{ - p \theta - N }
                |\phi|_{\Cont^{0,1}(\OmegaS)}^{ p \theta + N }
            \,\dif x \,\dif x'
        \right)^{\frac 1 p} 
        \\&\quad 
        \leq
        |\phi|_{\Cont^{0,1}(\OmegaS)}^{ \theta + \frac N p }
        \left( 
            \iint_{ \OmegaS \times \OmegaS }
                | u(\phi(x)) - u(\phi(x')) |_{\VEC}^{p}
                \dist(\phi(x),\phi(x'))^{ - p \theta - N }
            \,\dif x \,\dif x'
        \right)^{\frac 1 p} 
        \\&\quad 
        \leq
        |\phi|_{\Cont^{0,1}(\OmegaS)}^{ \theta + \frac N p }
        \left( 
            \iint_{ \OmegaT \times \OmegaT }
                | u(y) - u(y') |_{\VEC}^{p}
                \dist(y,y')^{ - p \theta - N }
                |\det \nabla\phi^{-1}_{|y}|^2
            \,\dif y \,\dif y'
        \right)^{\frac 1 p} 
        .
    \end{align*}
    This shows~\eqref{math:transformationssatz:slobodeckij:scalar} also when $p < \infty$.
    
    Next, let $u \in \Sobolev^{\theta,p}(\OmegaT,\LIN^{d}(Y,\VEC))$.
    Then $\phi^{\ast} {u} = ( u \circ \phi ) \contract ( \otimes^{d} \nabla\phi )$. 
    We use~\eqref{math:transformationssatz:slobodeckij:scalar} and Lemma~\ref{lemma:sobolevslobodeckijproduct} to deduce~\eqref{math:transformationssatz:slobodeckij:tensor}.
    The proof is complete. 
\end{proof}

We are now in a position to estimate the Sobolev-Slobodeckij seminorms of the higher derivatives of a pullback. 

\begin{proposition}\label{prop:pullbackoftensors:slobodeckij}
    Let $\OmegaS, \OmegaT \subseteq \bbR^{N}$ be open sets, and let $\VEC$ be a Banach space.
    Let $m \in \bbN$. 
    Suppose that $\phi : \OmegaS \rightarrow \OmegaT$ is locally bi-Lipschitz. 
Suppose that the weak derivatives $\nabla\phi, \dots, \nabla^{m+1}\phi$ are in $\Sobolev^{\sigma,\infty}(\OmegaS)$ for some $\sigma \in (0,1]$.
    Let $p \in [1,\infty]$ and $\theta \in (0,1]$, 
    and assume that either $p = \infty$ and $\theta \leq \sigma$ or that $p < \infty$ and $\theta < \sigma$. 
    
    If $u \in \Sobolev^{m+\theta,p}(\OmegaT,\LIN^{d}(Y,\VEC))$, 
    then $\phi^\ast u \in \Sobolev^{m+\theta,\infty}(\OmegaS,\LIN^{d}(X,\VEC))$,
    and we have 
    \begin{align*}
        \left|
            \nabla^{m} (\phi^{\ast} {u} )
        \right|_{\Sobolev^{\theta,p}(\OmegaS)}
        &
        \leq 
        \sum_{ 0 \leq k \leq m }
        \left| 
            \nabla^{k} u_{|\phi}
        \right|_{\Sobolev^{\theta,p}(\OmegaS)}
        \left\|
            B_{m,k,d}\left( |\nabla\phi|, \ldots, |\nabla^{m+1}\phi| \right)
        \right\|_{\Lebesgue^{\infty}(\OmegaS)}
        \\&\quad 
        + 
        \frac{ \SphereArea_{N-1}^{\frac{1}{p}} }{ (\sigma-\theta)^{\frac{1}{p}} p^{\frac{1}{p}} }
        \sum_{ 0 \leq k \leq m }
        \left\| 
            \nabla^{k} u_{|\phi}
        \right\|_{\Lebesgue^{p}(\OmegaS)}
        \left|
            B_{m,k,d}\left( |\nabla\phi|, \ldots, |\nabla^{m+1}\phi| \right)
        \right|_{\Cont^{0,\sigma}(\OmegaS)}
        \\&\quad 
        + 
        2 \frac{ \SphereArea_{N-1}^{\frac{1}{p}} }{ \theta^{\frac{1}{p}} p^{\frac{1}{p}} }
        \sum_{ 0 \leq k \leq m }
        \left\| 
            \nabla^{k} u_{|\phi}
        \right\|_{\Lebesgue^{p}(\OmegaS)}
        \left\|
            B_{m,k,d}\left( |\nabla\phi|, \ldots, |\nabla^{m+1}\phi| \right)
        \right\|_{\Lebesgue^{\infty}(\OmegaS)}
        .
    \end{align*}
    Here, we understand the constants in the limit if $p=\infty$.
    If $d = 0$, then it suffices to assume that the derivatives $\nabla\phi, \dots, \nabla^{m}\phi$ exist almost everywhere and are in $\Sobolev^{\sigma,\infty}(\OmegaS)$. 
\end{proposition}

\begin{proof}
    We use Proposition~\ref{prop:weakderivativesofpullback} together with Lemma~\ref{lemma:sobolevslobodeckijproduct} and Lemma~\ref{lemma:transformationssatz:slobodeckij}.
\end{proof}

\begin{remark}
    The reader may have noticed that the Sobolev and Sobolev-Slobodeckij norms of tensor fields $\nabla^{k} {u} \circ \phi$ appear in the final estimate.
    These are simply the integrals of $| \nabla^{k} {u} | \circ \phi$, 
    which are estimated by Lemmas~\ref{lemma:transformationssatz:lp}~and~\ref{lemma:transformationssatz:slobodeckij}.
\end{remark}

\begin{remark}
    Obviously, the Sobolev and Sobolev-Slobodeckij norms of the pullback $\phi^{\ast}u$
    depend on higher derivatives of the transformation $\phi$.
    We summarize how exactly this depends on the number of covariant indices $d$ of the tensor field $u$. 
    
    In the case $d=0$, when $u$ is a scalar-valued function with $m$ weak derivatives, 
    we need that $\phi$ is differentiable almost everywhere up to order $m$.
    But the situation changes drastically in the case $d > 0$.
    the nature of the pullback requires that the coordinate transformation is differentiable almost everywhere up to order $m+1$. 
    Generally speaking,
    $\Sobolev^{m,p}$ functions are even preserved under coordinate transformations along coordinate transformations with $m$ derivatives, 
    whereas $\Sobolev^{m,p}$ tensor fields are preserved under coordinate transformations with $m+1$ derivatives. 
    For example, this applies to the transformation of vector fields. 

    Sobolev-Slobodeckij spaces with non-integer smoothness involve additional subtlety. 
    The derivatives of the coordinate transformation up to the relevant orders 
    must feature H\"older regularity that at least matches the H\"older regularity of the original tensor field or function.
    More specifically, 
    unless $p=\infty$, 
    our estimates even require that the transformation has \emph{strictly higher} H\"older regularity.
    The estimate in Proposition~\ref{prop:pullbackoftensors:slobodeckij} gets worse as the fractional smoothness parameter $\theta$ approaches either zero or one. This is a typical phenomenon with Sobolev-Slobodeckij spaces. 
\end{remark}

\section{Musielak-Orlicz spaces and estimates}\label{sec:orlicz}

\emph{Musielak-Orlicz spaces}, or \emph{Orlicz spaces} for short, are an important generalization of Lebesgue spaces. 
They allow for a more nuanced analysis of growth conditions and appear in the analysis of nonlinear partial differential equations.
Taking into account weak derivatives, we generalize Sobolev spaces to Musielak-Orlicz-Sobolev spaces.
In a similar manner, we generalize Sobolev-Slobodeckij spaces to Musielak-Orlicz-Sobolev-Slobodeckij spaces. 
The latter spaces have only recently received wider attention in the literature.
The functions that we discuss in this section include the classical Orlicz spaces, Lebesgue spaces with variable exponents, and general Musielak-Orlicz spaces. 
We refer to the monographs~\cite{diening2011lebesgue},~\cite{Harjulehto2019} and~\cite{chlebicka2021partial} 
for more theoretical background and applications of Orlicz and Orlicz-Sobolev spaces,
and to the survey~\cite{baalal2019density,alberico2021fractional,azroul2022class} for more information on Orlicz-Sobolev-Slobodeckij spaces. 

As in the preceding section, 
we continue to assume that $X = Y = \bbR^{N}$, 
let $\OmegaS \subseteq X$ be an open set, 
and let $\VEC$ be a normed space. 
\\

We begin with basic definitions.
A measurable function $\Orlicz : \Omega \times [0,\infty] \rightarrow [0,\infty]$ is called a \emph{Musielak-Orlicz integrand} if
\begin{itemize}
    \item $\Orlicz(x,0) = 0$.
    \item $\Orlicz(x,\cdot)$ is convex, left-continuous, and non-decreasing in the second variable.
    \item $\Orlicz(x,\cdot)$ is a non-zero function for almost all $x \in \OmegaS$.
    \item $\Orlicz(\cdot,\xi)$ is measurable for all $\xi \in [0,\infty)$.
\end{itemize}
The \emph{Musielak-Orlicz} space is the linear span 
\begin{align}\label{math:orliczmusielakclass}
    \Lebesgue^{\Orlicz}(\Omega,\VEC)
    :=
    \operatorname{span}
    \left\{ 
        u : \Omega \rightarrow \VEC
        \suchthat* 
        \int_{\Omega} \Orlicz\left( x, |u(x)| \right) < \infty
    \right\}
    .
\end{align}
The \emph{Luxemburg norm} on $\Lebesgue^{\Orlicz}(\Omega,\VEC)$ is the functional
\begin{align}\label{math:luxemburgnorm}
    \| u \|_{\Lebesgue^{\Orlicz}(\Omega)}
    :=
    \left\{ 
        \lambda > 0
        \suchthat*
        \int_{\Omega} \Orlicz\left( x, \frac{ |u(x)| }{\lambda} \right) \leq 1
    \right\}
    .
\end{align}
Recall that the infimum of the empty set is $\infty$.
Below we will prove that this indeed defines a norm.

We say that a Musielak-Orlicz integrand is \emph{proper} if all members of $\Lebesgue^{\Orlicz}(\Omega,\VEC)$ are locally integrable.

\begin{example}
    We list a few standard examples for Musielak-Orlicz integrands.
    \begin{enumerate}
     \item 
     $\Orlicz(\xi) = \xi^{p}$ for any $p \in [1,\infty)$. 
     The associated Luxemburg norm is the Lebesgue $p$-norm.
     \item 
     $\Orlicz(\xi) = \infty \cdot \chi_{(1,\infty)}\left( \xi \right)$, which we also abbreviate as $\xi^{\infty}$.
     In other words, $\Orlicz(\xi)$ equals $0$ for $\xi \in [0,1]$ and equals $\infty$ for $\xi > 1$.
     The associated Luxemburg norm is the Lebesgue $\infty$-norm.
     \item 
     $\Orlicz(x,\xi) = \xi^{p(x)}$, where $p(x) \in [1,\infty]$ is a measurable function defined over the domain and mapping into the extended reals. The associated space is known as Lebesgue space with variable exponent. 
     \item 
     $\Orlicz(x,\xi) = \xi^{p(x)} + a(x) \xi^{q(x)}$ with $p, q : \Omega \rightarrow [1,\infty]$ measurable and $a : \Omega \rightarrow \bbR$ bounded and non-negative. This gives rise to a Musielak-Orlicz space known as \emph{double phase space}.
     \item 
     $\Orlicz(x,\xi) = \xi^{p(x)}\ln( e + \xi )$.
     \item 
     $\Orlicz(x,\xi) = e^{|\xi|^{\sigma}}-1$ for some $\sigma \geq 1$. 
     \item 
     $\Orlicz(x,\xi) = \infty \cdot \chi_{(0,\infty)}(\xi)$ is zero at the origin and $\infty$ everywhere else.
     The resulting Orlicz space is the trivial vector space.
     \item 
     Pointwise non-negative combinations of Musielak-Orlicz integrands yield Musielak-Orlicz integrands again.
    \end{enumerate}
\end{example}

\begin{remark}
    Members of Orlicz spaces are measurable but not necessarily locally integrable. 
    For example, if $\OmegaS = (-1,1)$ and $\Orlicz(x,\xi) = x^2 \xi$, then $x^{-2} \in \Lebesgue^{\Orlicz}(\OmegaS)$ but $x^{-2}$ is not locally integrable.

    But if for any compact set $E \subseteq \OmegaS$ there exists $\mu > 0$ such that $\Orlicz(x,\mu)$ has a positive lower bound over $E$, 
    then every member of $\Lebesgue^{\Orlicz}(\OmegaS)$ is locally integrable. 
    For example, this condition holds for variable exponent Lebesgue spaces. 
\end{remark}

\begin{remark}
    We remark on a few aspects of these definitions and the different variations that can be found in the literature. 
    Function with these properties are also known as \emph{$N$-functions} in the literature. 
    
    If $x \in \OmegaS$ and $\Orlicz(x,\cdot)$ is non-zero,
    then $\Orlicz(x,0)=0$, the convexity, left-continuity, and non-decreasing condition imply the following: the function $\Orlicz(x,\cdot)$ equals zero on a closed interval $[0,\xi_0]$ for some $\xi_0 > 0$; the function is finite, strictly increasing and continuous over a closed interval $[\xi_0,\xi_1]$ for some $\xi_1 > 0$, the latter possibly infinite, and $\Orlicz(x,\xi)=\infty$ for all $\xi > \xi_1$.
    
    Imposing left-continuity only ensures that $\Orlicz(x,\cdot)$ is finite on a closed interval. Note that we permit the two special cases where either $\Orlicz(x,\xi)=0$ or $\Orlicz(x,\xi)=\infty$ for all $\xi > 0$. 
    
    This definition is slightly more general than what is used in~\cite[Definition~2.3.1]{diening2011lebesgue} and the isotropic case in~\cite[Definition~2.2.2]{chlebicka2021partial}. 
    But in the contrast to the latter reference, we do not consider anisotropic Musielak-Orlicz spaces in our discussion. 
    Many authors impose additional assumptions on $\Orlicz(x,\xi)$, 
    such as finiteness, sublinearity as $\xi$ approaches zero, or superlinearity as $\xi$ goes to infinity.
    We refrain from using those additional assumptions 
    since that would exclude the relevant special cases of $\Lebesgue^{1}(\Omega,\VEC)$ and $\Lebesgue^{\infty}(\Omega,\VEC)$.
    We also do not require the so-called $\Delta_2$-condition (see~\cite{chlebicka2021partial}). 
    Notably, the discussion in~\cite[Chapter~2]{Harjulehto2019} uses weaker definitions than we do.
\end{remark}

\begin{proposition}
    The space $\Lebesgue^{\Orlicz}(\Omega,\VEC)$ together with the Luxemburg norm is a normed space.
\end{proposition}
\begin{proof}
    By definition, $\Lebesgue^{\Orlicz}(\Omega,\VEC)$ is a vector space,
    so it remains to be shown that the Luxemburg norm is a norm. 
    
    Before we proceed, we first show that for every $u : \Omega \rightarrow \VEC$ with $u \in \Lebesgue^{\Orlicz}(\Omega,\VEC)$ 
    the following condition is true: 
    there exists $\lambda > 0$ such that $\calA(x,\lambda^{-1}|u(x)|)$ has finite integral.
    This condition is obviously true if $\calA(x,|u(x)|)$ has finite integral,
    and if $u \in \Lebesgue^{\Orlicz}(\Omega,\VEC)$ satisfies the condition, 
    then so does $a u \in \Lebesgue^{\Orlicz}(\Omega,\VEC)$ for any $a \in \bbR$.
    Lastly, if $u, v \in \Lebesgue^{\Orlicz}(\Omega,\VEC)$ satisfy the condition, 
    with respective parameters $\lambda_{u}, \lambda_{v} > 0$, then the inequality 
    \begin{align*}
        \Orlicz\left( x, \frac{ |u(x) + v(x)| }{ \lambda_{u} + \lambda_{v} } \right)
        \leq 
        \frac{ \lambda_{u} }{ \lambda_{u} + \lambda_{v} }
        \Orlicz\left( x, \frac{ |u(x)| }{ \lambda_{u} } \right)
        +
        \frac{ \lambda_{v} }{ \lambda_{u} + \lambda_{v} }
        \Orlicz\left( x, \frac{ |v(x)| }{ \lambda_{v} } \right)
    \end{align*}
    shows the condition for $u+v$. By construction,
    every member of $\Lebesgue^{\Orlicz}(\Omega,\VEC)$ satisfies the condition. 
    
    If $u \in \Lebesgue^{\Orlicz}(\Omega,\VEC)$ and $\lambda > 0$ such that $\calA(x,\lambda^{-1}|u(x)|)$ has finite integral,
    then for all $\mu \geq 1$ we have
    \begin{align*}
        \int_{\OmegaS} \calA(x,\lambda^{-1}|u(x)|)
        \geq 
        \mu 
        \int_{\OmegaS} \calA(x,\mu^{-1}\lambda^{-1}|u(x)|
    \end{align*}
    due to convexity. 
    For $\mu > 0$ large enough, $\calA(x,\mu^{-1}\lambda^{-1}|u(x)|)$ has integral at most $1$.
    Hence $\|\cdot\|_{\Lebesgue^{\Orlicz}}$ is finite over $\Lebesgue^{\Orlicz}(\Omega,\VEC)$. 
    
    For any $\Lebesgue^{\Orlicz}(\Omega,\VEC)$ and $t \in \bbR$ we verify 
    \begin{align*}
        \| t u \|_{\Lebesgue^{\Orlicz}(\Omega)}
        &:=
        \inf\left\{ 
            \lambda > 0
            \suchthat*
            \int_{\Omega} \Orlicz\left( x, \frac{ |t u(x)| }{\lambda} \right) \leq 1
        \right\}
        \\&
        \leq 
        \inf\left\{ 
            \lambda' t > 0
            \suchthat*
            \int_{\Omega} \Orlicz\left( x, \frac{ |u(x)| }{\lambda'} \right) \leq 1
        \right\}
        =
        t \| u \|_{\Lebesgue^{\Orlicz}(\Omega)}
        .
    \end{align*}
    Let $u, v \in \Lebesgue^{\Orlicz}(\Omega,\VEC)$ and let $\lambda_{u}, \lambda_{v} > 0$ 
    with $\lambda_{u} > \|u\|_{\Lebesgue^{\Orlicz}(\Omega)}$ and $\lambda_{v} > \|v\|_{\Lebesgue^{\Orlicz}(\Omega)}$.
    By convexity of the Musielak-Orlicz integrand in the second variable, 
    \begin{align*}
        \int_{\Omega} \Orlicz\left( x, \frac{ |u(x) + v(x)| }{ \lambda_{u} + \lambda_{v} } \right)
        \leq 
        \frac{ \lambda_{u} }{ \lambda_{u} + \lambda_{v} }
        \int_{\Omega} \Orlicz\left( x, \frac{ |u(x)| }{ \lambda_{u} } \right)
        +
        \frac{ \lambda_{v} }{ \lambda_{u} + \lambda_{v} }
        \int_{\Omega} \Orlicz\left( x, \frac{ |v(x)| }{ \lambda_{v} } \right)
        \leq 
        1.
    \end{align*}
    Hence $\| u + v \|_{\Lebesgue^{\Orlicz}(\Omega)} \leq \lambda_{u} + \lambda_{v}$. 
    In the limit, we obtain the triangle inequality 
    \begin{align*}
        \| u + v \|_{\Lebesgue^{\Orlicz}(\Omega)} \leq \| u \|_{\Lebesgue^{\Orlicz}(\Omega)} + \| v \|_{\Lebesgue^{\Orlicz}(\Omega)}.
    \end{align*}
    If $u = 0$, then $\| u \|_{\Lebesgue^{\Orlicz}(\Omega)} = 0$ is easily seen. 
    Showing that $u \neq 0$ implies $\| u \|_{\Lebesgue^{\Orlicz}(\Omega)} \neq 0$ needs some preparations.
    
    Let $E \subseteq \Omega$ be a set of positive measure.
    Let $\alpha > 0$. 
    For every $i \in \bbN$ we write 
    \begin{align*}
        E_{i,\alpha} := \left\{ x \in X \suchthat* \Orlicz(x,i) \geq \alpha \right\}.
    \end{align*}
    Suppose that all $E_{i,\alpha}$ have measure zero;
    the countable union of sets with measure has measure zero again,
    and so for almost all $x \in \Omega$ it holds that for all $i \in \bbN$ we have $\Orlicz(x,i) < \alpha$.
    But since $\Orlicz(x,\xi)$ vanishes at $\xi=0$ and is convex and non-decreasing in $\xi$ almost everywhere
    we must have $\Orlicz(x,\cdot) = 0$ almost everywhere.
    This contradicts our assumptions on $\Orlicz$, 
    and we conclude that for every $\alpha > 0$ there exists $i \in \bbN$ such that $E_{i,\alpha}$ has positive measure.
    
    Let now $u : \Omega \rightarrow \VEC$ be measurable and assume it is not zero almost everywhere.
    Then there exists $\mu > 0$ such that $|u(x)| \geq \mu$ over a set $E \subseteq \OmegaS$ of positive measure. 
    If $\lambda > 0$ is such that $\Orlicz\left( x, \lambda^{-1} |u(x)| \right)$ has integral at most $1$, 
    then 
    \begin{align*}
        \int_{\Omega} \Orlicz\left( x, \frac{ \mu \chi_{E} }{\lambda} \right)
        \leq
        \int_{\Omega} \Orlicz\left( x, \frac{ |u(x)| }{\lambda} \right)
        \leq
        1.
    \end{align*}
    Fix any $\alpha > 0$.
    There exists $i \in \bbN$ such that $E_{i,\alpha} \subseteq E$ has positive measure.
    We show that assuming $\|u\|_{\Lebesgue^{\Orlicz}(\Omega)} = 0$ leads to a contradiction. 
    We pick $\lambda > 0$ so small that $\mu/\lambda \geq i$, so that
    \begin{align*}
        0
        < 
        \alpha \cdot |E_{i,\alpha}|
        \leq 
        \int_{\Omega} \Orlicz\left( x, \frac{ \mu \chi_{E} }{\lambda} \right)
        .
    \end{align*}
    Since we assume $\|u\|_{\Lebesgue^{\Orlicz}(\Omega)} = 0$, the integral on the right-hand side is bounded by $1$.
    On the other hand, for any $\beta \in (0,1)$ we use convexity once more to find  
    \begin{align*}
        \int_{\Omega} \Orlicz\left( x, \frac{ \mu \chi_{E} }{\lambda} \right)
        =
        \int_{\Omega} \Orlicz\left( x, \beta \frac{ \mu \chi_{E} }{\beta\lambda} \right)
        \leq
        \beta 
        \int_{\Omega} \Orlicz\left( x, \frac{ \mu \chi_{E} }{\beta\lambda} \right)
        \leq
        \beta 
        \int_{\Omega} \Orlicz\left( x, \frac{ |u(x)| }{\beta\lambda} \right)
        \leq 
        \beta
    \end{align*}
    This leads to the desired contradiction. Hence $\|u\|_{\Lebesgue^{\Orlicz}(\Omega)} > 0$ whenever $u \neq 0$.
\end{proof}

\begin{proposition}
    The space $\Lebesgue^{\Orlicz}(\Omega,\VEC)$ is complete. 
\end{proposition}
\begin{proof}
    Suppose that $w_{n} \in \Lebesgue^{\Orlicz}(\Omega,\VEC)$ is a Cauchy sequence,
    i.e., for every $\epsilon > 0$ there exists $N \geq 0$ 
    such that $\| w_{n} - w_{m} \|_{\Lebesgue^{\Orlicz}(\Omega)} < \epsilon$ for all $m, n \geq N$. 
    It suffices to show that a subsequence of that Cauchy sequence converges. 
    We can pick a subsequence $u_{n} \in \Lebesgue^{\Orlicz}(\Omega,\VEC)$
    such that 
    $\| u_{n} - u_{m} \|_{\Lebesgue^{\Orlicz}(\Omega)} < 2^{-n}$ for all $m, n \in \bbN$ with $m \geq n$. 
    
    We define $g_n := u_{n+1}-u_{n}$.
    By assumption, $\| u_{n+1} - u_{n} \|_{\Lebesgue^{\Orlicz}(\Omega)} < 2^{-n}$. Hence 
    \begin{align*}
        \left\| \sum_{n=1}^{N} |g_n| \right\|_{\Lebesgue^{\Orlicz}(\Omega)}
        \leq 
        \sum_{n=1}^{N} \left\| |g_n| \right\|_{\Lebesgue^{\Orlicz}(\Omega)}
        \leq 
        \sum_{n=1}^{N} 2^{-n}
        \leq 
        1.
    \end{align*}
    We conclude that there exists $\lambda_{0} > 1$ such that for all $N$
    \begin{align*}
        \int_{\Omega} \Orlicz\left( x, \lambda_{0}^{-1} \sum_{n=1}^{N} |g_n(x)| \right) \leq 1.
    \end{align*}
    We use Fatou's lemma and derive 
    \begin{align*}
        1 
        &\geq 
        \liminf_{N \rightarrow \infty}
        \int_{\Omega} \Orlicz\left( x, \lambda_{0}^{-1} \sum_{n=1}^{N} |g_n(x)| \right)
\geq 
        \int_{\Omega} \liminf_{N \rightarrow \infty} \Orlicz\left( x, \lambda_{0}^{-1} \sum_{n=1}^{N} |g_n(x)| \right)
        .
    \end{align*}
    Since the integrand is left-continuous and non-decreasing in $N$, the limes inferior is a limit.
    Since $\Orlicz\left( x, \cdot \right)$ is left-continuous, we find almost everywhere:
    \begin{align*}
        \int_{\Omega} \liminf_{N \rightarrow \infty} \Orlicz\left( x, \lambda_{0}^{-1} \sum_{n=1}^{N} |g_n(x)| \right)
        =
        \int_{\Omega} \Orlicz\left( x, \lambda_{0}^{-1} \sum_{n=1}^{\infty} |g_n(x)| \right)
        .
    \end{align*}
    We conclude that the series $\sum_{n=1}^{\infty} g_n(x)$ is absolutely convergent almost everywhere,
    and hence it is convergent almost everywhere. 
    Hence there exists a measurable function $G : \Omega \rightarrow \VEC$
    that is the pointwise limit of that series almost everywhere. 
    We find 
    \begin{align*}
        \int_{\Omega} \Orlicz\left( x, \lambda_{0}^{-1} \left|G(x)\right| \right)
        = 
        \int_{\Omega} \Orlicz\left( x, \lambda_{0}^{-1} \left|\sum_{n=1}^{\infty} g_n(x)\right| \right)
        \leq 
        \int_{\Omega} \Orlicz\left( x, \lambda_{0}^{-1} \sum_{n=1}^{\infty} |g_n(x)| \right)
        \leq 
        1.
    \end{align*}
    Hence $G \in \Lebesgue^{\Orlicz}$ with $\|G\|_{\Lebesgue^{\Orlicz}(\Omega)} \leq \lambda_{0}$.
    For almost every $x \in \Omega$ we have 
    \begin{align*}
        G(x) = \lim_{N\rightarrow\infty} \sum_{n=1}^{N} g_n(x) = \lim_{N\rightarrow\infty} u_{N+1} - u_{1}.
    \end{align*}
    So $u(x) = G(x) + u_{1}(x)$ is the pointwise limit of $u_{n}$ almost everywhere. 
    For any $\lambda > 0$ we use left-continuity and observe   
    \begin{align*}
        \int_{\Omega} \Orlicz\left( x, \lambda^{-1} \left|u(x)-u_{n}(x)\right| \right)
        &
        =
        \int_{\Omega} \Orlicz\left( x, \lambda^{-1} \left|\lim_{m\rightarrow\infty}u_{m}(x)-u_{n}(x)\right| \right)
        \\&
        =
        \int_{\Omega} \lim_{m\rightarrow\infty} \Orlicz\left( x, \lambda^{-1} \left|u_{m}(x)-u_{n}(x)\right| \right)
        .
    \end{align*}
    Next, by Fatou's lemma:
    \begin{align*}
        \int_{\Omega} \lim_{m\rightarrow\infty} \Orlicz\left( x, \lambda^{-1} \left|u_{m}(x)-u_{n}(x)\right| \right)
        \leq 
        \liminf_{m\rightarrow\infty} \int_{\Omega} \Orlicz\left( x, \lambda^{-1} \left|u_{m}(x)-u_{n}(x)\right| \right)
        .
    \end{align*}
    The last limes inferior of integrals is bounded by $1$ if $\lambda > 2^{-n}$.
    Consequently, $\| u-u_{n} \|_{\Lebesgue^{\Orlicz}(\Omega)} < 2^{-n}$.
    In other words, $u_{n}$ converges towards $u$ in the space $\Lebesgue^{\Orlicz}(\Omega,\VEC)$. 
\end{proof}

We commence the study of pullback formulas for Orlicz spaces.
This has many parallels to the previous section, but it seems most natural to work with a notion of \emph{pullback} for Musielak-Orlicz integrands. 

Suppose that $\phi : \OmegaS \rightarrow \OmegaT$ is continuous, invertible, and with locally Lipschitz inverse.
When $\Orlicz$ is a Musielak-Orlicz integrand over $\OmegaT$,
then we define the \emph{pullback} of $\Orlicz$ along $\phi$ as
\begin{align}
    \phi^{\ast} \Orlicz( x, \xi ) 
    = 
    \Orlicz( \phi(x), \xi ) \left|\det\nabla\phi_{|x}\right|
    ,
    \quad 
    x \in \OmegaS.
\end{align}
Thus we can formalize the following;
see also similarly~\cite[Proposition~9.3.7]{diening2011lebesgue} for the variable exponent case. 

\begin{lemma}\label{lemma:transformationssatz:orlicz}
    Let $\OmegaS, \OmegaT \subseteq \bbR^{N}$ be open sets, and let $\VEC$ be a Banach space.
    Suppose that $\phi : \OmegaS \rightarrow \OmegaT$ is continuous 
    and has a locally Lipschitz inverse $\phi^{-1} : \OmegaT \rightarrow \OmegaS$. 
    Let $\Orlicz$ be a Musielak-Orlicz integrand over $\OmegaT$. 
    If $u \in \Lebesgue^{\Orlicz}(\OmegaT,\VEC)$,
    then $u \circ \phi \in \Lebesgue^{\phi^{\ast}\Orlicz}(\OmegaS,\VEC)$ with 
    \begin{align}\label{math:transformationssatz:orlicz:scalar}
        \left\| 
            u \circ \phi
        \right\|_{\Lebesgue^{\phi^{\ast}\Orlicz}(\OmegaS,\VEC)}
        = 
        \left\| 
            {u} 
        \right\|_{\Lebesgue^{\Orlicz}(\OmegaT,\VEC)}
        .
    \end{align}
    Suppose in addition that $\phi$ is locally Lipschitz. 
    If  $p \in [1,\infty]$ and $u \in \Lebesgue^{\Orlicz}(\OmegaT,\LIN^{d}(Y,\VEC))$,
    then $\phi^\ast u \in \Lebesgue^{\phi^{\ast}\Orlicz}(\OmegaS,\LIN^{d}(X,\VEC))$ with 
    \begin{align}\label{math:transformationssatz:orlicz:tensor}
        \left\| 
            \phi^\ast u
        \right\|_{\Lebesgue^{\phi^{\ast}\Orlicz}(\OmegaS,\VEC)}
        \leq 
        \left\| 
            {u} 
        \right\|_{\Lebesgue^{\Orlicz}(\OmegaT,\VEC)}
        \left\| 
            \nabla \phi
        \right\|_{\Lebesgue^{\infty}(\OmegaS )}^{d}
        .
    \end{align}
\end{lemma}
\begin{proof}
    By an obvious application of the transformation theorem,
    \begin{align*}
        \int_{\OmegaT} \Orlicz\left( y,\frac{|u(y)|}{\lambda} \right) \,\dif y
        &=
        \int_{\OmegaS} \Orlicz\left( \phi(x),\frac{|u\circ\phi(x)|}{\lambda} \right) \cdot |\det\nabla\phi_{|x}| \,\dif x
    \end{align*}
    for any measurable function $u : \OmegaT \rightarrow \VEC$. 
    Thus~\eqref{math:transformationssatz:orlicz:scalar} follows from the definition of the Luxemburg norm.
    Next,~\eqref{math:transformationssatz:orlicz:tensor} follows by 
    \begin{align*}
        \left\| 
            \phi^{\ast} {u} 
        \right\|_{\Lebesgue^{\phi^{\ast}\Orlicz}(\OmegaS,\VEC)}
        \leq 
        \left\| 
            |\nabla\phi|^{d} \cdot 
            ({u} \circ \phi)
        \right\|_{\Lebesgue^{\phi^{\ast}\Orlicz}(\OmegaS,\VEC)}
        \leq 
        \left\| 
            \nabla \phi
        \right\|_{\Lebesgue^{\infty}(\OmegaS )}^{d}
        \left\| 
            {u} \circ \phi
        \right\|_{\Lebesgue^{\Orlicz}(\OmegaS,\VEC)},
    \end{align*}
    where we have used~\eqref{math:auxiliary:pointwise} and the homogeneity of the Luxemburg norm. 
\end{proof}

We are interested in \emph{Musielak-Orlicz-Sobolev} spaces, which just call \emph{Orlicz-Sobolev spaces}. 
The \emph{$m$-th order Orlicz-Sobolev space} $W^{m,\Orlicz}(\Omega,\VEC)$ denote the set of locally integrable functions with weak derivatives up to order $m$, all of which are in the Orlicz-Sobolev space $\Lebesgue^{\Orlicz}(\Omega,\VEC)$. 
The pullback of functions in Orlicz-Sobolev spaces is again in a Orlicz-Sobolev space, but with a different Musielak-Orlicz integrand.

\begin{corollary}\label{corollary:weakderivativeofpullback:orlicz}
    Let $\OmegaS, \OmegaT \subseteq \bbR^{N}$ be open sets, and let $\VEC$ be a Banach space.
    Let $m \in \bbN$. 
    Suppose that $\phi : \OmegaS \rightarrow \OmegaT$ is locally bi-Lipschitz.
    Let $\Orlicz$ be a Musielak-Orlicz integrand. 
    
    \noindent 
    If the weak derivatives $\nabla\phi, \dots, \nabla^{m+1}\phi$ are essentially bounded
    and $u \in \Sobolev^{m,\Orlicz}(\OmegaT,\LIN^{d}(Y,\VEC))$,
    then $\phi^\ast u \in \Sobolev^{m,\phi^{\ast}\Orlicz}(\OmegaS,\LIN^{d}(X,\VEC))$ and 
    \begin{align}
        \begin{aligned}
        \left\| 
            \nabla^{m} 
            \left( \phi^{\ast} {u} \right)
        \right\|_{\Lebesgue^{\phi^\ast\Orlicz}(\OmegaS)}
        &
        \leq 
        \sum_{k=0}^{m}
        \|
            \nabla^{k} u 
        \|_{\Lebesgue^{\Orlicz}(\OmegaT)}
\left\|
        B_{k,m,d}\left(
            | \nabla^{} \phi |,
            \ldots,
            | \nabla^{m+1} \phi |
        \right)
        \right\|_{\Lebesgue^{\infty}(\OmegaS)}
        \\&
        \leq 
        \sum_{k=0}^{m}
        \|
            \nabla^{k} u 
        \|_{\Lebesgue^{\Orlicz}(\OmegaT)}
        B_{k,m,d}\left(
            \| \nabla^{   } \phi \|_{\Lebesgue^{\infty}(\OmegaS)},
            \ldots,
            \| \nabla^{m+1} \phi \|_{\Lebesgue^{\infty}(\OmegaS)}
        \right)
        .
        \end{aligned}
    \end{align}
    If the weak derivatives $\nabla\phi, \dots, \nabla^{m}\phi$ are essentially bounded 
    and $u \in \Sobolev^{m,\Orlicz}(\OmegaT,\VEC)$,
    then $u \circ \phi \in \Sobolev^{m,\phi^{\ast}\Orlicz}(\OmegaS,\VEC)$ and 
    \begin{align}
        \begin{aligned}
        \left\| 
        \nabla^{m} 
        \left( {u} \circ \phi \right)
        \right\|_{\Lebesgue^{\phi^\ast\Orlicz}(\OmegaS)}
        &
        \leq 
        \sum_{k=1}^{m}
        \|
            \nabla^{k} u 
        \|_{\Lebesgue^{\Orlicz}(\OmegaT)}
\left\|
        B_{k,m}\left(
            | \nabla^{} \phi |,
            \ldots,
            | \nabla^{m} \phi |
        \right)
        \right\|_{\Lebesgue^{\infty}(\OmegaS)}
        \\&
        \leq 
        \sum_{k=1}^{m}
        \|
            \nabla^{k} u 
        \|_{\Lebesgue^{\Orlicz}(\OmegaT)}
        B_{k,m}\left(
            \| \nabla^{ } \phi \|_{\Lebesgue^{\infty}(\OmegaS)},
            \ldots,
            \| \nabla^{m} \phi \|_{\Lebesgue^{\infty}(\OmegaS)}
        \right)
        .
        \end{aligned}
    \end{align}
\end{corollary}

\begin{proof}
    This is essentially the same proof as for Corollary~\ref{corollary:weakderivativeofpullback:sobolev}.
\end{proof}

We study the following notion of fractional Musielak-Orlicz spaces.
Given any Musielak-Orlicz integrand $\DoubleOrlicz$ over the product space $\Omega \times \Omega$,
the product $\dist(x,x')^{-N}\DoubleOrlicz$ defines another such Musielak-Orlicz integrand. 
Given $\theta \in (0,1)$, we define the \emph{Orlicz-Slobodeckij seminorm} 
\begin{align}\label{math:orliczslobodeckijseminorm}
    |u|_{\Sobolev^{\theta,\DoubleOrlicz}(\OmegaS)}
    :=
    \inf_{ \lambda > 0 }
    \left\{ 
        \lambda 
        \suchthat* 
        \iint_{\OmegaS \times \OmegaS} 
        \DoubleOrlicz\left( x, x', \frac{|u(x)-u(x')|}{\dist(x,x')^{\theta}} \right) \dist(x,x')^{-N}
        \,\dif x \,\dif x'
        \leq 1
    \right\}
\end{align}
for any measurable function $u : \OmegaS \rightarrow \VEC$.

Finally, whenever $\Orlicz$ and $\DoubleOrlicz$ are Musielak-Orlicz integrands over $\OmegaS$ and $\OmegaS\times\OmegaS$, respectively,
we write $\Sobolev^{\Orlicz,\theta,\DoubleOrlicz}(\OmegaS,\VEC)$ for the subspace $\Lebesgue^{\Orlicz}(\OmegaS)$ whose members have bounded seminorm~\eqref{math:orliczslobodeckijseminorm}. 
More generally, when $m \in \bbN_{0}$, we write $\Sobolev^{m,\Orlicz,\theta,\DoubleOrlicz}(\OmegaS,\VEC)$ for the subspace $\Sobolev^{m,\Orlicz}(\OmegaS)$ whose weak derivatives up to order $m$ are in $\Sobolev^{\Orlicz,\theta,\DoubleOrlicz}(\OmegaS)$.

\begin{remark}
    In principle, there is no canonical relation between the functions $\Orlicz$ and $\DoubleOrlicz$.
    However, the following estimates require an additional relation between the two functions and we comment further below in how far those assumptions are realistic. 
\end{remark}

For handling the change of variables, we once more need a notion of pullback.
Suppose that $\DoubleOrlicz$ is a Musielak-Orlicz integrand over the product space $\Omega \times \Omega$. 
Every transformation $\phi : \OmegaS \rightarrow \OmegaT$ defines a transformation between the corresponding product spaces, 
and we define the \emph{pullback} of $\DoubleOrlicz$ along $\phi$ as
\begin{align}
    \phi^{\ast} \DoubleOrlicz( x, x', \xi ) 
    = 
    \DoubleOrlicz( \phi(x), \phi(x'), \xi ) \left|\det\nabla\phi_{|\phi(x)}\right| \left|\det\nabla\phi_{|\phi(x')}\right|
    .
\end{align}
We first study how the seminorm behaves under multiplication with H\"older continuous functions and under change of variables. 

\begin{lemma}\label{lemma:orliczsobolevslobodeckijproduct}
    Let $\OmegaS \subseteq \bbR^{N}$ be an open set and let $\VEC_1, \VEC_2, \VEC_3$ be Banach spaces. 
    Let $\theta,\sigma \in (0,1]$ with $\theta < \sigma$. 
    Let $\DoubleOrlicz$ be a Musielak-Orlicz integrand over $\OmegaS\times\OmegaS$
    and let $\Orlicz$ be a Musielak-Orlicz integrand over $\OmegaS$ such that $\DoubleOrlicz(x,x') \leq C_{\Orlicz,\DoubleOrlicz} \Orlicz(x)$
    for almost all $x, x' \in \OmegaS$.
    
    If ${u} \in \Sobolev^{\Orlicz,\theta,\DoubleOrlicz}(\OmegaS,\LIN(\VEC_1,\VEC_2))$
    and $\psi \in \Cont^{0,\sigma}(\OmegaS,\LIN(\VEC_2,\VEC_3))$, then 
    \noindent 
    \begin{align}\label{math:orliczsobolevslobodeckijproduct}
        \begin{aligned}
            | {u} \contract \psi |_{\Sobolev^{\theta,\DoubleOrlicz}(\OmegaS)}
            &
            \leq 
            | {u} |_{\Sobolev^{\theta,\DoubleOrlicz}(\OmegaS)}
            \| \psi \|_{\Lebesgue^{\infty}(\OmegaS)}
            \\&\quad\quad
            +
            C_{\Orlicz,\DoubleOrlicz}
            \frac{ \SphereArea_{N-1} }{ \sigma-\theta }
            \| {u} \|_{\Lebesgue^{\Orlicz}(\OmegaS)}
            | \psi |_{\Cont^{0,\sigma}(\OmegaS)}
            +
            2
            C_{\Orlicz,\DoubleOrlicz}
            \frac{ \SphereArea_{N-1} }{ \theta }
            \| {u} \|_{\Lebesgue^{\Orlicz}(\OmegaS)}
            \| \psi \|_{\Lebesgue^{\infty}(\OmegaS)}
            .
        \end{aligned}
    \end{align}        
    Here, $\SphereArea_{N-1}$ denotes the area of the Euclidean unit sphere of dimension $N-1$. 
\end{lemma}

\begin{proof}
    We first observe that for almost all $(x,x') \in \Omega\times\Omega$ we have 
    \begin{align*}
        | u(x) \contract \psi(x) - u(x') \contract \psi(x') |
        &
        \leq 
        | u(x) \contract \psi(x) - u(x') \contract \psi(x) |
        +
        | u(x') \contract \psi(x) - u(x') \contract \psi(x') |
        \\&
        \leq 
        | u(x) - u(x') | \cdot \| \psi \|_{\Lebesgue^{\infty}}
+
        | u(x') \contract \psi(x) - u(x') \contract \psi(x') |
.
    \end{align*}
    Using the Minkowski inequality for the Luxemburg norm over $\OmegaS \times \OmegaS$,
    we see
    \begin{align*}
        &
        | {u} \contract \psi |_{\Sobolev^{\theta,\DoubleOrlicz}(\OmegaS)}
        \leq 
        | {u} |_{\Sobolev^{\theta,\DoubleOrlicz}(\OmegaS)}
        \| \psi \|_{\Lebesgue^{\infty}(\OmegaS)}
        \\&\qquad 
        +
        \inf_{\lambda > 0}
        \left\{ 
            \lambda
            \suchthat*
            \iint_{\OmegaS \times \OmegaS}
            \DoubleOrlicz\left( x, x',
            \frac{
                | u(x') \contract \psi(x) - u(x') \contract \psi(x') |
            }{
                \lambda \dist( x, x' )^{\theta} 
            } \right)
            \dist( x, x' )^{-N}
            \,\dif x\,\dif x'
        \right\}
        .
    \end{align*}
    By an argument involving the triangle inequality, 
    we can split the integral.
    On the one hand,
    \begin{align*}
        &
        \iint\limits_{ \substack{ \OmegaS \times \OmegaS \\ \dist(x,x') < 1 } }
            \DoubleOrlicz\left( x, x',
                \frac{
                    | u(x') \contract \psi(x) - u(x') \contract \psi(x') |
                }{
                    \lambda \dist( x, x' )^{\theta} 
                } 
            \right) \dist( x, x' )^{-N}
        \,\dif x \,\dif x'
        \\&\quad
        \leq
        \iint\limits_{ \substack{ \OmegaS \times \OmegaS \\ \dist(x,x') < 1 } }
            \DoubleOrlicz\left( x, x',
                \lambda^{-1} 
                | u(x') | \cdot | \psi |_{\Cont^{0,\sigma}(\OmegaS)} \dist( x, x' )^{\sigma-\theta}
            \right) \dist( x, x' )^{-N}
        \,\dif x \,\dif x'
        \\&\quad
        \leq
        \iint\limits_{ \substack{ \OmegaS \times \OmegaS \\ \dist(x,x') < 1 } }
            \DoubleOrlicz\left( x, x',
                \lambda^{-1}
                | u(x') | \cdot | \psi |_{\Cont^{0,\sigma}(\OmegaS)}
            \right)  
            \dist( x, x' )^{\sigma-\theta-N}
        \,\dif x \,\dif x'
        \\&\quad
        \leq
        C_{\Orlicz,\DoubleOrlicz}
        \frac{ \SphereArea_{N-1} }{ \sigma - \theta }
        \cdot 
        \int\limits_{ \substack{ \OmegaS } }
            \Orlicz\left( x,
                \lambda^{-1}
                | u(x) | \cdot | \psi |_{\Cont^{0,\sigma}(\OmegaS)}
            \right)  
        \,\dif x
        .
    \end{align*}
    On the other hand, 
    \begin{align*}
        &
        \iint\limits_{ \substack{ \OmegaS \times \OmegaS \\ \dist(x,x') > 1 } }
            \DoubleOrlicz\left( x, x',
                \frac{
                    | u(x') \contract \psi(x) - u(x') \contract \psi(x') |
                }{
                    \lambda \dist( x, x' )^{\theta} 
                } 
            \right) \dist( x, x' )^{-N}
        \,\dif x \,\dif x'
        \\&\quad
        \leq
        \iint\limits_{ \substack{ \OmegaS \times \OmegaS \\ \dist(x,x') > 1 } }
            \DoubleOrlicz\left( x, x',
                \frac{
                    | u(x') | \cdot 2 \| \psi \|_{\Lebesgue^{\infty}(\OmegaS)}
                }{
                    \lambda \dist( x, x' )^{\theta}
                } 
            \right) \dist( x, x' )^{-N}
        \,\dif x \,\dif x'
        \\&\quad
        \leq
        \iint\limits_{ \substack{ \OmegaS \times \OmegaS \\ \dist(x,x') > 1 } }
            \DoubleOrlicz\left( x, x',
                \lambda^{-1}
                | u(x') | \cdot 2 \| \psi \|_{\Lebesgue^{\infty}(\OmegaS)}
            \right)  
            \dist( x, x' )^{-\theta-N}
        \,\dif x \,\dif x'
        \\&\quad
        \leq
        C_{\Orlicz,\DoubleOrlicz}
        \frac{ \SphereArea_{N-1} }{ \theta }
        \cdot 
        \int\limits_{ \substack{ \OmegaS } }
            \Orlicz\left( x,
                \lambda^{-1}
                | u(x) | \cdot | \psi |_{\Cont^{0,\sigma}(\OmegaS)}
            \right)  
        \,\dif x
        .
    \end{align*}
    This shows~\eqref{math:sobolevslobodeckijproduct}, thus completing the proof of the auxiliary lemma. 
\end{proof}

With that preparation in place, 
we estimate the Sobolev-Slobodeckij seminorms of tensors and their higher derivatives after a pullback.

\begin{lemma}\label{lemma:transformationssatz:orliczslobodeckij}
    Let $\OmegaS, \OmegaT \subseteq \bbR^{N}$ be open sets, and let $\VEC$ be a Banach space.
    Suppose that $\phi : \OmegaS \rightarrow \OmegaT$ is bi-Lipschitz. 
    Let $\DoubleOrlicz$ be a Musielak-Orlicz integrand over $\OmegaT\times\OmegaT$
    and let $\Orlicz$ be a Musielak-Orlicz integrand over $\OmegaT$ such that $\DoubleOrlicz(x,x') \leq C_{\Orlicz,\DoubleOrlicz} \Orlicz(x)$
    for almost all $x, x' \in \OmegaS$.
    Let $\theta \in (0,1]$.
    If $u \in \Sobolev^{\Orlicz,\theta,\DoubleOrlicz}(\OmegaT,\VEC)$, 
    then $u \circ \phi \in \Sobolev^{\phi^{\ast}\Orlicz,\theta,\phi^{\ast}\DoubleOrlicz}(\OmegaS,\VEC)$ with
    \begin{align}\label{math:transformationssatz:orliczslobodeckij:scalar}
        \left| u \circ \phi \right|_{\Sobolev^{\theta,\phi^{\ast}\DoubleOrlicz}(\OmegaS,\VEC)}
        \leq 
        \left|\phi\right|_{\Cont^{0,1}(\OmegaS)}^{\theta+N}
        \left| u \right|_{\Sobolev^{\theta,\DoubleOrlicz}(\OmegaT,\VEC)}
        .
    \end{align}
    Now assume that $\theta < \sigma$. 
    If $u \in \Sobolev^{\Orlicz,\theta,\DoubleOrlicz}(\OmegaT,\LIN^{d}(Y,\VEC))$, 
    then $\phi^{\ast} u \in \Sobolev^{\phi^{\ast}\Orlicz,\theta,\phi^{\ast}\DoubleOrlicz}(\OmegaS,\LIN^{d}(X,\VEC))$ with
    \begin{align}\label{math:transformationssatz:orliczslobodeckij:tensor}
        \begin{aligned}
            | \phi^{\ast} u |_{\Sobolev^{\theta,\phi^{\ast}\DoubleOrlicz}(\OmegaS)}
            &
            \leq 
            \left|\phi\right|_{\Cont^{0,1}(\OmegaS)}^{\theta+N}
            | {u} |_{\Sobolev^{\theta,\DoubleOrlicz}(\OmegaT)}
            \| {\otimes}^d \nabla\phi \|_{\Lebesgue^{\infty}(\OmegaS)}
            \\&\quad\quad
            +
            C_{\Orlicz,\DoubleOrlicz}
            \frac{ \SphereArea_{N-1} }{ \sigma-\theta }
            \| {u} \|_{\Lebesgue^{\Orlicz}(\OmegaT)}
            | {\otimes}^d \nabla\phi |_{\Cont^{0,\sigma}(\OmegaS)}
            \| \det\nabla\phi \|_{\Lebesgue^{\infty}(\OmegaS)}
            \\&\quad\quad
            +
            2
            C_{\Orlicz,\DoubleOrlicz}
            \frac{ \SphereArea_{N-1} }{ \theta }
            \| {u} \|_{\Lebesgue^{\Orlicz}(\OmegaT)}
            \| {\otimes}^d \nabla\phi \|_{\Lebesgue^{\infty}(\OmegaS)}
            \| \det\nabla\phi \|_{\Lebesgue^{\infty}(\OmegaS)}
            .
        \end{aligned}
    \end{align}    
\end{lemma}

\begin{proof}
    Let $u \in \Sobolev^{\Orlicz,\theta,\DoubleOrlicz}(\OmegaT,\VEC)$.
    Let $\lambda > 0$. 
    A change of variables shows 
    \begin{align*}
        &
        \iint_{\OmegaS \times \OmegaS} 
        \Orlicz\left( \phi(x), \phi(x'), \frac{|u(\phi(x))-u(\phi(x)')|}{\lambda\dist(x,x')^{\theta}} \right) |\det\nabla\phi_{|x}| \cdot |\det\nabla\phi_{|x'}| \cdot \dist(x,x')^{-N}
        \,\dif x \,\dif x'
        \\&
        \leq 
        \left|\phi\right|_{\Cont^{0,1}(\OmegaS)}^{N}
        \iint_{\OmegaS \times \OmegaS} 
        \Orlicz\left( \phi(x), \phi(x'), \frac{|u(\phi(x))-u(\phi(x)')|}{\lambda \left|\phi\right|_{\Cont^{0,1}(\OmegaS)}^{-\theta}|\phi(x)-\phi(x')|^{\theta}} \right) 
        \frac{ |\det\nabla\phi_{|x}| \cdot |\det\nabla\phi_{|x'}| }{|\phi(x)-\phi(x')|^{N}}
        \,\dif x \,\dif x'
        \\&
        =
        \left|\phi\right|_{\Cont^{0,1}(\OmegaS)}^{N}
        \iint_{\OmegaS \times \OmegaS} 
        \Orlicz\left( x, x', \frac{|u(x)-u(x')|}{\lambda \left|\phi\right|_{\Cont^{0,1}(\OmegaS)}^{-\theta}\dist(x,x')^{\theta}} \right) \cdot \dist(x,x')^{-N}
        \,\dif x \,\dif x'
        .
    \end{align*}
    Recalling the definition of the Luxemburg norm, we immediately conclude~\eqref{math:transformationssatz:orliczslobodeckij:scalar}.
    
    Next, let $u \in \Sobolev^{\Orlicz,\theta,\DoubleOrlicz}(\OmegaT,\LIN^{d}(Y,\VEC))$.
    Then $\phi^{\ast} {u} = ( u \circ \phi ) \contract ( \otimes^{d} \nabla\phi )$, hence
    \begin{align*}
        \left| \phi^{\ast} {u} \right| 
        \leq 
        \left| u \circ \phi \right|
        \cdot
        \left| \otimes^{d} \nabla\phi \right|
        .
    \end{align*}
    We also find that 
    \begin{align*}
        \phi^{\ast}\DoubleOrlicz(x,x',\xi)
        &
        =
        \DoubleOrlicz(\phi(x),\phi(x'),\xi) |\det\nabla\phi_{|x}| \cdot |\det\nabla\phi_{|x'}|
        \\&
        \leq 
        C_{\Orlicz,\DoubleOrlicz}
        \Orlicz(\phi(x),\xi) |\det\nabla\phi_{|x}| \| \det\nabla\phi \|_{\Lebesgue^{\infty}(\OmegaS)}
        = 
        C_{\Orlicz,\DoubleOrlicz}
        \| \det\nabla\phi \|_{\Lebesgue^{\infty}(\OmegaS)}
        \phi^{\ast}\Orlicz(x,\xi)
        .
    \end{align*}
    We use~\eqref{math:transformationssatz:orliczslobodeckij:scalar} and Lemma~\ref{lemma:orliczsobolevslobodeckijproduct} to deduce~\eqref{math:transformationssatz:orliczslobodeckij:tensor}.
    The proof is complete. 
\end{proof}

\begin{proposition}\label{prop:pullbackoftensors:orliczslobodeckij}
    Let $\OmegaS, \OmegaT \subseteq \bbR^{N}$ be open sets, and let $\VEC$ be a Banach space.
    Let $m \in \bbN$. 
    Suppose that $\phi : \OmegaS \rightarrow \OmegaT$ is locally bi-Lipschitz. 
    Suppose that the weak derivatives $\nabla\phi, \dots, \nabla^{m+1}\phi$ are in $\Sobolev^{\sigma,\infty}(\OmegaS)$ for some $\sigma \in (0,1]$.
    Let $\theta \in (0,1]$ with $\theta < \sigma$. 
    Let $\DoubleOrlicz$ be a Musielak-Orlicz integrand over $\OmegaS\times\OmegaS$
    and let $\Orlicz$ be a Musielak-Orlicz integrand over $\OmegaS$ such that $\DoubleOrlicz(x,x') \leq C_{\Orlicz,\DoubleOrlicz}\Orlicz(x)$
    for almost all $x, x' \in \OmegaS$.
    
    If $u \in \Sobolev^{m,\Orlicz,\theta,\DoubleOrlicz}(\OmegaT,\LIN^{d}(Y,\VEC))$, 
    then $\phi^\ast u \in \Sobolev^{m,\phi^{\ast}\Orlicz,\theta,\phi^{\ast}\DoubleOrlicz}(\OmegaS,\LIN^{d}(X,\VEC))$,
    and we have 
    \begin{align*}
        &
        \left|
            \nabla^{m} (\phi^{\ast} {u} )
        \right|_{\Sobolev^{\theta,\phi^{\ast}\DoubleOrlicz}(\OmegaS)}
        \\&\quad 
        \leq 
        \sum_{ 0 \leq k \leq m }
        \left| 
            \nabla^{k} u 
        \right|_{\Sobolev^{\theta,\DoubleOrlicz}(\OmegaT)}
        \left\|
            B_{m,k,d}\left( |\nabla\phi|, \ldots, |\nabla^{m+1}\phi| \right)
        \right\|_{\Lebesgue^{\infty}(\OmegaS)}
        \\&\quad 
        + 
        C_{\Orlicz,\DoubleOrlicz}
        \frac{ \SphereArea_{N-1} }{ \sigma-\theta }
        \sum_{ 0 \leq k \leq m }
        \left\| 
            \nabla^{k} u 
        \right\|_{\Lebesgue^{\Orlicz}(\OmegaT)}
        \left|
            B_{m,k,d}\left( |\nabla\phi|, \ldots, |\nabla^{m+1}\phi| \right)
        \right|_{\Cont^{0,\sigma}(\OmegaS)}
        \| \det\nabla\phi \|_{\Lebesgue^{\infty}(\OmegaS)}
        \\&\quad 
        + 
        2 
        C_{\Orlicz,\DoubleOrlicz}
        \frac{ \SphereArea_{N-1} }{ \theta }
        \sum_{ 0 \leq k \leq m }
        \left\| 
            \nabla^{k} u 
        \right\|_{\Lebesgue^{\Orlicz}(\OmegaT)}
        \left\|
            B_{m,k,d}\left( |\nabla\phi|, \ldots, |\nabla^{m+1}\phi| \right)
        \right\|_{\Lebesgue^{\infty}(\OmegaS)}
        \| \det\nabla\phi \|_{\Lebesgue^{\infty}(\OmegaS)}
        .
    \end{align*}
    If $d = 0$, then it suffices to assume that the derivatives $\nabla\phi, \dots, \nabla^{m}\phi$ exist almost everywhere and are in $\Sobolev^{\sigma,\infty}(\OmegaS)$. 
\end{proposition}

\begin{proof}
    We use Proposition~\ref{prop:weakderivativesofpullback} together with Lemma~\ref{lemma:orliczsobolevslobodeckijproduct} and Lemma~\ref{lemma:transformationssatz:orliczslobodeckij}.
\end{proof}

\begin{example}
    We point out a few applications where the assumption $\DoubleOrlicz \leq C_{\Orlicz,\DoubleOrlicz} \Orlicz$ is satisfied.
    \begin{itemize}
        \item 
        An obvious special case is $\Orlicz(x,\cdot) = \DoubleOrlicz(x,x',\cdot)$ for almost all $(x,x') \in \Omega\times\Omega$,
        and even more specifically, the case where $\Orlicz$ and $\DoubleOrlicz$ do not depend on $x$.
        \item 
        Another simple example relates to variable exponent spaces. 
        Let $p \in [1,\infty)$ and $\theta \in (0,1)$ such that $sp < N$.
        We abbreviate $q = p \frac{N}{N-sp}$.
        Let $\Omega$ be a bounded Lipschitz domain.
        Then $\Sobolev^{\theta,p}(\Omega,\VEC)$ embeds continuously into $\Lebesgue^{q}(\Omega,\VEC)$.
        
        Suppose that we have $\Orlicz_{0}(x,\xi)=\xi^{p}$ and $\DoubleOrlicz(x,x',\xi)=\xi^{p \delta(x,x')}$,
        where $\delta(x,x') \in [1,\frac{N}{N-sp})$.
        By standard results on variable exponent spaces,
        $\Sobolev^{\Orlicz_{0},\theta,\DoubleOrlicz}(\OmegaS)$ embeds continuously into $\Sobolev^{\theta,p}(\OmegaS)$.
We have $\xi^{q} \geq \xi^{p \delta(x,x')}$ for $\xi \geq 1$ 
        and $\xi^{p} \geq \xi^{p \delta(x,x')}$ for $\xi \leq 1$.
        This suggests the choice $\Orlicz(x,\xi) = \max(\xi^{p},\xi^{q})$.
        
        We want to bound $\|u\|_{\Lebesgue^{\Orlicz}(\Omega)}$ in terms of the Lebesgue norms with exponent $p$ and $q$.
        For any $\lambda > 0$, we estimate 
        \begin{align*}
            \int_{\OmegaS} \calB( \lambda^{-1} |u| ) \,\dif x
            \leq 
            \int_{\Omega} \lambda^{-p} |u(x)|^{p} \,\dif x
            +
            \int_{\Omega} \lambda^{-q} |u(x)|^{q} \,\dif x
        \end{align*}
        The two integrals on the right-hand side are each bounded by $\frac 1 2$ 
        if $\lambda \geq 2^{\frac 1 p}\|u\|_{\Lebesgue^{p}(\Omega,\VEC)}$
        and $\lambda \geq 2^{\frac 1 q}\|u\|_{\Lebesgue^{q}(\Omega,\VEC)}$,
        respectively.
        Hence,
        \begin{align*}
            \|u\|_{\Lebesgue^{\Orlicz}(\Omega)} \leq \max\{ 2^{\frac 1 p}\|u\|_{\Lebesgue^{p}(\Omega,\VEC)}, 2^{\frac 1 q}\|u\|_{\Lebesgue^{q}(\Omega,\VEC)} \},
        \end{align*}
        which is bounded in terms of the norm of $u$ in $\Sobolev^{\Orlicz_{0},\theta,\DoubleOrlicz}(\OmegaS)$.
    \end{itemize}
    From a broader perspective, 
    if the functions $\Orlicz_{0}$ and $\DoubleOrlicz$ are ``close enough'',
    then we expect a generalization of the Sobolev embedding theorem to embed $\Sobolev^{\Orlicz_{0},\theta,\DoubleOrlicz}(\OmegaS)$ into a ``sharper'' Orlicz space $\Lebesgue^{\Orlicz}(\Omega,\VEC)$.
\end{example}

We finish this section with the discussion of dual Musielak-Orlicz integrands and H\"older's inequality. 
Given a Musielak-Orlicz integrand $\Orlicz$, the \emph{dual Musielak-Orlicz integrand} is defined by
\begin{align*}
    \Dual\Orlicz(x,\zeta) := \sup_{ \xi \geq 0 } \xi\zeta - \Orlicz(x,\xi).
\end{align*}
Indeed, one easily verifies that $\Dual\Orlicz(x,\zeta)$ is a Musielak-Orlicz integrand over $\OmegaS$.
Note that this is just the Legendre transform for every $x \in \OmegaS$.

\begin{lemma}
    $\Dual\Orlicz$ is a Musielak-Orlicz integrand.
\end{lemma}
\begin{proof}
    This is easily seen. 
\end{proof}

We can now state the following result~\cite[Lemma~2.6.5]{diening2011lebesgue},
which generalizes H\"older's inequality known for Lebesgue spaces to the setting of Orlicz spaces. 

\begin{proposition}\label{proposition:generalizedhoelder}
    Let $\Orlicz$ be a Musielak-Orlicz integrand.
    Then there exists $C_{\Orlicz,\Dual} \in (0,2]$ such that
    for all $u \in \Lebesgue^{\Orlicz}(\OmegaS)$ and $v \in \Lebesgue^{\Dual\Orlicz}(\OmegaS)$
    we have 
    \begin{align*}
        \int_{\OmegaS} |uv| \dif x
        \leq 
        C_{\Orlicz,\Dual}
        \| u \|_{\Lebesgue^{\Orlicz}(\OmegaS)} 
        \| v \|_{\Lebesgue^{\Dual\Orlicz}(\OmegaS)} 
        .
    \end{align*}
\end{proposition}
\begin{proof}
    The statement is trivial if $u=0$ or $v=0$, so let us assume that both functions are non-zero. 
    Pick $\lambda_{u} > \| u \|_{\Lebesgue^{\Orlicz}(\Omega)}$ and $\lambda_{v} > \|v\|_{\Lebesgue^{\Dual\Orlicz}(\Omega)} $.
    Then, by the definition of $\Dual\Orlicz$, we have 
    \begin{align*}
        \int_{\OmegaS} \frac{u}{\lambda_{u}} \frac{v}{\lambda_{v}}  \dif x
        \leq 
        \int_{\OmegaS} 
        \Orlicz\left( x, \frac{u}{\lambda_{u}} \right)
        +
        \Orlicz\left( x, \frac{v}{\lambda_{v}} \right)
        \dif x
        \leq 2.
    \end{align*}
    Hence,
    \begin{align*}
        \int_{\OmegaS} uv \dif x \leq 2 \lambda_{u} \lambda_{v}.
    \end{align*}
    In the limit, the desired statement follows. 
\end{proof}

Lastly, we study how dualization affects the pullback of Musielak-Orlicz integrands. 
Suppose that $\phi : \OmegaS \rightarrow \OmegaT$ is continuous, invertible, and with locally Lipschitz inverse;
for any Musielak-Orlicz integrand $\Orlicz$ we observe 
\begin{align*}
    \Dual\left( \phi^{\ast}\Orlicz \right)( x, \zeta ) 
    &
    = 
    \sup_{ \xi \geq 0 }
    \xi\zeta 
    - 
    \phi^{\ast}\Orlicz(x,\xi)
    \\&
    = 
    \sup_{ \xi \geq 0 } \xi\zeta - \left|\det\nabla\phi_{|x}\right| \Orlicz(\phi(x),\xi)
    =
    \phi^{\ast} (\Dual\Orlicz)\left( x, \left|\det\nabla\phi_{|x}\right|^{-1} \zeta \right) 
    .
\end{align*}
An easy consequence is this:
if $u \in \Lebesgue^{\Dual\Orlicz}(\OmegaT,\VEC)$, 
then $\phi^{\ast} u \in \Lebesgue^{\phi^{\ast}\Dual\Orlicz}(\OmegaT,\VEC)$,
and we have
\begin{align*}
    \| \phi^{\ast} u \|_{\Lebesgue^{\Dual\phi^{\ast}\Orlicz}}
    \leq 
    \left\|\det\nabla\phi^{-1}\right\|_{\Lebesgue^{\infty}(\OmegaT)}
    \| u \|_{\Lebesgue^{\phi^{\ast}\Dual\Orlicz}(\OmegaS)}
    .
\end{align*}

\begin{example}
    When $p \in [1,\infty]$, then the dual of the Musielak-Orlicz integrand $\Orlicz(x,\xi) = \xi^{p}$ is $\Dual\Orlicz(x,\xi) = \xi^{q}$,
    where $q \in [1,\infty]$ satisfies $1 = 1 - \frac 1 p$. 
    The constant in Proposition~\ref{proposition:generalizedhoelder} equals $1$.
\end{example}

\begin{example}
    The dual of the Musielak-Orlicz integrand $\Orlicz(x,\xi) = \exp(\xi)-1$ is
    \begin{align*}
        \Dual\Orlicz(x,\zeta) 
        = 
        \sup_{\xi \geq 0} \zeta \xi - \exp(\xi)+1
        =
        \zeta \ln(\zeta) - \zeta + 1
        .
    \end{align*}
    We demonstrate Proposition~\ref{proposition:generalizedhoelder} when $\OmegaS = (0,1)$ and $u = v = 1$.
    One calculates 
    \begin{align*}
        \| 1 \|_{\Lebesgue^{\Orlicz}(\Omega)} \leq 1
\;\equivalent\;
        \lambda \geq \frac{1}{\ln(2)}
        ,
        \quad 
        \| 1 \|_{\Lebesgue^{\Dual\Orlicz}(\Omega)} \leq 1
\;\equivalent\; 
        \lambda \geq \frac{1}{e}
    \end{align*}
    By H\"older's inequality, $1 \leq \frac{2}{e \ln(2)} \leq 2 \cdot 0.5308$.
    This example shows that the numerical factor $C_{\Orlicz,\Dual}$ must generally be larger than $1$.
\end{example}

\section{Bounded Variation Spaces and Estimates}\label{sec:bv}

This final section discusses functions with bounded variation 
and how their total variation changes under coordinate transformations. 
Specifically, we consider tensor fields in Lebesgue and Sobolev spaces and 
estimate the \emph{total variation} of the highest derivative of the pullback, building on the chain rule derived above. 
\\

We operate with the same geometric setting as in the preceding section.
We continue to assume that $X = Y = \bbR^{N}$,
and we let $\OmegaS \subseteq X$ be an open set and let $\VEC$ be a normed space. 

Before we discuss the notion of total variation, we briefly review the notion of divergence within our setting. 
When $\vectest \in \Cont^{\infty}_{c}(\OmegaS, X \otimes \VEC^\ast )$ is a smooth vector field with compact support,
which can be uniquely written as 
\begin{align}
    \vectest_{|x} = \sum_{i=1}^{N} e_i \otimes \vectest^{i}_{|x},
    \quad 
    \vectest^{i} \in \Cont^{\infty}_{c}(\OmegaS, \VEC^{\ast})
    ,
\end{align}
then we define the divergence of this vector field as 
\begin{align}
    \divergence\vectest_{|x} := \sum_{i=1}^{N} \nabla_{e_i} \vectest^{i}_{|x}
    .
\end{align}
The notion of weak derivatives is tightly connected with the notion of divergence of vector fields. 
For any weakly differentiable $u \in \Lebesgue^{1}_{loc}(\OmegaS,\VEC)$ we have the relationship 
\begin{align}
    \int_\OmegaS 
    \langle \nabla u_{|x}, \vectest_{|x} \rangle 
    \,\dif{x}
    =
    -
    \int_\OmegaS 
    \langle u_{|x}, \divergence \vectest_{|x} \rangle 
    \,\dif{x}
    .
\end{align}
Moreover, if we are given two vector fields $\vectest_{1} \in \Cont^{\infty}_{c}(\OmegaS, X \otimes \VEC_{1}^\ast )$
and $\vectest_{2} \in \Cont^{\infty}_{c}(\OmegaS, X \otimes \VEC_{2}^\ast )$
with values in the duals of Banach spaces $\VEC_{1}$ and $\VEC_{2}$, 
then the divergence of the tensor product $\vectest_1 \otimes \vectest_2$ is a vector field
in $\Cont^{\infty}_{c}(\OmegaS, X \otimes X \otimes \VEC_{1}^\ast \otimes \VEC_{2}^\ast )$
and it satisfies 
\begin{align}
    \divergence \left( \vectest_{1} \otimes \vectest_{2} \right)
    =
    \left( \divergence \vectest_{1} \right) \otimes \vectest_{2}
    +
    \vectest_{1} \otimes \left( \divergence \vectest_{2} \right)
    .
\end{align} 
This identity extends to tensor products of more than two factors in the obvious way.
\\

We now define the total variation of a function. 
Suppose that $u \in \Lebesgue^{1}_{loc}(\OmegaS,\VEC)$ is locally integrable.
We say that $u$ has \emph{bounded variation} over $\OmegaS$
if there exists a constant $C \geq 0$ such that 
\begin{align}\label{math:boundedvariationinequality}
    \int_{\Omega}
    \langle u, \divergence \vectest \rangle 
    \,\dif{x}
    \leq 
    C \| \vectest \|_{\Lebesgue^{\infty}(\OmegaS)}
    ,
    \quad 
    \vectest \in \Cont^{\infty}_{c}(\OmegaS,X \otimes \VEC^{\ast})
    .
\end{align}
The smallest such constant $C \geq 0$ is called the \emph{total variation} of $u$ over the set $\Omega$
and we write $\TV(u,\OmegaS)$ for this quantity.
\\

Suppose that $\Orlicz$ is a Musielak-Orlicz integrand 
and recall the definition of the Orlicz-Sobolev space $\Sobolev^{m,p}(\OmegaS,\VEC)$.
\emph{We henceforth assume that $\Lebesgue^{\Orlicz}(\OmegaS,\VEC)$ embeds into $\Lebesgue^{1}_{loc}(\OmegaS,\VEC)$.}

We let $\Sobolev^{m,\Orlicz,\TV}(\OmegaS,\VEC)$ be the subspace of $\Sobolev^{m,\Orlicz}(\OmegaS,\VEC)$
whose members have an $m$-th derivative of bounded variation. 
This space is equipped with the norm 
\begin{align}\label{math:sobolevTVnorm}
    \| u \|_{\Sobolev^{m,\Orlicz,\TV}(\OmegaS,\VEC)}
    :=
    \| u \|_{\Sobolev^{m,\Orlicz}(\OmegaS,\VEC)}
    +
    \TV( \nabla^{m} u, \OmegaS)
    .
\end{align}
An important special case is when the derivatives are in Lebesgue spaces. 
We let $\Sobolev^{m,p,\TV}(\OmegaS,\VEC)$ be the subspace of $\Sobolev^{m,p}(\OmegaS,\VEC)$
whose members have an $m$-th derivative of bounded variation. 

This normed space is a Banach space. 
For completeness, we state a proof.

\begin{proposition}\label{proposition:sobolevBVspacecomplete}
    The normed space $\Sobolev^{m,\Orlicz,\TV}(\OmegaS,\VEC)$ is complete. 
\end{proposition}

\begin{proof}
    Suppose that $u_{i}$ is a Cauchy sequence in $\Sobolev^{m,\Orlicz,\TV}(\OmegaS,\VEC)$.
    Then it is a Cauchy sequence in $\Sobolev^{m,\Orlicz}(\OmegaS,\VEC)$
    with a limit $u \in \Sobolev^{m,\Orlicz}(\OmegaS,\VEC)$. 
    We now show that $u \in \Sobolev^{m,\Orlicz,\TV}(\OmegaS,\VEC)$ 
    and that $u$ is the limit of $u_{i}$ in $\Sobolev^{m,\Orlicz,\TV}(\OmegaS,\VEC)$.
    
    We first prove that the function $u$ has bounded variation. 
    Let $\vectest \in \Cont^{\infty}_{c}(\OmegaS, X \otimes \VEC^{\ast})$ be any smooth compactly supported vector field. 
    Since $\Lebesgue^{\Orlicz}(\OmegaS,\VEC)$ embeds into $\Lebesgue^{1}_{loc}(\OmegaS,\VEC)$,
    integration against the divergence of $\vectest$ 
    defines a bounded linear functional on $\Lebesgue^{\Orlicz}(\OmegaS,\VEC)$. 
    Thus, 
    \begin{align*}
        \int_{\Omega} 
        \langle \nabla^{m} u, \divergence \vectest \rangle   
        \,\dif{x}
        &
        =
        \int_{\Omega}
        \left\langle \lim_{i \rightarrow \infty} \nabla^{m} u_{i}, \divergence \vectest \right\rangle 
        \,\dif{x}
        \\&
        \leq 
        \lim_{i \rightarrow \infty}
        \int_{\Omega}
        \langle \nabla^{m} u_{i}, \divergence \vectest \rangle
        \,\dif{x}
\leq 
        \| \vectest \|_{\Lebesgue^{\infty}(\OmegaS)} 
        \cdot 
        \lim_{i \rightarrow \infty}
\TV( \nabla^{m} u_{i}, \OmegaS )
        .
    \end{align*}
    Here, we have used the dominated convergence theorem. 
    Since $u_{i}$ is a Cauchy sequence in $\Sobolev^{m,\Orlicz,\TV}(\OmegaS)$,
    the last limit exists. 
Hence $\nabla^{m} u$ has bounded variation.
    We use a similar technique to show convergence in the total variation seminorm. 
    For any $\epsilon > 0$ there exists $i > 0$ such that for all $j \geq i$
    we have $\TV( \nabla^{m} ( u_{i} - u_{j} ), \OmegaS ) < \epsilon$.
    Consequently, 
    \begin{align*}
        \int_{\Omega} \left\langle \nabla^{m} ( u_{i} - u ), \divergence \vectest \right\rangle \,\dif{x}
        &=
        \int_{\Omega}
        \left\langle \lim_{j \rightarrow \infty} \nabla^{m} ( u_{i} - u_{j} ), \divergence \vectest \right\rangle \,\dif{x}
        \\&
        \leq 
        \lim_{j \rightarrow \infty}
        \int_{\Omega}
        \left\langle \nabla^{m} ( u_{i} - u_{j} ), \divergence \vectest \right\rangle \,\dif{x}
        \\&
        \leq 
        \| \vectest \|_{\Lebesgue^{\infty}(\OmegaS)} 
        \cdot 
        \sup_{j \geq i}
        \TV( \nabla^{m}u_{i} - \nabla^{m}u_{j}, \OmegaS )
        \leq 
        \epsilon 
        \| \vectest \|_{\Lebesgue^{\infty}(\OmegaS)} 
        .
    \end{align*}
    Hence $u_{i}$ converges to $u$ in the total variation seminorm. 
    The proof is complete. 
\end{proof}

\begin{remark}
    The space $\BV(\OmegaS) := \Sobolev^{0,1,\TV}(\OmegaS,\bbR)$ is widely known as the space of multivariate \emph{functions of bounded variation}. 
The literature on bounded variation spaces seems to have focused on this special case, 
    since it suffices for setting up the theory over bounded domains:
    all $p$-integrable functions over bounded domains are integrable. 
    
    The situation is different over unbounded domains. For example, the constant functions over domains with infinite volume are not integrable, but it makes sense to assign them a total variation of zero. Therefore, it is only natural to consider the spaces $\Sobolev^{0,p,\TV}(\OmegaS,\bbR)$ when $p > 1$.
    
    Any scalar function $u \in \Sobolev^{1,1}(\OmegaS)$ is a function of bounded variation,
    and in that case $\TV(u,\OmegaS) \leq \|\nabla u\|_{\Lebesgue^{1}(\OmegaS)}$.
    Of course, there are functions in $\Lebesgue^{p}(\OmegaS)$ that do not have finite total variation. 
\end{remark}

We want to study how functions of bounded variation transform under bi-Lipschitz mappings. 
This will be easier when we characterize the total variation in terms of divergences of non-smooth vector fields. 
We let $\Sobolev^{\divergence,\infty}(\OmegaS, X \otimes \VEC^{\ast})$ be the vector space of those vector fields 
in $\Lebesgue^{\infty}(\OmegaS, X \otimes \VEC^{\ast})$
whose distributional divergence is essentially bounded. 
The vector space of members of $\Sobolev^{\divergence,\infty}(\OmegaS, X \otimes \VEC^{\ast})$
with compact support is written $\Sobolev^{\divergence,\infty}_{c}(\OmegaS, X \otimes \VEC^{\ast})$.

\begin{lemma}\label{math:flatdivergence}
    Suppose that $u \in \Lebesgue^{1}_{loc}(\OmegaS,\VEC)$ has bounded variation over $\Omega$.
    Then 
    \begin{align}\label{math:boundedvariationinequality:flat}
        \int_{\Omega}
        \langle u, \divergence \vectest \rangle
        \,\dif{x}
        \leq 
        \TV(u,\OmegaS) \| \vectest \|_{\Lebesgue^{\infty}(\OmegaS)}
        ,
        \quad 
        \vectest \in \Sobolev^{\divergence,\infty}_{c}(\OmegaS, X \otimes \VEC^{\ast} )
        .
    \end{align}
\end{lemma}

\begin{proof}
Consider a smooth non-negative function $\eta : \bbR^{N} \rightarrow \bbR$ with support in the $N$-dimensional unit ball and having integral one. Define $\eta_\eps(x) = \eps^{-N}\eta(\eps x)$ for $\eps > 0$.
    For $\eps$ small enough, 
    the mollifications $\vectest_{\epsilon} := \eta_\eps \star \vectest$ are smooth and their supports lie within a set compactly contained in $\OmegaS$. 
    Moreover,
    \[
        \divergence( \eta_\eps \star \vectest ) = \eta_\eps \star \divergence\vectest.
    \]
    The sequence
    $\vectest_{\epsilon} \in \Cont^{\infty}_{c}(\OmegaS, X \otimes \VEC^{\ast} )$
    converges to $\vectest$ almost everywhere 
    and 
    $\divergence\vectest_{\epsilon}$
    converges to $\divergence\vectest$ almost everywhere. 
We note that $\| \vectest_{\epsilon} \|_{\Lebesgue^{\infty}(\OmegaS)}$ is bounded by $\| \vectest \|_{\Lebesgue^{\infty}(\OmegaS)}$. 
Since $u$ is locally integrable,  
    \begin{align*}
        \left|
        \int_{\Omega}
        \langle u, \divergence \vectest \rangle 
        \,\dif{x}
        \right|
        &
        =
        \left|
        \int_{\Omega}
        \langle u, \lim_{ \epsilon \rightarrow 0 }
        \divergence \vectest_{\epsilon} \rangle 
        \,\dif{x}
        \right|
        \\&
        =
        \lim_{ \epsilon \rightarrow 0 }
        \left|
        \int_{\Omega}
        \langle u, \divergence \vectest_{\epsilon} \rangle 
        \,\dif{x}
        \right|
\leq 
        \TV(u,\OmegaS) 
        \limsup_{ \epsilon \rightarrow 0 } \| \vectest_{\epsilon} \|_{\Lebesgue^{\infty}(\OmegaS)}
        \leq 
        \TV(u,\OmegaS) \| \vectest \|_{\Lebesgue^{\infty}(\OmegaS)}
        .
    \end{align*}
    This finishes the proof.
\end{proof}

The \emph{Piola transformation} of a vector field $\vectest \in \Lebesgue^{\infty}(\OmegaS,X \otimes \VEC^{\ast})$
along a bi-Lipschitz mapping $\phi : \OmegaS \rightarrow \OmegaT$ is a vector field 
$\phi_{\ast}\vectest \in \Lebesgue^{\infty}(\OmegaT,Y \otimes \VEC^{\ast})$ defined almost everywhere as 
\begin{align}\label{math:pioladefinition}
    \phi_{\ast}\vectest_{|y}
    =
\adj\left(\Jacobian\phi\right)_{|y}^{-1}
    \cdot
    \vectest_{|\phi^{-1}(y)}
    .
\end{align}
Here, $\adj(M)$ is the adjugate matrix of $M \in \bbR^{N \times N}$, also known as the classical adjoint. 

\begin{lemma}\label{lemma:piolacommutes}
    The Piola transform of a vector field in $\vectest \in \Sobolev^{\divergence,\infty}(\OmegaS, X \otimes \VEC^{\ast} )$
    along a bi-Lipschitz transformation $\phi : \OmegaS \rightarrow \OmegaT$
    is a vector field in $\phi_\ast\vectest \in \Sobolev^{\divergence,\infty}(\OmegaT, Y \otimes \VEC^{\ast} )$. 
    Moreover, 
    \begin{align*}
        \divergence( \phi_\ast\vectest )
        =
        \det\Jacobian\phi^{-1}_{|y} \cdot (\divergence \vectest)_{\phi^{-1}(y)} 
    \end{align*}
    almost everywhere. 
\end{lemma}

\begin{proof}
    Let $f \in \Cont^{\infty}_{c}(\OmegaT,\VEC)$ with connected support. 
    Since $\phi$ is bi-Lipschitz, 
    $\sign\det\Jacobian\phi^{-1}$ is essentially constant over the support of $f$~\cite[Corollary 4.1.26]{federer2014geometric}.
For almost all $y \in \OmegaT$ we have 
    \begin{align*}
        \adj\Jacobian\phi^{-1}_{|y}
        =
        \det(\Jacobian\phi^{-1}_{|y}) \cdot \left( \Jacobian\phi_{|y}^{-1} \right)^{-1}
        =
        \det(\Jacobian\phi^{-1}_{|y}) \cdot \Jacobian\phi_{|\phi^{-1}(y)} 
        ,
        \quad 
        \det\Jacobian\phi^{-1}_{|y}
        =
        \det\Jacobian\phi^{-1}_{|\phi^{-1}(y)}
        .
    \end{align*}
    On the one hand, 
    \begin{align*}
        \int_{\OmegaS} \left\langle \nabla ( f \circ \phi )_{|x}, \vectest_{|x} \right\rangle 
        \,\dif{x}
        &
        =
        \int_{\OmegaS} \left\langle \nabla f_{|\phi(x)} \cdot \Jacobian\phi_{|x}, \vectest_{|x} \right\rangle 
        \,\dif{x}
        \\&=
        \int_{\OmegaT} \left\langle \nabla f_{|y} \cdot \Jacobian\phi_{|\phi^{-1}(y)}, \vectest_{|\phi^{-1}(y)} \right\rangle 
        |\det\Jacobian\phi^{-1}_{|y}| \,\dif{y}
        \\&=
        \int_{\OmegaT}
        \sign\left(\det\Jacobian\phi^{-1}_{|y}\right)
        \left\langle \nabla f_{|y}, \Jacobian\phi_{|\phi^{-1}(y)} \det(\Jacobian\phi^{-1}_{|y}) \cdot \vectest_{|\phi^{-1}(y)} \right\rangle 
        \,\dif{y}
        \\&=
        \int_{\OmegaT}
        \sign\left(\det\Jacobian\phi^{-1}_{|y}\right)
        \left\langle \nabla f_{|y}, \adj\Jacobian\phi^{-1}_{|y} \cdot \vectest_{|\phi^{-1}(y)} \right\rangle
        \,\dif{y}
        .
    \end{align*}
    On the other hand, 
    \begin{align*}
        \int_{\OmegaS} \left\langle \nabla ( f \circ \phi )_{|x} , \vectest_{|x} \right\rangle
        \,\dif{x}
        &
        =
        -
        \int_{\OmegaS} \left\langle ( f \circ \phi )_{|x} , \divergence \vectest_{|x} \right\rangle
        \,\dif{x}
=
        -
        \int_{\OmegaT} \left\langle f_{|y} , (\divergence \vectest)_{|\phi^{-1}(y)} \right\rangle \left| \det\Jacobian\phi^{-1}_{|\phi^{-1}(y)} \right| 
        \,\dif{y}
.
    \end{align*}
    Therefore, $\det(\Jacobian\phi^{-1}_{|y}) \cdot (\divergence \vectest)_{|\phi^{-1}(y)}$ is the distributional divergence of $\adj\Jacobian\phi^{-1}_{|y} \cdot \vectest_{|\phi^{-1}(y)}$. 
    That completes the proof. 
\end{proof}

\begin{remark}
    Different authors define the Piola transformation in different, sometimes non-equivalent forms. 
    Often, the adjugate matrix is written as the Jacobian multiplied by the determinant,
    and the direction of the transformation may be either in the direction of $\phi$ or in the opposite direction, 
    depending on the convention. 
    Our convention emphasizes that the Piola transformation~\eqref{math:pioladefinition} of vector fields is dual to the pullback of gradients.
\end{remark}

\begin{lemma}\label{lemma:invarianceBV:scalar}
    Suppose that $\phi : \OmegaS \rightarrow \OmegaT$ is bi-Lipschitz.
    If $u \in \Lebesgue^{1}_{loc}(\OmegaT,\VEC)$ has bounded variation over $\OmegaT$, 
    then $u \circ \phi \in \Lebesgue^{1}_{loc}(\OmegaS,\VEC)$ has bounded variation over $\OmegaS$, and 
    \begin{align}\label{math:invarianceBV:scalar}
        \TV( u \circ \phi, \OmegaS )
        \leq
        \| \adj\Jacobian\phi \|_{\Lebesgue^{\infty}(\OmegaS)}
        \cdot
        \TV(u,\OmegaS)
        .
    \end{align}
\end{lemma}

\begin{proof}
    Suppose that $\vectest \in \Sobolev^{\divergence,\infty}(\OmegaS, X \otimes \VEC^{\ast} )$.
    Let $\hat\vectest \in \Sobolev^{\divergence,\infty}(\OmegaT, Y \otimes \VEC^{\ast} )$
    such that $\phi_{\ast}^{-1}\hat\vectest = \vectest$.
    The properties of the Piola transform imply 
    \begin{align*}
        \int_{\OmegaS}
        \left\langle u_{|\phi(x)}, \divergence( \vectest )_{|x} \right\rangle 
        \,\dif{x}
        &
        =
        \int_{\OmegaS}
        \left\langle u_{|\phi(x)}, \divergence( \phi_{\ast}^{-1}\hat\vectest )_{|x} \right\rangle 
        \,\dif{x}
=
        \int_{\OmegaS}
        \left\langle u_{|\phi(x)}, \det(\Jacobian\phi) \divergence( \hat\vectest )_{|\phi(x)} \right\rangle 
        \,\dif{x}
        .
    \end{align*}
    We thus estimate 
    \begin{align*}
        \left|
        \int_{\OmegaS}
        \left\langle u_{|\phi(x)}, \divergence( \vectest )_{|x} \right\rangle 
        \,\dif{x}
        \right|
        &=
        \left|
        \int_{\OmegaS}
        \left\langle u_{|\phi(x)}, \det(\Jacobian\phi) \divergence( \hat\vectest )_{|\phi(x)} \right\rangle 
        \,\dif{x}
        \right|
        \\&
        =
        \left|
        \int_{\OmegaT}
        \left\langle u_{|y}, \divergence( \hat\vectest )_{|y} \right\rangle 
        \,\dif{y}
        \right|
\leq 
        \TV(u,\OmegaT)
        \cdot 
        \| \hat\vectest \|_{\Lebesgue^{\infty}(\OmegaT)}
        .
    \end{align*}
    Finally,    
    \begin{align*}
        \| \hat\vectest \|_{\Lebesgue^{\infty}(\OmegaT)}
        =
        \| \phi_{\ast}\vectest \|_{\Lebesgue^{\infty}(\OmegaT)}
        \leq 
        \| \adj\Jacobian\phi \|_{\Lebesgue^{\infty}(\OmegaS)}
        \cdot 
        \| \vectest \|_{\Lebesgue^{\infty}(\OmegaS)}
        .
    \end{align*}
    This shows~\eqref{math:invarianceBV:scalar}. The proof is complete. 
\end{proof}

\begin{lemma}\label{lemma:invarianceBV:tensor}
    Suppose that $\phi : \OmegaS \rightarrow \OmegaT$ is bi-Lipschitz and $\nabla\phi \in \Sobolev^{\divergence,\infty}(\OmegaS)$.
    Let $\Orlicz$ be a proper Musielak-Orlicz integrand over $\OmegaS$. 
    If $u \in \Lebesgue^{\Orlicz}(\OmegaT,\LIN^{d}(Y,\VEC))$ has bounded variation over $\OmegaT$, 
    then 
    \begin{align}\label{math:invarianceBV:tensor}
        \begin{aligned}
            \TV( \phi^{\ast} u, \OmegaS )
            &\leq
            \left\| \otimes^{d}\nabla\phi_{} \right\|_{\Lebesgue^{\infty}(\OmegaS)}
            \| \adj\Jacobian\phi \|_{\Lebesgue^{\infty}(\OmegaS)}
            \TV(u,\OmegaS)
            \\
            &\quad\quad
            +
            C_{\Orlicz,\Dual}
            \left\| \divergence \left( \otimes^{d}\nabla\phi_{} \right) \right\|_{\Lebesgue^{\Dual\phi^{\ast}\Orlicz}(\OmegaS)}
            \left\| {u}_{|\phi} \right\|_{\Lebesgue^{\phi^{\ast}\Orlicz}(\OmegaS)}
            .
        \end{aligned}
    \end{align}
    In particular, if the right-hand side is bounded, 
    then $\phi^{\ast} u \in \Lebesgue^{p}(\OmegaS,\LIN^{d}(X,\VEC))$ has bounded variation over $\OmegaS$.
\end{lemma}

\begin{proof}
    Let $u \in \Lebesgue^{1}_{loc}(\OmegaT,\LIN^{d}(Y,\VEC))$ 
    and let $\vectest \in \Sobolev^{\divergence,\infty}(\OmegaS, \otimes^{d+1} X \otimes \VEC^{\ast} )$.
    We use the definition of the pullback and observe 
    \begin{align*}
        &
        \int_{\OmegaS}
        \left( \phi^{\ast} {u} \right)_{|x}
        \contract
        \divergence\vectest_{|x}
        \,\dif{x}
        =
        \int_{\OmegaS}
        {u}_{|\phi(x)}
        \contract
        \otimes^{d} \nabla\phi_{|x}
        \contract
        \divergence 
        \vectest_{|x}
        \,\dif{x}
        \\
        &\quad 
        =
        \int_{\OmegaS}
        {u}_{|\phi(x)}
        \contract
        \divergence 
        \left( 
            \otimes^{d} \nabla\phi_{|x}
            \contract
            \vectest_{|x}
        \right)
        \,\dif{x}
        -
        \int_{\OmegaS}
        {u}_{|\phi(x)}
        \contract
        \divergence 
        \left( \otimes^{d} \nabla\phi_{|x} \right)
        \contract
        \vectest_{|x}
        \,\dif{x}
        .
    \end{align*}  
    We bound the last two terms. 
    First, $\otimes^{d} \nabla\phi_{|x} \contract \vectest_{|x}$ is essentially bounded with essentially bounded divergence, 
    and we apply Lemma~\ref{lemma:invarianceBV:scalar}.
    Second, we use H\"older's inequality together with the assumptions and derive 
    \begin{align*}
        &
        \left|
        \int_{\OmegaS}
            {u}_{|\phi(x)}
            \contract
            \divergence 
            \left( \otimes^{d} \nabla\phi_{|x} \right)
            \contract  
            \vectest_{|x}
        \,\dif{x}
        \right|
        \leq 
        C_{\Orlicz,\Dual}
        \left\| {u}_{|\phi} \right\|_{\Lebesgue^{\phi^{\ast}\Orlicz}(\OmegaS)}
        \left\| \divergence \left( \otimes^{d}\nabla\phi_{} \right) \right\|_{\Lebesgue^{\Dual\phi^{\ast}\Orlicz}(\OmegaS)}
        \left\| \vectest \right\|_{\Lebesgue^{\infty}(\OmegaS)}
        .
    \end{align*}
The desired statement thus follows. 
\end{proof}

We see that scalar functions of bounded variation are preserved under bi-Lipschitz transformations.
However, tensor fields of bounded variation are preserved under bi-Lipschitz transformations
whose Jacobians additionally have divergence within a Lebesgue space. 
That regularity requirement is higher than for preserving tensor fields in Lebesgue spaces~(see Corollary~\ref{corollary:weakderivativeofpullback:sobolev}). 
We now generalize this crucial observation to functions with higher weak derivatives. 

\begin{proposition}\label{prop:higherboundedvariation}
Let $\OmegaS, \OmegaT \subseteq \bbR^{N}$ be open sets, and let $\VEC$ be a Banach space.
    Let $m \in \bbN$ 
    and let $\Orlicz$ be a proper Musielak-Orlicz integrand over $\OmegaS$. 
    Suppose that $\phi : \OmegaS \rightarrow \OmegaT$ is locally bi-Lipschitz. 
Suppose that the weak derivatives $\nabla\phi, \dots, \nabla^{m+1}\phi$ are in $\Sobolev^{\divergence,\infty}(\OmegaS)$.
If $u \in \Sobolev^{p,m}(\OmegaT,\LIN^d(Y,\VEC))$
    and $\nabla^{m} u$ has bounded variation,
    then 
    \begin{align*}
        &
        \TV( \nabla^{m} (\phi^{\ast} u) , \OmegaS )
\leq 
        \TV( \nabla^{m} u_{|\phi}, \OmegaS )
        \left\| \nabla\phi \right\|_{\Lebesgue^{\infty}(\OmegaS)}^{m+d}
        \\&\quad\quad
        +
C_{\Orlicz,\Dual}
        \sum_{ \substack{ 0 \leq k \leq m-1 \\ {P} \in \calP_{0}(m,d) \\ k \leq |{P}_{0}| \\ \calC \in \calP({P}_{0},k) } }
        \left\|
            \nabla^{k+1} {u}_{|\phi}
        \right\|
        _{\Lebesgue^{\phi^{\ast}\Orlicz}(\OmegaS)}
        \left\|
            \nabla\phi_{}
            \otimes 
            \nabla_{ C_1 }\phi_{}
            \otimes 
            \cdots 
            \otimes 
            \nabla_{ C_k }\phi_{}
            \otimes 
            \nabla_{ P_1 }\nabla\phi_{}
            \otimes 
            \cdots 
            \otimes 
            \nabla_{ P_d }\nabla\phi_{}
        \right\|
        _{\Lebesgue^{\Dual\phi^{\ast}\Orlicz}(\OmegaS)}
        \\&\quad\quad
        +
        C_{\Orlicz,\Dual}
        \sum_{ \substack{ 0 \leq k \leq m \\ {P} \in \calP_{0}(m,d) \\ k \leq |{P}_{0}| \\ \calC \in \calP({P}_{0},k) } }
        \left\|
            \nabla^{k} {u}_{|\phi}
        \right\|
        _{\Lebesgue^{\phi^{\ast}\Orlicz}(\OmegaS)}
        \left\|
            \divergence \left(
            \nabla_{ C_1 }\phi_{}
            \otimes 
            \cdots 
            \otimes 
            \nabla_{ C_k }\phi_{}
            \otimes 
            \nabla_{ P_1 }\nabla\phi_{}
            \otimes 
            \cdots 
            \otimes 
            \nabla_{ P_d }\nabla\phi_{}
            \right)
        \right\|
        _{\Lebesgue^{\Dual\phi^{\ast}\Orlicz}(\OmegaS)}
        .
    \end{align*}
    In particular, 
    if the right-hand side of that inequality is finite, 
    then $\phi^{\ast} u \in \Sobolev^{p,m}(\OmegaS,\LIN^d(X,\VEC))$ and $\nabla^{m} (\phi^{\ast} u)$ has bounded variation.

    If $d=0$, we only need to assume that the derivatives $\nabla\phi, \dots, \nabla^{m}\phi$ exist almost everywhere and are in $\Sobolev^{\divergence,\infty}(\OmegaS)$.
\end{proposition}

\begin{proof}
    We apply Proposition~\ref{prop:weakderivativesofpullback}. 
    Using the representation of $\nabla^m u$ as stated in that result,
    multiplying with the directional divergence of an arbitrary test vector field $\vectest \in \Cont_{c}^{\infty}(\OmegaS, X \otimes \VEC^{\ast} )$,
    and integrating,
    we derive 
    \begin{align*}
        &
        \int_{\OmegaS}
        \nabla_{v_1,v_2,\ldots,v_m} 
        \left( \phi^{\ast} {u} \right)_{|x}
        \contract
        \divergence\vectest_{|x}
        \,\dif{x}
        \\&\quad 
        =
        \sum_{ \substack{ 0 \leq k \leq m \\ {P} \in \calP_{0}(m,d) \\ k \leq |{P}_{0}| \\ \calC \in \calP({P}_{0},k) } }
        \int_{\OmegaS}
        \nabla^{k} {u}_{|\phi(x)}
        \contract
        \left( 
            \nabla_{ C_1 }\phi_{|x}
            \otimes 
            \cdots 
            \otimes 
            \nabla_{ C_k }\phi_{|x}
            \otimes 
            \nabla_{ P_1 }\nabla\phi_{|x}
            \otimes 
            \cdots 
            \otimes 
            \nabla_{ P_d }\nabla\phi_{|x}
        \right)
        \contract
        \divergence 
        \vectest_{|x}
        \,\dif{x}
    \end{align*}
    We take a look at each of those summands. 
    Let $0 \leq k \leq m$, ${P} \in \calP_{0}(m,d)$, and $\calC \in \calP({P}_{0},k)$
    such that $k \leq |{P}_{0}|$.
    We see that 
    \begin{align*}
        &
        \int_{\OmegaS}
        \nabla^{k} {u}_{|\phi(x)}
        \contract
        \left( 
            \nabla_{ C_1 }\phi_{|x}
            \otimes 
            \cdots 
            \otimes 
            \nabla_{ C_k }\phi_{|x}
            \otimes 
            \nabla_{ P_1 }\nabla\phi_{|x}
            \otimes 
            \cdots 
            \otimes 
            \nabla_{ P_d }\nabla\phi_{|x}
        \right)
        \contract 
        \divergence 
        \vectest_{|x}
        \,\dif{x}
        \\&=
        \int_{\OmegaS}
        \nabla^{k} {u}_{|\phi(x)}
        \contract
        \divergence 
        \left( 
            \nabla_{ C_1 }\phi_{|x}
            \otimes 
            \cdots 
            \otimes 
            \nabla_{ C_k }\phi_{|x}
            \otimes 
            \nabla_{ P_1 }\nabla\phi_{|x}
            \otimes 
            \cdots 
            \otimes 
            \nabla_{ P_d }\nabla\phi_{|x}
            \contract 
            \vectest_{|x}
        \right)
        \,\dif{x}
        \\&\quad 
        -
        \int_{\OmegaS}
        \nabla^{k} {u}_{|\phi(x)}
        \contract
        \divergence 
        \left( 
            \nabla_{ C_1 }\phi_{|x}
            \otimes 
            \cdots 
            \otimes 
            \nabla_{ C_k }\phi_{|x}
            \otimes 
            \nabla_{ P_1 }\nabla\phi_{|x}
            \otimes 
            \cdots 
            \otimes 
            \nabla_{ P_d }\nabla\phi_{|x}
        \right)
        \contract  
        \vectest_{|x}
        \,\dif{x}
        .
    \end{align*}
    We use H\"older's inequality together with the assumptions and derive 
    \begin{align*}
        &
        \left|
        \int_{\OmegaS}
            \nabla^{k} {u}_{|\phi(x)}
            \contract
            \divergence 
            \left( 
                \nabla_{ C_1 }\phi_{|x}
                \otimes 
                \cdots 
                \otimes 
                \nabla_{ C_k }\phi_{|x}
                \otimes 
                \nabla_{ P_1 }\nabla\phi_{|x}
                \otimes 
                \cdots 
                \otimes 
                \nabla_{ P_d }\nabla\phi_{|x}
            \right)
            \contract  
            \vectest_{|x}
        \,\dif{x}
        \right|
        \\&\quad 
        \leq 
        C_{\Orlicz,\Dual}
        \left\|
            \nabla^{k} {u}_{|\phi}
        \right\|
        _{\Lebesgue^{\phi^{\ast}\Orlicz}(\OmegaS)}
        \left\|
            \divergence \left(
            \nabla_{ C_1 }\phi_{}
            \otimes 
            \cdots 
            \otimes 
            \nabla_{ C_k }\phi_{}
            \otimes 
            \nabla_{ P_1 }\nabla\phi_{}
            \otimes 
            \cdots 
            \otimes 
            \nabla_{ P_d }\nabla\phi_{}
            \right)
        \right\|
        _{\Lebesgue^{\Dual\phi^{\ast}\Orlicz}(\OmegaS)}
        \left\|
            \vectest
        \right\|
        _{\Lebesgue^{\infty}(\OmegaS)}
        .
    \end{align*}
    If now $k < m$, we integrate by parts and use H\"older's inequality once more,
    \begin{align*}
        &
        \left|
        \int_{\OmegaS}
        \nabla^{k} {u}_{|\phi(x)}
        \contract
        \divergence 
        \left( 
            \nabla_{ C_1 }\phi_{|x}
            \otimes 
            \cdots 
            \otimes 
            \nabla_{ C_k }\phi_{|x}
            \otimes 
            \nabla_{ P_1 }\nabla\phi_{|x}
            \otimes 
            \cdots 
            \otimes 
            \nabla_{ P_d }\nabla\phi_{|x}
            \contract  
            \vectest_{|x}
        \right)
        \,\dif{x}
        \right|
        \\&\quad 
        =
        \left|
        \int_{\OmegaS}
        \nabla^{k+1} {u}_{|\phi(x)}
        \contract
        \left( 
            \nabla\phi_{|x}
            \otimes 
            \nabla_{ C_1 }\phi_{|x}
            \otimes 
            \cdots 
            \otimes 
            \nabla_{ C_k }\phi_{|x}
            \otimes 
            \nabla_{ P_1 }\nabla\phi_{|x}
            \otimes 
            \cdots 
            \otimes 
            \nabla_{ P_d }\nabla\phi_{|x}
            \contract  
            \vectest_{|x}
        \right)
        \,\dif{x}
        \right|
        \\&\quad 
        \leq 
        C_{\Orlicz,\Dual}
        \left\|
            \nabla^{k+1} {u}_{|\phi}
        \right\|
        _{\Lebesgue^{\phi^{\ast}\Orlicz}(\OmegaS)}
        \left\|
            \nabla\phi_{}
            \otimes 
            \nabla_{ C_1 }\phi_{}
            \otimes 
            \cdots 
            \otimes 
            \nabla_{ C_k }\phi_{}
            \otimes 
            \nabla_{ P_1 }\nabla\phi_{}
            \otimes 
            \cdots 
            \otimes 
            \nabla_{ P_d }\nabla\phi_{}
        \right\|
        _{\Lebesgue^{\Dual\phi^{\ast}\Orlicz}(\OmegaS)}
        \left\|
            \vectest
        \right\|
        _{\Lebesgue^{\infty}(\OmegaS)}
        .
    \end{align*}
If instead $k=m$, then 
    \begin{align*}
        &
        \left|
        \int_{\OmegaS}
        \nabla^{m} {u}_{|\phi(x)}
        \contract
        \divergence 
        \left( 
            \nabla_{ C_1 }\phi_{|x}
            \otimes 
            \cdots 
            \otimes 
            \nabla_{ C_m }\phi_{|x}
            \otimes 
            \nabla_{ P_1 }\nabla\phi_{|x}
            \otimes 
            \cdots 
            \otimes 
            \nabla_{ P_d }\nabla\phi_{|x}
            \contract  
            \vectest_{|x}
        \right)
        \,\dif{x}
        \right|
        \\&\quad 
        =
        \left|
        \int_{\OmegaS}
        \nabla^{m} {u}_{|\phi(x)}
        \contract
        \divergence 
        \left( 
            \otimes^{m+d} \nabla\phi_{|x}
            \contract  
            \vectest_{|x}
        \right)
        \,\dif{x}
        \right|
\leq 
        \TV( \nabla^{m} {u}_{|\phi}, \OmegaS )
        \cdot 
        \left\|
            \vectest
        \right\|_{\Lebesgue^{\infty}(\OmegaS)}
        .
    \end{align*}
    Lemma~\ref{lemma:invarianceBV:scalar} implies that $\nabla^{m} {u}_{|\phi}$ has bounded variation.
    The proof is complete. 
\end{proof}

\begin{remark}
    A more comprehensive picture emerges. 
    We already know that the pullback of a scalar function with $m$ weak derivatives will have $m$ weak derivatives again provided that the transformation has essentially bounded derivatives of orders $1$ to $m$. 
    The pullback of a tensor field with $m$ weak derivatives having $m$ weak derivatives again additionally requires a transformation with essentially bounded derivative of order $m+1$. 
    Preserving the bounded variation class imposes further requirements on the transformation:
    the derivative of order $m$ (for scalar fields) or of order $m+1$ (for tensor fields) must additionally have an essentially bounded divergence. 
    
    It is instructive to review special cases of Lemma~\ref{lemma:invarianceBV:tensor} and Lemma~\ref{prop:higherboundedvariation} when the Orlicz norm is a classical Lebesgue norm. 
    In the case $p=1$, which has received the bulk of attention in the literature, 
    we only need that the divergence of the transformation's highest derivative is essentially bounded. 
    The same is also true for $p > 1$ when the domain is bounded. 
    However, the situation is different when $p>1$ and the domain has infinite volume. 
    Then we must choose $q = p/(p-1)$, 
    which effectively enforces a decay condition on the higher derivatives, in addition to being essentially bounded. 
    For example, the class of essentially bounded functions with bounded variation 
    is preserved under any coordinate transformation whose Jacobian has integrable and essentially bounded divergence. 
\end{remark}

We take a look at univariate functions of bounded variation as an interesting special case.

\begin{corollary}\label{prop:higherboundedvariation:univariate}
    Let $\IS, \IT \subseteq \bbR$ be open intervals.
    Let $m \in \bbN$ 
    and let $\Orlicz$ be a proper Musielak-Orlicz integrand over $\OmegaS$. 
    Suppose that $\phi : \IS \rightarrow \IT$ is locally bi-Lipschitz. 
Suppose that the weak derivatives $\nabla\phi, \dots, \nabla^{m+1}\phi$ are in $\Lebesgue^{\infty}(\IS)$.
If $u \in \Sobolev^{m,\Orlicz}(\IT,\bbR)$
    and $\nabla^{m} u$ has bounded variation,
    then 
    \begin{align*}
\TV( \nabla^{m} ( u \circ \phi ) , \IS )
        &\leq 
        \TV( \nabla^{m} u \circ \phi, \IS )
        \left\| \nabla\phi \right\|_{\Lebesgue^{\infty}(\IS)}^{m+d}
        \\&\quad\quad
        +
        C_{\Orlicz,\Dual}
        \sum_{ \substack{ 0 \leq k \leq m-1 \\ \calC \in \calP(m,k) } }
        \left\|
            \nabla^{k+1} {u} \circ \phi
        \right\|
        _{\Lebesgue^{\phi^{\ast}\Orlicz}(\IS)}
        \left\|
            \nabla\left(
            \nabla\phi_{}
            \otimes 
            \nabla_{ C_1 }\phi_{}
            \otimes 
            \cdots 
            \otimes 
            \nabla_{ C_k }\phi_{}
            \right)
        \right\|
        _{\Lebesgue^{\Dual\phi^{\ast}\Orlicz}(\IS)}
        \\&\quad\quad
        +
        C_{\Orlicz,\Dual}
        \sum_{ \substack{ 1 \leq k \leq m \\ \calC \in \calP(m,k) } }
        \left\|
            \nabla^{k} {u} \circ \phi
        \right\|
        _{\Lebesgue^{\phi^{\ast}\Orlicz}(\IS)}
        \left\|
            \nabla \left(
            \nabla_{ C_1 }\phi_{}
            \otimes 
            \cdots 
            \otimes 
            \nabla_{ C_k }\phi_{}
            \right)
        \right\|
        _{\Lebesgue^{\Dual\phi^{\ast}\Orlicz}(\IS)}
        .
    \end{align*}
    In particular, 
    if the right-hand side of that inequality is finite, 
    then $u \circ \phi \in \Sobolev^{m,\Orlicz}(\IS,\bbR)$ and $\nabla^{m} ( u \circ \phi )$ has bounded variation.
\end{corollary}

\subsection{An application with harmonic maps}

The estimate in Proposition~\ref{prop:higherboundedvariation} uses a regularity condition that requires the highest-order Jacobian of the transformation $\phi$ to have a weak divergence. We close this presentation with sketching in how far this requirement might naturally arise in the study of diffeomorphisms that minimize energy functionals.
Such diffeomorphisms appear in the theory of linearized elasticity and the theory of harmonic maps.

Suppose that $\OmegaS \subseteq \bbR^{N}$ is a bounded Lipschitz domain and consider the Dirichlet energy functional 
\begin{align*}
    E : \Sobolev^{1,2}(\OmegaS)^{N} \rightarrow \bbR,
    \quad 
    \bfD
    \mapsto 
    \int_{\OmegaS} \langle \nabla\bfD, \nabla\bfD \rangle - \langle \bfF, \bfD \rangle \;\dif{x},
\end{align*}
where $\bfF \in \Lebesgue^{2}(\OmegaS)^{N}$.
We now suppose that $\bfD \in \Sobolev^{1,2}(\OmegaS)^{N}$ is the unique minimizer to this energy functional subject to, say, homogeneous Dirichlet boundary conditions,
\begin{align*}
    \bfD_{|\partial\OmegaS} = 0.
\end{align*}
In a physical interpretation, $\bfF$ represents the external force applied to a body that is fixed at the boundary, and $\bfD$ is resulting displacement vector field. $\bfD$ is the weak solution to 
\begin{align*}
    \divergence (\nabla \bfD) = \bfF
\end{align*}
with homogeneous Dirichlet boundary conditions. Of course, this simple model problem is just a system of independent Poisson problems. 

We introduce the mapping 
\begin{align*}
    \phi : \OmegaS \rightarrow \bbR^{N}, \quad x \mapsto x + \bfD(x).
\end{align*}
If $\bfD$ and $\nabla\bfD$ are essentially bounded over $\OmegaS$ and $\nabla\bfD$ is small enough, 
then $\phi$ is a bi-Lipschitz mapping $\phi : \OmegaS \rightarrow \OmegaS$.
It is clear that the pullback along $\phi$ preserves the function class $\Sobolev^{1,p}$.
However, it is not clear whether $\phi$ preserves any function classes of higher regularity,
$\nabla\nabla\phi$ may not exist as an essentially bounded field. 

We now make the additional assumption that the external force field $\bfF$ is essentially bounded. 
Then $\nabla\phi \in \Sobolev^{\divergence,\infty}(\OmegaS)$. 
This extra regularity allows us to apply Proposition~\ref{prop:higherboundedvariation}, which shows that $u \circ \phi \in \Sobolev^{1,p,\TV}(\OmegaS)$
for any $u \in \Sobolev^{1,p,\TV}(\OmegaT)$ and $p \in [1,\infty]$. 
\\

This simple example demonstrates that the regularity condition $\nabla\phi \in \Sobolev^{\divergence,\infty}(\OmegaS)$ arises naturally in applications such as linearized elasticity in the low regularity regime.

\subsection*{Acknowledgment}

This material is based upon work supported by the National Science Foundation under Grant No.\ DMS-1439786 while the author was in residence at the Institute for Computational and Experimental Research in Mathematics in Providence, RI, during the ``Advances in Computational Relativity'' program.

\end{document}